\documentclass[12pt,reqno]{amsart}
\usepackage{amsmath}
\usepackage{amssymb}
\usepackage{amstext}
\usepackage{a4wide}
\usepackage{graphicx}
\usepackage{bm}
\usepackage[numbers,sort&compress]{natbib}
\allowdisplaybreaks
\numberwithin{equation}{section}

\usepackage{cases}

\newtheorem{theorem}{Theorem}[section]
\newtheorem{proposition}[theorem]{Proposition}
\newtheorem{corollary}[theorem]{Corollary}
\newtheorem{lemma}[theorem]{Lemma}
\newtheorem*{Yudovich's Theorem}{Yudovich's Theorem}

\theoremstyle{definition}

\theoremstyle{remark}
\newtheorem{remark}[theorem]{Remark}

\begin{document}

\title
[Stability of degree-2 Rossby-Haurwitz waves ]{Stability of degree-2 Rossby-Haurwitz waves}

 \author{Daomin Cao, Guodong Wang, Bijun Zuo}
 \address{State Key Laboratory of Mathematical Sciences, Academy of Mathematics and Systems
Science, Chinese Academy of Sciences, Beijing 100190, P.R. China, and University of Chinese Academy of Sciences, Beijing 100049,  P.R. China}
\email{dmcao@amt.ac.cn}
\address{School of Mathematical Sciences, Dalian University of Technology, Dalian 116024, P.R. China}
\email{gdw@dlut.edu.cn}
\address{College of Mathematical Sciences, Harbin Engineering University, Harbin 150001, P.R. China}
\email{bjzuo@amss.ac.cn}

%\thanks{}

\begin{abstract}
Rossby-Haurwitz (RH) waves are important explicit solutions of the incompressible Euler equation on a two-dimensional rotating sphere. In this paper, we prove the orbital stability of degree-2 RH waves, which confirms a conjecture proposed by A. Constantin and P. Germain in [Arch. Ration. Mech. Anal. {\bf245}, 587--644, 2022].
The proofs are based on a variational approach,  with the main challenge being to establish  suitable variational characterizations for  the solutions under consideration.   In this process, the set of rearrangements of a fixed function plays a vital role.
We also apply our approach  to the stability analysis of degree-1 RH waves,  Arnold-type flows, and  zonal flows with monotone absolute vorticity.
 \end{abstract}

\maketitle
\tableofcontents
\section{Introduction}\label{sec1}
 In this article, we investigate the motion of an incompressible inviscid fluid of unit density on the two-dimensional sphere $\mathbb S^{2}$ equipped with the standard metric,
\begin{equation}\label{hb1}
\mathbb S^{2}=\left\{\mathbf x=(x_1,x_2,x_3)\in\mathbb R^{3}\mid x_1^{2}+x_2^{2}+x_3^{2}=1\right\}.
\end{equation}
The sphere is allowed to rotate with constant angular speed $\omega\in\mathbb R$ about the polar axis, i.e., the $x_3$-axis. The governing equation of motion is the incompressible Euler equation
\begin{equation}\label{ave0}
\partial_t\zeta+\mathbf v\cdot\nabla\zeta=0,
\qquad \mathbf x\in\mathbb S^2,\quad t>0,
\end{equation}
where $\zeta$ is the absolute vorticity, and $\mathbf v$ is the velocity field. The absolute vorticity $\zeta$, which combines the relative vorticity of the flow with the planetary vorticity induced by the sphere's rotation, measures the total local rotation of a fluid particle relative to an inertial frame; see \eqref{abvort} in Section~\ref{sec22} for the precise definition. The velocity field $\mathbf v$ can be recovered from $\zeta$ through the Biot--Savart law; see \eqref{ave} in Section~\ref{sec22}.

As in a two-dimensional Euclidean domain or on a flat 2-torus (cf. \cite{MB,MP,WZ0,WS}), for a fluid moving on a rotating sphere, the  kinetic energy and the distribution function of the absolute vorticity remain invariant.   In particular,  solutions of the Euler equation satisfy an infinite
number of integral conservation laws.  Furthermore, due to the rich symmetry structure of $\mathbb S^{2}$, there is an additional conservation law involving the projection of the absolute vorticity onto the first eigenspace of the Laplace-Beltrami operator (cf. \eqref{cl3} in Section~\ref{sec24}). As we will see later, this feature brings many new distinctive aspects to the study of the Euler equation  on $\mathbb S^{2}$.

The Rossby-Haurwitz (RH) waves are a class of steady or traveling-wave solutions to the Euler equation on a rotating sphere with explicit expressions, which have a very wide range of  applications in meteorology.  See \cite{Crai,Haur,Ross,Verk}. In the longitude-latitude spherical coordinates,
\begin{equation}\label{hb2}
(\varphi,\theta)\mapsto (\cos\theta\cos\varphi,\cos\theta\sin\varphi,\sin\theta)\in\mathbb S^2,\quad (\varphi,\theta)\in(-\pi,\pi)\times(-\frac{\pi}{2},\frac{\pi}{2}),
\end{equation}
RH waves of degree $j$  in terms of the absolute vorticity $\zeta$ take the following form:
\begin{equation}\label{av00}
\zeta(\varphi,\theta,t)=\alpha\sin\theta+Y(\varphi-ct,\theta),\quad Y\in\mathbb E_{j},
\end{equation}
where $\alpha$ is a real number, $\mathbb E_{j}$ is the $j$-th eigenspace of the Laplace-Beltrami operator (cf. Section~\ref{sec21}), and $c$ is the traveling speed
satisfying
\begin{equation}\label{tspeed}
c= \alpha\left(\frac{1}{2}- \frac{1}{j(j+1)}\right)- \omega.
\end{equation}
Note that $\zeta$ is steady (i.e., $\zeta$ does not depend on the time variable $t$) if and only if $Y$ is zonal (i.e., $Y$ depends only on $\theta$), or
\[\omega=\alpha\left(\frac{1}{2}- \frac{1}{j(j+1)}\right).\]
 Many   phenomena in the terrestrial atmosphere, such as  the  flows with a banded structure in polar regions and the  5-day wave (cf. \cite{CG,HYHH}), are closely related to   RH waves.  Besides,  RH waves   have also been observed in the atmospheres of the outer  planets of our solar system such as Jupiter and Saturn. See the discussion in \cite{CG,Dow}.

In this paper, we focus on the stability issue of
 RH waves, which is crucial for understanding the complicated dynamics of the atmosphere.
In the literature, there has  already been considerable research on this topic. See \cite{Baines,Benard,CG,Hos,Hos2,Lorenz,Skiba0,Skiba1,Skiba2} and the references therein.
However, most of these results deal with linear stability\footnote{As mentioned in \cite{CG},  the linearization  approach helps to gain valuable insights into nonlinear stability, but the question of whether it can produce definitive conclusions remains a topic of ongoing debate and investigation.}, or are obtained by numerical means.
The question of whether RH waves are nonlinearly stable  in a rigorous sense is less well understood. We summarize some related results in the literature as follows:
\begin{itemize}
\item [(1)]If $\omega=0,$ then every degree-1 RH wave is nonlinearly stable in the $L^{2}$ norm of the absolute vorticity. For general $\omega\in\mathbb R,$
any degree-1 RH wave is, modulo the group of rotations about the polar axis, nonlinearly stable in the $L^{2}$ norm of the absolute vorticity.
This can be easily verified by applying the conservation laws of the Euler equation. See \cite{CG} or the detailed discussion in Section~\ref{sec24}.
\item [(2)] If $\alpha\neq 0$, then every zonal RH wave of degree $2$
is nonlinearly stable in the $L^2$ norm of the absolute
vorticity under perturbations with uniformly bounded absolute
vorticity. This is the nondegenerate case in the proof of
\cite[Theorem 7(i)]{CG}.
\item[(3)]  Non-zonal  RH waves of any degree are nonlinearly unstable. See \cite[Theorem 7(ii)]{CG}.
\item[(4)] The instability of steady RH waves of degree 2 can only occur by energy transfer between spherical harmonics of degree 2. See \cite[Theorem 7(iii)]{CG}.
\end{itemize}
Although non-zonal RH waves are known to be nonlinearly unstable,  some studies, such as the  numerical analysis in  \cite{Baines,Hos} and the linear stability result in \cite{T},  seem  to support the assertion that degree-2 RH waves are  orbitally stable, i.e., they are stable modulo some continuous symmetries.
In a recent paper, Constantin and Germain proposed this assertion as a conjecture (cf. \cite[Remark 1]{CG}) and suggested that the method developed by Wirosoetisno and Shepherd \cite{WS} could be used to prove it.

 The main purpose of the present paper is to verify Constantin and Germain's  conjecture.
Roughly speaking, our results can be  stated as follows:
 \begin{itemize}
 \item[(S1)]  For any solution of the Euler equation with absolute vorticity $\zeta$, if $\zeta_{0}$ is ``close'' to some RH state $\alpha\sin\theta+Y$ with $\alpha\neq 0,$ $Y\in\mathbb E_{2},$ then at any time $t>0$,  the evolved absolute vorticity $\zeta_{t}$ remains ``close'' to some RH state $\alpha\sin\theta+Y^{t}$, where $Y^{t}$ is a rigid rotation of $Y$ about the polar axis. Here and hereafter, $\zeta_t:=\zeta(\cdot,t)$; in particular, $\zeta_0$ denotes the initial absolute vorticity.
 \item[(S2)]   For any solution of the Euler equation with absolute vorticity $\zeta$, if  $\zeta_{0}$ is ``close'' to some RH state $Y$ with $Y\in\mathbb E_{2},$ then  at any time $t>0$, the evolved absolute vorticity $\zeta_{t}$ remains ``close'' to some RH state $Y^{t}$, where $Y^{t}$ is a rotation of $Y$ under the three-dimensional special orthogonal group $\mathbf S\mathbf O(3)$.
 \end{itemize}
 The rigorous statements of (S1) and (S2) are given in Section~\ref{sec25}.  One notable feature of our results is that the ``closeness" in (S1) and (S2) is  measured   in the $L^p$ norm of the absolute vorticity for any $1<p<+\infty$, which is more general than most previous studies  performed in the Hilbert space setting.

The proofs of our main results
 are based on the Lyapunov functional method. In comparison to Wirosoetisno and Shepherd's approach  in \cite{WS},   our method  entails less technical complications (although  still involves a large amount of algebraic and integral calculations), and
has a wider range of applications (cf. Section~\ref{sec5}).   Roughly speaking,  our  proofs consist of three main steps:
 \begin{itemize}
 \item [(i)] (Variational characterization) We need to study a variational problem regarding the maximization or minimization of some suitable  conserved functional  relative to some invariant set of the Euler equation, and then show that the  solutions under consideration  constitute an isolated set of   maximizers or  minimizers.
   \item [(ii)] (Compactness) We need to prove that any maximizing or minimizing sequence in the variational problem is compact.
    \item [(iii)] (Proof of stability) Using the  conservation laws and the time continuity of  solutions of the Euler equation, we can prove stability based on  the variational characterization in the first step and the compactness in the second step.
 \end{itemize}
 For the sake of convenience in applications, we formulate the third step as a stability criterion (cf. Theorem~\ref{stac} in Section~\ref{sec4}), thus reducing the  stability problem to establishing a suitable variational characterization and verifying the compactness.

 The first step is of significance, as it requires constructing a conserved functional that characterizes the  solutions  appropriately and, at the same time, ensures the compactness in the subsequent step. In fact, establishing a suitable variational characterization is usually the main challenge in  the study of hydrodynamic stability problems. See \cite{Abe,BG,BAR,Bjde,Bcmp,CWCV,CWN,CD,WGu1,WGu2,WZ}.
In the present paper, the variational characterizations  involve  maximizing a combination of conserved quantities of the Euler equation relative to a set of rearrangements of a given function, which are mostly inspired by Burton's work \cite{BMA,BHP,BAR}. In the 1980s, Burton \cite{BMA,BHP} developed a variational theory for convex functionals on a set of  rearrangements of a given function and applied it to the existence of vortex rings.  In 2005, based on the results in \cite{BMA,BHP},  Burton \cite{BAR} established some criteria for the nonlinear stability of planar ideal flows in a general bounded domain. More specifically, Burton proved that, if the vorticity of a steady flow is an isolated maximizer or minimizer of the kinetic energy relative to a set of  rearrangements of a given function, then this steady flow is nonlinearly stable in the $L^{p}$ norm of the vorticity.   For further developments or applications of Burton's stability criteria, see \cite{Bjde,Bcmp,CWCV,CWN,WGu1,WGu2,Wang,WZ} and the references therein.
In these studies, the conserved functional is usually a combination of the kinetic energy of the fluid and some other conserved quantities related to  symmetries of the domain (such as the impulse functional in the case of traveling vortex pairs in the whole plane). For the Euler equation on a  sphere,  an important conserved quantity, due to rotational symmetry, is the projection of the absolute vorticity onto the first eigenspace. Appropriate use of this conservation law is one of the main innovations of this paper.

 It is worth mentioning that   both (S1) and (S2) are \emph{optimal}, which can be verified by making proper use of the traveling RH solutions \eqref{av00} and the rotational invariance of the Euler equation. The detailed discussion is given in Section~\ref{sec8}.

This paper is organized as follows. In  Section~\ref{sec2}, we give the rigorous mathematical formulation of the problem and state our main theorems  concerning the orbital stability of degree-2 RH waves. In Section~\ref{sec3}, we collect the necessary preliminaries, including the $L^{p}$ theory for the Poisson equation on a sphere,  a Poincar\'e-type inequality, and  some properties of the set of rearrangements of a fixed function. In Section~\ref{sec4}, we  prove a  general  stability criterion for subsequent use.
In Section~\ref{sec5}, as a warm-up, we apply the stability criterion    to  study the stability of degree-1 RH waves,  Arnold-type flows, and zonal flows with monotone absolute vorticity.
In Sections~\ref{sec6} and \ref{sec7}, we prove the  main theorems.
In Section~\ref{sec8}, we show that our main theorems are in fact optimal.

\section{Mathematical setting and main results}\label{sec2}
\subsection{Notation and basic facts}\label{sec21}

  Let us start with some notation and definitions that will be used frequently. Throughout this paper, $\mathbb S^2$ denotes the 2-sphere given by \eqref{hb1} with the standard metric induced via pullback of the Euclidean space $\mathbb R^3$.
The longitude-latitude spherical coordinates  for   $\mathbb S^2$ are given by \eqref{hb2}.
In spherical coordinates,   the area element related to the standard metric is
 \[d\sigma=\cos\theta d\varphi d\theta.\]

In  spherical coordinates,  for a scalar function $f$ and two tangent vector fields $\mathbf u,\mathbf v$ on $\mathbb S^2$,
\[\mathbf u=u_{\varphi}\mathbf e_{\varphi}+u_{\theta}\mathbf e_{\theta},\quad \mathbf v=v_{\varphi}\mathbf e_{\varphi}+v_{\theta}\mathbf e_{\theta},\quad \mathbf e_\varphi:=\frac{1}{\cos\theta}\partial_\varphi,\quad \mathbf e_\theta:=\partial_\theta,\]
the covariant derivative of $f$ along $\mathbf v$ is given by
\[\nabla_{\mathbf v} f=\frac{1}{\cos\theta}\partial_{\varphi}fv_\varphi +\partial_{\theta}fv_{\theta},\]
and the covariant derivative of $\mathbf u$ along $\mathbf v$  can be computed by
\[
\nabla_{\mathbf v} \mathbf u=\left(\frac{v_\varphi\partial_\varphi u_\varphi}{\cos\theta} +v_\theta\partial_\theta u_\varphi-v_\varphi u_\theta \tan\theta\right)\mathbf e_\varphi+\left(\frac{v_\varphi\partial_\varphi u_\theta}{\cos\theta} +v_\theta\partial_\theta u_\theta+v_\varphi u_\varphi \tan\theta\right)\mathbf e_\theta.
\]
Classical differential operators used in this paper include
\begin{itemize}
\item   the gradient operator $\nabla$ acting on scalar functions, given by
\[\nabla f=\frac{1}{\cos\theta}\partial_\varphi f\mathbf e_{\varphi}+\partial_{\theta}f\mathbf e_{\theta},\]
\item the divergence operator ${\rm div}$ acting on vector fields, given by
\[{\rm div}\mathbf v=\frac{1}{\cos\theta}\left(\partial_{\varphi}v_{\varphi}+\partial_{\theta}(\cos\theta v_{\theta})\right),\]
\item  and the  Laplace-Beltrami operator $\Delta$ acting on scalar functions, given by
\[\Delta f={\rm div}(\nabla f)=\partial_{\theta\theta}f+\frac{1}{\cos^2\theta}\partial_{\varphi\varphi}f-\tan\theta\partial_\theta f.\]
\end{itemize}
One can easily check that
\begin{equation}\label{fla1}
{\rm div}(f\mathbf v)=\nabla f\cdot\mathbf v+f{\rm div}\mathbf v,
\end{equation}
\begin{equation}\label{fla2}
\nabla\left(\frac{|\mathbf v|^2}{2}\right)=\nabla_{\mathbf v}\mathbf v+{\rm div}(J\mathbf v)J\mathbf v,
\end{equation}
\begin{equation}\label{fla3}
{\rm div}(J\nabla f)=0,
\end{equation}
where $J$ represents the anticlockwise rotation through $\pi/2$ in the $\mathbf e_\varphi$-$\mathbf e_\theta$ plane,
\begin{equation}\label{fla4}
J\mathbf v =-v_\theta\mathbf e_\varphi+v_\varphi\mathbf e_\theta,\quad J\nabla f=-\partial_{\theta}f\mathbf e_{\varphi}+\frac{1}{\cos\theta}\partial_\varphi f\mathbf e_{\theta}.
\end{equation}

Consider the following Laplace-Beltrami eigenvalue problem
\[\Delta u=\lambda u\quad \mbox{on }\mathbb S^2,\quad\int_{\mathbb S^2}ud\sigma=0.\]
The set of eigenvalues is $ \{-j(j+1)\}_{j=1}^{+\infty}$. Note that the zero mean restriction  excludes zero as an eigenvalue.
The eigenspace $\mathbb E_j$ corresponding to the eigenvalue $-j(j+1)$ has a complex basis
$\{Y_j^{m}\}_{|m|\leq j},$
with $ Y_j^{m}$ being the spherical harmonic of degree $j$ and zonal number $m$,
\begin{equation}\label{yjm}
    Y_j^m(\varphi,\theta)=(-1)^m\sqrt{\frac{(2j+1)(j-m)!}{4\pi(j+m)!}}P_j^m(\sin\theta)e^{im\varphi}, \quad m=-j,\cdot\cdot\cdot,0,\cdot\cdot\cdot,j.
\end{equation}
Here $P_j^m:[-1,1]\to\mathbb R$ is an associated  Legendre polynomial, given by
\[P_j^m(x)=\frac{1}{2^jj!}(1-x^2)^{m/2}\frac{d^{j+m}}{dx^{j+m}}(x^2-1)^j.\]
Note that the basis $\{Y_j^m\}$  satisfies
\begin{equation}\label{cjut}
Y_j^{-m}=(-1)^m\overline{Y_j^m},
\end{equation}
where the overbar denotes complex conjugation.  See \cite{Muller,CG}.
A real basis $\{R_j^m\}_{|m|\leq j}$ of the eigenspace  $\mathbb E_j$ can be obtained from the complex basis $\{Y_j^m\}_{|m|\leq j}$:
\[R_j^m(\varphi,\theta)=
\begin{cases}
\sqrt{\frac{(2j+1)(j-|m|)!}{2\pi(j+|m|)!}}P_j^{|m|}(\sin\theta)\sin(|m|\varphi),\quad &-j\leq m\leq -1,\\
 \sqrt{\frac{2j+1 }{4\pi}}P_j^0(\sin\theta),\quad &m=0,\\
(-1)^m\sqrt{\frac{(2j+1)(j-m)!}{2\pi(j+m)!}}P_j^{m}(\sin\theta)\cos(m\varphi),\quad& 1\leq m\leq j.
\end{cases}\]
The following explicit formulas for the degree-1 and degree-2 spherical harmonics are used in this paper:
\[  Y_1^0=\sqrt{\frac{3}{4\pi}}\sin\theta,\quad Y_1^{\pm1}=\mp\sqrt{\frac{3}{8\pi}}\cos\theta e^{\pm i\varphi},
\]
\[Y_2^0=\sqrt{\frac{5}{16\pi}}(3\sin^2\theta-1),\quad  Y_2^{\pm 1}=\mp\sqrt{\frac{15}{8\pi}}\sin\theta\cos \theta e^{\pm  i\varphi},\quad Y_2^{\pm 2}=\sqrt{\frac{15}{32\pi}}\cos^2\theta e^{\pm 2i\varphi}. \]
The corresponding real spherical harmonics have the following expressions:
\[   R_1^0=\sqrt{\frac{3}{4\pi}}\sin\theta,\quad  R_1^{1}= -\sqrt{\frac{3}{4\pi}}\cos\theta\cos\varphi,\quad R_1^{-1}= \sqrt{\frac{3}{4\pi}}\cos\theta \sin\varphi.\]
\[ R_2^0=\sqrt{\frac{5}{16\pi}}(3\sin^2\theta-1),\quad R_2^{1}=-\sqrt{\frac{15}{4\pi}}\sin\theta\cos\theta \cos\varphi,\quad R_2^{-1}=\sqrt{\frac{15}{4\pi}}\sin\theta\cos\theta \sin\varphi.\]
\[R_2^{2}=\sqrt{\frac{15}{16\pi}}\cos^2\theta \cos(2\varphi),\quad  R_2^{-2}=\sqrt{\frac{15}{16\pi}}\cos^2\theta \sin(2\varphi).\]
In particular,
\begin{equation}\label{e2span}
\mathbb E_2=\mbox{span}\{3\sin^2\theta-1, \sin(2\theta)\cos\varphi, \sin(2\theta)\sin\varphi, \cos^2\theta\cos(2\varphi), \cos^2\theta\sin(2\varphi)\}.
\end{equation}

 For $1\leq p\leq +\infty,$ denote by $\mathring L^p(\mathbb S^2)$ the set of all $L^p$ functions\footnote{All functions on $\mathbb S^2$ involved in this paper are assumed to be real-valued except for the complex spherical harmonics $\{Y_j^m\}_{j\geq 1,|m|\leq j}$.} on $\mathbb S^2$ with zero mean, i.e.,
 \[\mathring {L}^p(\mathbb S^2):=\left\{f\in L^{p}(\mathbb S^{2})\mid \int_{\mathbb S^{2}}fd\sigma=0\right\}.\]
 Analogously, denote
 \[ \mathring {W}^{k,p}(\mathbb S^2):=\left\{f\in W^{k,p}(\mathbb S^{2})\mid \int_{\mathbb S^{2}}fd\sigma=0\right\},\quad k\in\mathbb Z_+,\,1\leq p\leq +\infty,\]
 \[\mathring C^k(\mathbb S^2):=\left\{f\in C^{k}(\mathbb S^{2})\mid \int_{\mathbb S^{2}}fd\sigma=0\right\},\quad k\in\mathbb Z_+\cup\{0,+\infty\}.\]
 It is easy to check that $\mathring {L}^p(\mathbb S^2),$ $\mathring {W}^{k,p}(\mathbb S^2)$ and $\mathring C^k(\mathbb S^2)$  are closed subspaces of ${L}^p(\mathbb S^2)$,  ${W}^{k,p}(\mathbb S^2)$  and $C^k(\mathbb S^2)$, respectively.

For  $f\in \mathring L^2(\mathbb S^2)$, we have the following Fourier expansion in terms of  $\{Y_j^m\}_{j\geq 1,|m|\leq j}$:
 \begin{equation}\label{ty0}
 f=\sum_{j\geq 1}\sum_{|m|\le j}c_j^mY_j^m, \quad c_j^m:=\int_{\mathbb S^2}f\overline{ Y_j^m}d\sigma.
 \end{equation}
Since $f$ is real-valued,  in view of \eqref{cjut}, one can easily check that the Fourier coefficients $\{c_j^m\}_{j\geq 1, |m|\leq j}$ satisfy
\[c_j^{-m}=(-1)^m\overline{{c}_j^m},\quad\forall\,j\geq 1,\,|m|\leq j.\]

 Denote by $\mathcal G$ the inverse of $-\Delta$ on $\mathbb S^2$ with zero mean condition, i.e., for any scalar function $f$ with zero mean,
\begin{equation}\label{grop}
-\Delta (\mathcal Gf)=f,\qquad \int_{\mathbb S^2}\mathcal Gf d\sigma=0.
\end{equation}
For example,
\begin{equation}\label{gsint}
\mathcal G(\sin\theta)=\frac{1}{2}\sin\theta.
\end{equation}
The operator $\mathcal G$ is a bounded linear operator from $\mathring L^p(\mathbb S^2)$ to $\mathring W^{2,p}(\mathbb S^2)$ for any $1<p<+\infty$ (cf. Lemma~\ref{lpt} in Section~\ref{sec3}). If $f\in\mathring L^2(\mathbb S^2)$ has the Fourier expansion \eqref{ty0}, then   $\mathcal Gf$ can be expressed as
 \[\mathcal Gf=\sum_{j\geq 1}\sum_{|m|\le j}\frac{c_j^m}{j(j+1)}Y_j^m.\]

  Denote by $ \mathbf O(3)$  the three-dimensional orthogonal group, and by $ \mathbf S\mathbf O(3)$ the three-dimensional special orthogonal group (also called the three-dimensional rotation group).
For any function $f:\mathbb S^2\to\mathbb R,$ define
 \begin{equation}\label{oy}
\mathcal O_{f}:=\{ f\circ\mathsf g\mid \mathsf g\in\mathbf O(3)\},
\end{equation}
 \begin{equation}\label{oyp}
\mathcal O^+_{f}:=\{ f\circ\mathsf g\mid \mathsf g\in\mathbf S\mathbf O(3)\}.
\end{equation}
 Denote by $\mathbf H$ the orthogonal group in the $x_{1}$-$x_{2}$ plane, and by $\mathbf H^+$ the special orthogonal group in the $x_{1}$-$x_{2}$ plane. Note that $\mathbf H$ is a subgroup of $\mathbf O(3)$, and $\mathbf H^+$ is a subgroup of $  \mathbf S\mathbf O(3)$.
For any function $f:\mathbb S^2\to\mathbb R,$ define
 \begin{equation}\label{hy}
 \mathcal H_f:=\{f\circ \mathsf g\mid \mathsf g\in\mathbf H\},
 \end{equation}
 \begin{equation}\label{hyp}
 \mathcal H^{+}_f:=\{f\circ \mathsf g\mid \mathsf g\in\mathbf H^+\}.
 \end{equation}
If $f=f(\varphi,\theta)$ is given in spherical coordinates, then it is easy to see that
 \begin{equation}\label{hy2}
 \mathcal H_f=\{f(\pm\varphi+\beta,\theta)\mid \beta\in\mathbb R\},
 \end{equation}
 \begin{equation}\label{hyp2}
 \mathcal H^{+}_f=\{f(\varphi+\beta,\theta)\mid \beta\in\mathbb R\}.
 \end{equation}
 Intuitively, $\mathcal H^+_{f}$ is the set of all rotations of $f$ about the polar axis.
 From \eqref{hy2}--\eqref{hyp2} and the expressions of degree-1 spherical harmonics, it is easy to check that
  \begin{equation}\label{y1eq}
  \mathcal H_Y=\mathcal H^+_Y,  \quad \forall\,Y\in\mathbb E_1.
 \end{equation}

 Let $\mathsf m$ be the measure on $\mathbb S^2$ associated with the standard metric.  Given an $\mathsf m$-measurable  function $f:\mathbb S^2\to\mathbb R,$ denote by $\mathcal R_{f}$ the set of  rearrangements of $f$ on $\mathbb S^2$ with respect to the measure $\mathsf m$, i.e.,
 \begin{equation}\label{uu01}
 \mathcal R_{f}=\left\{ g:\mathbb S^{2}\to\mathbb R\mid \mathsf m\left(\{\mathbf  x\in\mathbb S^2\mid g(\mathbf   x)>s\}\right)=\mathsf m\left(\{\mathbf   x\in\mathbb S^2\mid f(\mathbf   x)>s\}\right)\,\,\forall\,s\in\mathbb R\right\} .
 \end{equation}

\subsection{Euler equation on a rotating sphere}\label{sec22}
Consider the Euler equation of an ideal fluid of unit density on  $\mathbb S^2$  rotating with speed $\omega$ about the polar axis (cf. \cite{CG,KBH,T}):
   \begin{equation}\label{e0}
\begin{cases}
\partial_t\mathbf v+\nabla_{\mathbf v}\mathbf v+2\omega\sin\theta J\mathbf v  =-\nabla P,\\
{\rm div}\mathbf v=0,
\end{cases}
\end{equation}
where $\mathbf v=v_\varphi\mathbf e_{\varphi}+v_\theta\mathbf e_\theta$ is the velocity field, $P$ is the scalar pressure, and $J$  denotes the anticlockwise  rotation through $\pi/2$ in the tangent space as in \eqref{fla4},
\[J\mathbf v=-v_\theta\mathbf e_{\varphi}+v_\varphi\mathbf e_\theta.\]

By the standard energy method \cite{MB}, together with an analogue of the Beale--Kato--Majda criterion \cite{BKM}, the Euler equation \eqref{e0} is globally well-posed in $H^s$ at the level of the velocity field for any $s>2$. See, for example, Taylor \cite[Section 2]{T} for a detailed proof. In particular, it makes sense to talk about smooth solutions of \eqref{e0}.

 The vorticity-stream function formulation of the Euler equation \eqref{e0} is used in this paper. Since $\mathbf v$ is divergence-free, there is a scalar function $\psi$, called the \emph{stream function}, such that
\begin{equation}\label{jna1}
\mathbf v=J\nabla \psi.
\end{equation}
The \emph{relative vorticity} $\Omega$ and the \emph{absolute vorticity} $\zeta$ are defined by
\begin{equation}\label{abvort}
\Omega:=-{\rm div}(J\mathbf v),\quad\zeta:=\Omega+2\omega\sin\theta.
\end{equation}
Note that  $\Omega$ and $\zeta$ are of zero mean automatically. Applying the operator $-{\rm div}J$ to the first equation of \eqref{e0},  using the formulas \eqref{fla1}--\eqref{fla3},  and taking into  account the divergence-free condition ${\rm div}\mathbf v=0$, we obtain
\begin{equation}\label{omeg0}
\partial_t\Omega+\mathbf v\cdot\nabla( \Omega+2\omega\sin\theta)=0.
\end{equation}
At the level of absolute vorticity,
  the equation \eqref{omeg0} becomes
\[
\partial_t\zeta+ \mathbf v\cdot\nabla\zeta=0.
\]
On the other hand, by \eqref{jna1} and the definition of $\Omega$ we see that $\psi$ and $\Omega$ satisfy the Poisson equation
$\Delta\psi=\Omega.$
Throughout this paper, we always assume that the stream function $\psi$ has a zero mean, which can be achieved by adding a suitable constant to it without changing \eqref{jna1}. Recalling the operator $\mathcal G$  given by \eqref{grop}, we have that
\begin{equation}\label{gsint2}
\psi=-\mathcal G\Omega=-\mathcal G(\zeta-2\omega\sin\theta)=\omega\sin\theta-\mathcal G\zeta.
\end{equation}
Here we have used \eqref{gsint}.
Therefore the velocity field $\mathbf v$ can be expressed in terms of $\zeta$ via the following Biot--Savart law:
\[\mathbf v=J\nabla \psi=J\nabla( \omega\sin\theta-\mathcal G\zeta).\]
 To conclude, we have obtained the following vorticity equation for $\zeta$:
\begin{equation}\label{ave}
\partial_t\zeta+J\nabla(\omega\sin\theta-\mathcal G\zeta)\cdot\nabla\zeta=0. \tag{$V_{\omega}$}
\end{equation}
This is a first-order nonlinear transport equation on $\mathbb S^2.$

Note that $(V_{0})$ is invariant under three-dimensional rigid rotations, and \eqref{ave} is invariant under  three-dimensional rigid rotations about the polar axis for general $\omega\in\mathbb R$. Also note that
\begin{equation}\label{iffsl}
\mbox{$\zeta(\varphi,\theta,t)$ solves $(V_{0})$ if and only if $\zeta(\varphi+\omega t,\theta,t)$ solves \eqref{ave}.}
\end{equation}
 See \cite[Section 2.3]{CG}.

 In the rest of this paper, we focus primarily on  \eqref{ave} rather than its primitive form \eqref{e0}.

\subsection{Steady solutions and RH  waves}\label{sec23}
 Steady solutions  of the Euler equation are characterized   by having $\nabla\psi$ and $\nabla(\Delta \psi+2\omega\sin\theta)$ parallel.
In particular, this holds when  $\psi$ is zonal, i.e., $\psi$ depends only on $\theta$. For some nonlinear stability results of zonal flows, we refer the interested reader to \cite{Cap,CG,T} and the references therein.

In addition to zonal solutions, steady solutions can also be obtained by considering the following semilinear elliptic equation for the stream function $\psi$:
\begin{equation}\label{semm}
\Delta\psi+2\omega\sin\theta=\mathfrak g(\psi),
\end{equation}
where $\mathfrak g\in C^1(\mathbb R)$. In terms of the absolute vorticity,   \eqref{semm} can be written as
\[\zeta=\mathfrak g(\omega\sin\theta-\mathcal G\zeta).\]
Here we have used \eqref{gsint2}. Recently,  some rigidity results on steady solutions related to the elliptic equation \eqref{semm} were proved by Constantin and Germain (\cite[Theorem 4]{CG}).

Now we introduce the Rossby-Haurwitz (RH) waves. Consider
 \begin{equation}\label{drhw}
 \zeta(\varphi,\theta,t)=\alpha\sin\theta+Y(\varphi-ct,\theta),\quad Y\in \mathbb E_j\setminus\{0\}.
 \end{equation}
 By direct calculations, one can check that if $\alpha, \omega$ and $c$ satisfy the relation
 \begin{equation}\label{cao}
 c=\alpha\left(\frac{1}{2}- \frac{1}{j(j+1)}\right)- \omega,
 \end{equation}
 then $\zeta$ given by \eqref{drhw} solves the vorticity equation \eqref{ave}.
 We call such a solution $\zeta$ a \emph{Rossby-Haurwitz wave}  of degree $j$.
 It is clear that an RH wave of the form \eqref{drhw} is steady if and only if $Y$ depends only on $\theta$, or $\alpha$ and $\omega$ satisfy the relation
\[\omega=\alpha\left(\frac{1}{2}- \frac{1}{j(j+1)}\right).\]
For example,
a degree-1 RH wave  has the form
   \begin{equation}\label{dg1rh}
   \zeta(\varphi,\theta,t)=Y(\varphi+\omega t,\theta),\quad Y\in \mathbb E_1\setminus\{0\},
   \end{equation}
  which is steady if and only if
 $\omega=0$, or $Y$ depends only on $\theta$.
A degree-2 RH wave has the form
   \begin{equation}\label{dg2rh}
   \zeta(\varphi,\theta,t)=\alpha\sin\theta+Y(\varphi-ct,\theta),\quad Y\in \mathbb E_2\setminus\{0\},\,\,c=\frac{1}{3}\alpha-\omega,
   \end{equation}
 which is steady if and only if $\alpha=3\omega,$ or  $Y$ depends only on $\theta$. We refer the interested reader to \cite{Nualart}
for some recent results on the structure of the set of steady solutions   near a  zonal RH wave  of degree 1 or 2.

Given $\alpha\in\mathbb R$ and $Y\in\mathbb E_j,$ we call $\alpha\sin\theta+Y$ an \emph{RH state} of degree $j$.
In view of  \eqref{drhw},  the  set of all rotations of the RH state $\alpha\sin\theta+Y$ about the polar axis, denoted by $\mathcal H^+_{\alpha\sin\theta+Y}$  according to \eqref{hyp},
is an invariant set  of  the vorticity equation \eqref{ave}. More precisely, for any smooth solution $\zeta$ of \eqref{ave},
\begin{equation}\label{rhorbt}
 \zeta_0\in \mathcal H^+_{\alpha\sin\theta+Y}\quad \Longrightarrow \quad \zeta_t\in \mathcal H^+_{\alpha\sin\theta+Y}\,\,\forall\,t>0.
\end{equation}
Here $\zeta_t$ is the abbreviation for $\zeta(\cdot,t).$  For this reason, we call $\mathcal H^+_{\alpha\sin\theta+Y}$ an  \emph{RH orbit}.

 \subsection{Conservation laws}\label{sec24}
In the study of nonlinear stability of equilibria of non-dissipative systems,
 conservation laws usually play an important role.
For the Euler equation on a rotating sphere, there are three basic conservation laws:
\begin{itemize}
    \item [(i)] The kinetic energy of the fluid is conserved, i.e.,
    \begin{equation}\label{cc1}
    \frac{1}{2}\int_{\mathbb S^{2}}|\mathbf v_t|^{2}d\sigma=\frac{1}{2}\int_{\mathbb S^{2}}|\mathbf v_0|^{2}d\sigma,\quad\forall\,t>0.
    \end{equation}
    \item [(ii)]  The distribution function of the absolute vorticity is conserved, i.e.,
       \begin{equation}\label{cc2}
 \zeta_{t}\in\mathcal R_{\zeta_{0}},\quad \forall\,t>0.
    \end{equation}
    See \eqref{uu01} for the definition of $\mathcal R_{\zeta_{0}}$.
       \item [(iii)] The quantities $e^{-im\omega t}c_{1}^{m}(t),$ $m=-1,0,1,$ are conserved, i.e.,
            \begin{equation}\label{cc3}
                   e^{-im\omega t}c_{1}^{m}(t)=  c_{1}^{m}(0), \quad \forall\,t>0,\,m=-1,0,1,
      \end{equation}
       where $c_{1}^{m}(t)$  is the Fourier coefficient  of $\zeta_t$ associated with the spherical harmonic $Y_{1}^{m}$, i.e.,
     \begin{equation}\label{cc32}
              c_{1}^{m}(t)=\int_{\mathbb S^{2}}\zeta_{t}\overline{Y_{1}^{m}}d\sigma,\quad m=-1,0,1.
      \end{equation}

       \end{itemize}
For detailed proofs of the conservation laws \eqref{cc1}--\eqref{cc3},  we refer the reader to \cite{CG}, Section 2.4 (therein the conservation law \eqref{cc3} is given in terms of the relative vorticity).

Let $\zeta$ be a smooth solution to the vorticity equation \eqref{ave}.  Then $\zeta_t$ is of zero mean for any  $t>0$, and thus can be expanded in terms of the basis $\{Y_j^m\}_{j\geq 1,|m|\leq j}$ as follows:
\begin{equation}\label{expd}
\zeta_t =\sum_{j\geq 1}\sum_{|m|\le j}c_j^m(t)Y_j^m,\quad c_j^m(t)=\int_{\mathbb S^2}\zeta_t\overline{ Y_j^m} d\sigma
\end{equation}
One can check that the three conservation laws \eqref{cc1}--\eqref{cc3}  can be equivalently expressed as follows:
\begin{equation}\label{cl1}
\sum_{j\geq 1}\sum_{|m|\le j}\frac{|c_j^m(t)|^2}{j(j+1)}=\sum_{j\geq 1}\sum_{|m|\le j}\frac{|c_j^m(0)|^2}{j(j+1)},\quad\forall\,t>0,
\end{equation}
\begin{equation}\label{cl2}
     {\zeta_t}\in\mathcal R_{\zeta_0},\quad\forall\,t>0.
 \end{equation}
\begin{equation}\label{cl3}
       e^{ i\omega t}c_1^{ -1}(t)=c_1^{-1}(0),\quad  e^{-i\omega t}c_1^{ 1}(t)=c_1^1(0),\quad  c_1^0(t)=c_1^0(0), \quad\forall\,t>0.
\end{equation}
The  conservation laws \eqref{cl1}--\eqref{cl3} will be frequently used in this paper.
Note that a direct consequence of \eqref{cl2} is $\|\zeta_t\|_{L^2(\mathbb S^2)}=\|\zeta_0\|_{L^2(\mathbb S^2)}$, or, equivalently,
 \begin{equation}\label{cl4}
 \sum_{j\geq 1}\sum_{|m|\le j} |c_j^m(t)|^2= \sum_{j\geq 1}\sum_{|m|\le j} |c_j^m(0)|^2,\quad\forall\,t>0.
\end{equation}

As an application of the conservation laws \eqref{cl3} and \eqref{cl4}, we can
easily prove the stability of degree-1 RH waves in  $L^2$ norm of the absolute vorticity.
Consider a smooth solution $\zeta$ to the vorticity equation \eqref{ave} with the expansion \eqref{expd}.
Given $Y\in\mathbb E_1$ with
 \[Y=a Y_1^0+bY_1^1+cY_1^{-1}\in\mathbb E_1,\quad b=-\overline c,\]
it is clear that
\begin{equation}\label{inc1}
    \|\zeta_0-Y\|_{L^2(\mathbb S^2)}^2=| c_1^0(0)-a|^2+|c_1^{1}(0)-b|^2+|c_1^{-1}(0)-c|^2+\sum_{j\geq 2}\sum_{|m|\leq j}|c_j^m(0)|^2.
\end{equation}
By the conservation laws \eqref{cl3} and \eqref{cl4}, it holds for any $t>0$ that
  \begin{equation}\label{inc10}
  c_1^0(t)= c_1^0(0),\quad  e^{\mp i\omega t}c_1^{\pm 1}(t)= c_1^{\pm1}(0),\quad \sum_{j\geq 2}\sum_{|m|\le j} |c_j^m(t)|^2=\sum_{j\geq 2}\sum_{|m|\le j} |c_j^m(0)|^2.
  \end{equation}
From \eqref{inc1} and \eqref{inc10}, we have that
\begin{equation}\label{inc20}
| c_1^0(t)-a|^2+|c_1^{1}(t)-e^{i\omega t}b|^2+|c_1^{-1}(t)-e^{-i\omega t}c|^2+\sum_{j\geq 2}\sum_{|m|\leq j}|c_j^m(t)|^2= \|\zeta_0-Y\|_{L^2(\mathbb S^2)}^2.
\end{equation}
Denote
\[Y^t=
Y(\varphi+\omega t,\theta)=aY_1^0+e^{i\omega t}b Y_1^{1}+e^{-i\omega t}c Y_1^{-1}.\]
Then \eqref{inc20} can be written as
\begin{equation}\label{inc27}
  \|\zeta_0-Y\|_{L^2(\mathbb S^2)}= \|\zeta_t-Y^t\|_{L^2(\mathbb S^2)}.
\end{equation}
By the definition of $Y^t$ and \eqref{hyp2}, we have $Y^t\in\mathcal H_Y^+$. From \eqref{inc27}, we immediately see that the RH orbit $\mathcal H^+_{Y}$ is stable in the $L^2$ norm of the absolute vorticity. More precisely, for any smooth solution $\zeta$ of the vorticity equation \eqref{ave} and any $\varepsilon>0$, there exists $\delta>0$ such that
\begin{equation}\label{dygstb}
 \min_{f\in \mathcal H^+_{Y} }\|f-\zeta_0\|_{L^2(\mathbb S^2)}<\delta \quad\Longrightarrow\quad \min_{f\in \mathcal H^+_{Y}}\|f-\zeta_t\|_{L^2(\mathbb S^2)}<\varepsilon,\,\,\forall\,t>0.
 \end{equation}
 If, in addition, $\omega=0$, then $Y^t=Y$ for all $t>0$. Hence, \eqref{inc27} implies the stability of $Y$ in the $L^2$ norm of the absolute vorticity, in the sense of \eqref{dygstb}.

\subsection{Main results}\label{sec25}
The main purpose of this paper is to investigate the nonlinear stability of
the  degree-2 RH waves. Consider the following degree-2 RH state at the absolute vorticity level:
\[\alpha\sin\theta+Y,\quad \alpha\in\mathbb R,\,\,  Y\in\mathbb E_2.\]
Our first result deals with the case  $\alpha\neq 0.$

     \begin{theorem}[Stability of degree-2 RH waves: $\alpha\neq 0$]\label{thm1}
Let $1<p<+\infty$, and let $\alpha\in\mathbb R\setminus\{0\}$, $Y\in\mathbb E_2\setminus\{0\}$. Then the RH orbit $\mathcal H^+_{\alpha\sin\theta+Y}$ (defined according to  \eqref{hyp} or \eqref{hyp2}) is stable in the $L^p$ norm of the absolute vorticity. More precisely, for  any $\varepsilon>0,$ there exists $\delta>0$ such that
 for any smooth solution $\zeta$ of the vorticity equation \eqref{ave},
 \begin{equation}\label{hrb1}
 \min_{f\in \mathcal H^+_{\alpha\sin\theta+Y} }\|f-\zeta_0\|_{L^p(\mathbb S^2)}<\delta\quad\Longrightarrow\quad \min_{f\in \mathcal H^+_{\alpha\sin\theta+Y}}\|f-\zeta_t\|_{L^p(\mathbb S^2)}<\varepsilon,\,\,\forall\,t>0.
 \end{equation}
\end{theorem}

A straightforward corollary of Theorem~\ref{thm1} is the following stability result for degree-2 zonal RH waves.

      \begin{corollary}[Stability of degree-2 zonal RH waves: $\alpha\neq 0$]\label{thm10}
Let $1<p<+\infty$, and let $\alpha$, $\beta$ be non-zero real numbers. Then  the degree-2 RH wave  $\alpha\sin\theta+\beta Y_2^0$  is   stable in the $L^p$ norm of the absolute vorticity. More precisely, for  any $\varepsilon>0,$ there exists $\delta>0$ such that
 for any smooth solution $\zeta$ of the vorticity equation \eqref{ave},
 \begin{equation*}
  \| \zeta_0-(\alpha\sin\theta+\beta Y_2^0)\|_{L^p(\mathbb S^2)}<\delta\quad\Longrightarrow\quad  \| \zeta_t-(\alpha\sin\theta+\beta Y_2^0)\|_{L^p(\mathbb S^2)}<\varepsilon,\,\,\forall\,t>0.
 \end{equation*}
    \end{corollary}

For the case $\alpha=0$,  we can prove stability modulo rotations in  $\mathbf S\mathbf O(3)$.

 \begin{theorem}[Stability of degree-2 RH waves: $\alpha=0$]\label{thm2}
Let $1<p<+\infty,$ and let $Y\in\mathbb E_2\setminus\{0\}$. Then the set  $\mathcal O^+_{Y}$ (defined by  \eqref{oyp}) is  stable in the $L^p$ norm of the absolute vorticity. More precisely, for  any $\varepsilon>0,$ there exists $\delta>0$ such that
 for any smooth solution $\zeta$ of the vorticity equation \eqref{ave},
 \begin{equation}\label{hrb2}
 \min_{f\in\mathcal O^+_{ Y}}\|f-\zeta_0\|_{L^p(\mathbb S^2)}<\delta\quad\Longrightarrow\quad \min_{f\in\mathcal O^+_{Y}}\|f-\zeta_t\|_{L^p(\mathbb S^2)}<\varepsilon,\,\,\forall\,t>0.
 \end{equation}
 \end{theorem}

It is worth mentioning that our stability results can equivalently be formulated in terms of the $L^p$ norm of the relative vorticity or the $W^{2,p}$ norm of the stream function. For example, the conclusion in Theorem~\ref{thm1} can be equivalently stated as follows:
\begin{itemize}
\item [(i)] (Stability in terms of the relative vorticity) Let $Y\in\mathbb E_{2}\setminus\{0\}$, and let $\alpha\in\mathbb R$ with $\alpha\neq -2\omega.$ Then for any $\varepsilon>0,$ there exists some $\delta>0$ such that for any  smooth Euler flow on $\mathbb S^{2} $ with   relative vorticity  $\Omega$, we have that
\begin{equation*}
 \min_{f\in \mathcal H^+_{\alpha\sin\theta+Y} }\|f-\Omega_0\|_{L^p(\mathbb S^2)}<\delta\,\,\Longrightarrow\,\, \min_{f\in \mathcal H^+_{\alpha\sin\theta+Y}}\|f-\Omega_t\|_{L^p(\mathbb S^2)}<\varepsilon,\,\,\forall\,t>0.\end{equation*}
This can be easily verified using the fact that $\zeta=\Omega+2\omega\sin\theta.$
\item [(ii)] (Stability in terms of the stream function) Let $Y\in\mathbb E_{2}\setminus\{0\}$, and let $\alpha\in\mathbb R$ with $\alpha\neq  \omega.$ Then for any $\varepsilon>0,$ there exists some $\delta>0$ such that for any  smooth Euler flow on $\mathbb S^{2} $ with   stream function  $\psi$, we have that
 \[\min_{f\in \mathcal H^{+}_{\alpha\sin\theta+Y}}\|f-\psi_0\|_{W^{2,p}(\mathbb S^2)}<\delta\,\,\Longrightarrow\,\, \min_{f\in \mathcal H^{+}_{\alpha\sin\theta+Y}}\|f-\psi_t\|_{W^{2,p}(\mathbb S^2)}<\varepsilon,\,\,\forall\,t>0.\]
This follows from the fact that $\psi=\omega\sin\theta-\mathcal G\zeta$ and  the $L^p$ estimate in Lemma~\ref{lpt}.

\end{itemize}

We also mention that the stabilities in Theorems~\ref{thm1} and \ref{thm2} may also hold for less regular perturbations. See  Remark~\ref{lessregular} in Section~\ref{sec4}.

 \section{Preliminaries}\label{sec3}
 \subsection{$L^p$ theory for Poisson equation on a sphere}\label{sec31}

The following result may be known to some extent, but we have been unable to locate a reference for it in the literature. We therefore state and prove it here for the reader's convenience.

\begin{lemma}\label{lpt}
Let $1<p<+\infty$. Then for any $f\in \mathring L^p(\mathbb S^2)$, there is a unique solution $u\in \mathring W^{2,p}(\mathbb S^2)$ to the following Poisson equation on $\mathbb S^{2}$:
\begin{equation}\label{deuf}
 -\Delta u=f \,\,\mbox{ \rm a.e. on }\,\, \mathbb S^2.
\end{equation}
Moreover, this unique solution $u$ satisfies the following $L^{p}$ estimate:
\begin{equation}\label{estp}
\|u\|_{W^{2,p}(\mathbb S^2)}\leq C\|f\|_{L^p(\mathbb S^2)},
\end{equation}
where $C$ is a positive number depending only on $p$.
\end{lemma}

\begin{proof}
For simplicity, denote by $C$  various positive constants depending only on $p$. For clarity, the proof is divided into four steps.

{\bf Step 1.}
If $f\in\mathring C^\infty(\mathbb S^2),$
then \eqref{deuf} admits a unique solution  $u\in\mathring C^\infty(\mathbb S^2).$
To prove this, we expand $f$ in terms of the basis $\{Y_j^m\}_{j\geq 1,|m|\leq j}$ as follows:
 \begin{equation}\label{expf}
 f=\sum_{j\geq 1}\sum_{|m|\le j}c_j^mY_j^m, \quad c_j^m=\int_{\mathbb S^2}f\overline{ Y_j^m}d\sigma.
 \end{equation}
Then it is easy to check that
 \begin{equation}\label{expu}
 u=\sum_{j\geq 1}\sum_{|m|\le j}\frac{c_j^m}{j(j+1)}Y_j^m
 \end{equation}
 belongs to $\mathring C^\infty(\mathbb S^2)$ and satisfies \eqref{deuf}. To prove uniqueness, we assume that \eqref{deuf} has two solutions $u_1, u_2\in\mathring C^\infty(\mathbb S^2)$. Then
$ \Delta (u_1-u_2)=0$.
By integration by parts,
\[\int_{\mathbb S^2}|\nabla(u_1-u_2)|^2d\sigma=0,\]
implying that $u_1-u_2$ is constant. Taking into account the fact that both $u_{1}$ and $u_{2}$ are of zero mean,
we obtain $u_1\equiv u_2$.

{\bf Step 2.} If $u\in\mathring C^\infty(\mathbb S^2)$ is a solution to the Poisson equation \eqref{deuf}, then the $L^p$ estimate \eqref{estp} holds.  First by standard interior $L^p$ estimates (cf. \cite[Theorem 9.11]{GT}) and a gluing procedure, it is easy to prove  the  estimate:
\begin{equation}\label{lso00}
\|u\|_{W^{2,p}(\mathbb S^2)}\leq C(\|u\|_{L^p(\mathbb S^2)}+\|f\|_{L^p(\mathbb S^2)}).
\end{equation}
Below we show that the term $\|u\|_{L^p(\mathbb S^2)}$ can be removed. We distinguish three cases.
\begin{itemize}
\item[Case 1.]  $p=2$. Expand $f$ as in \eqref{expf}. Then
\begin{equation}\label{lso1}
\|f\|_{L^2(\mathbb S^2)}^2=\sum_{j\ge 1}\sum_{|m|\leq j}|c_j^m|^2.
\end{equation}
On the other hand, by the uniqueness  assertion in Step 1, $u$ must be given by \eqref{expu}. Hence
\begin{equation}\label{lso2}
\|u\|_{L^2(\mathbb S^2)}^2=\sum_{j\ge 1}\sum_{|m|\leq j}\frac{|c_j^m|^2}{j^2(j+1)^2}.
\end{equation}
From \eqref{lso1} and \eqref{lso2}, we obtain
\begin{equation}\label{lso3}
\|u\|_{L^2(\mathbb S^2)}\leq \frac{1}{2}\|f\|_{L^2(\mathbb S^2)}.
\end{equation}
Hence \eqref{estp} holds.
    \item [Case 2.] $p> 2$. In this case, by H\"older's inequality and Sobolev embeddings (cf. \cite[Section 3.3]{Hebey}),
    \begin{equation}\label{lso9}
    \begin{split}
        \|u\|_{L^p(\mathbb S^2)}&\leq \|u\|^{2/p}_{L^2(\mathbb S^2)}\|u\|^{1-2/p}_{L^\infty(\mathbb S^2)}\\
      & \leq C\|u\|^{2/p}_{L^2(\mathbb S^2)}\|u\|^{1-2/p}_{W^{2,p}(\mathbb S^2)}\\
      &\leq  C\|f\|^{2/p}_{L^2(\mathbb S^2)}\|u\|^{1-2/p}_{W^{2,p}(\mathbb S^2)}\\
      &\leq \varepsilon \|u\|_{W^{2,p}(\mathbb S^2)}+C_\varepsilon\|f\|_{L^2(\mathbb S^2)}\\
     &\leq \varepsilon \|u\|_{W^{2,p}(\mathbb S^2)}+C_\varepsilon\|f\|_{L^p(\mathbb S^2)},
    \end{split}
    \end{equation}
    where $\varepsilon>0$ is to be determined and $C_\varepsilon$ depends on $\varepsilon$ and $p$. Note that in the third inequality we have used the estimate \eqref{lso3} in Case 1.  The desired estimate \eqref{estp} follows from \eqref{lso00} and \eqref{lso9} by choosing $\varepsilon$ small enough.

     \item [Case 3.] $1<p<2$. We deal with this case  by duality. For  any $g\in \mathring C^\infty(\mathbb S^2)$, by Step 1, there is a unique  $v\in \mathring C^\infty(\mathbb S^2)$  such that
     \[-\Delta v=g.\]
     Since $p/(p-1)> 2$, we deduce from     Case 2  that
     \begin{equation}\label{gpw0}
         \|v\|_{L^{p/(p-1)}(\mathbb S^2)}\leq C\|g\|_{L^{p/(p-1)}(\mathbb S^2)}.
     \end{equation}
     By integration by parts,
      \[
        \left|\int_{\mathbb S^2}ug\,d\sigma\right|=\left|\int_{\mathbb S^2}vf\,d\sigma\right|\leq\|f\|_{L^p(\mathbb S^2)}\|v\|_{L^{p/(p-1)}(\mathbb S^2)}\leq C\|f\|_{L^p(\mathbb S^2)}\|g\|_{L^{p/(p-1)}(\mathbb S^2)}.
    \]
Now let $h\in C^\infty(\mathbb S^2)$ and set
\[
 c_h:=\frac{1}{4\pi}\int_{\mathbb S^2}h\,d\sigma,
 \qquad h_0:=h-c_h,
 \qquad p':=\frac{p}{p-1}.
\]
Then $h_0\in\mathring C^\infty(\mathbb S^2)$, and
H\"older's inequality gives
\[
 \|h_0\|_{L^{p'}(\mathbb S^2)}
 \leq \|h\|_{L^{p'}(\mathbb S^2)}
       +(4\pi)^{1/p'}|c_h|
 \leq 2\|h\|_{L^{p'}(\mathbb S^2)}.
\]
Since $u$ has zero mean, the preceding estimate yields
\[
 \left|\int_{\mathbb S^2}uh\,d\sigma\right|
 =\left|\int_{\mathbb S^2}uh_0\,d\sigma\right|
 \leq 2C\|f\|_{L^p(\mathbb S^2)}
          \|h\|_{L^{p'}(\mathbb S^2)}.
\]
Taking the supremum over all smooth $h$ with
$\|h\|_{L^{p'}(\mathbb S^2)}\leq1$, we obtain
$\|u\|_{L^p(\mathbb S^2)}\leq2C\|f\|_{L^p(\mathbb S^2)}$.
Combining this with \eqref{lso00} proves \eqref{estp}.

\end{itemize}

 {\bf Step 3.} If  $f\in \mathring L^{p}(\mathbb S^2)$, then there is a solution   $u\in \mathring W^{2,p}(\mathbb S^2)$ to the Poisson equation \eqref{deuf}, and such $u$ satisfies the $L^p$ estimate \eqref{estp}. First we choose a sequence of  functions $\{f_n\}\subset \mathring C^\infty(\mathbb S^2)$ such that $f_n\to f$ in the $L^p(\mathbb S^2)$ as $n\to+\infty$. For every $n\in\mathbb Z_+$, there is a solution $u_n\in\mathring C^\infty(\mathbb S^2)$ to  \eqref{deuf}   by   Step 1.
 Moreover, the following estimates hold   by   Step 2:
  \begin{equation}\label{bsp1}
      \| u_{n}\|_{W^{2,p}(\mathbb S^2)}\leq C\| f_{n}\|_{L^{p}(\mathbb S^2)},
 \end{equation}
 \begin{equation}\label{bsp2}
      \|u_{m}-u_{n}\|_{W^{2,p}(\mathbb S^2)}\leq C\|f_{m}-f_{n}\|_{L^{p}(\mathbb S^2)}.
 \end{equation}
From \eqref{bsp2}, $\{u_n\}$ is a Cauchy sequence in $\mathring W^{2,p}(\mathbb S^2),$ and thus $\{u_n\}$ has  a limit  $u$ in $\mathring W^{2,p}(\mathbb S^2).$
  It is clear that $u$ solves \eqref{deuf}. Moreover, passing to the limit $n\to+\infty$ in \eqref{bsp1} yields the $L^p$ estimate \eqref{estp}.

 {\bf Step 4.} Finally we prove uniqueness.
Suppose that \eqref{deuf} has two solutions $u_1, u_2\in\mathring W^{2,p}(\mathbb S^2)$. Then
$ \Delta (u_1-u_2)=0$. By standard elliptic regularity theory (cf.  \cite[Section 6]{GT}),  $u_1-u_2$ is smooth. Therefore we can apply the conclusion in Step 1 to obtain
 $u_1\equiv u_2$.
\end{proof}

For $f\in\mathring L^p(\mathbb S^2),$ denote by $\mathcal Gf$ the unique solution to \eqref{deuf} in $\mathring W^{2,p}(\mathbb S^2)$.
The following lemma provides some basic properties of the operator $\mathcal G$.
\begin{lemma}\label{pg0}
Let $1<p<+\infty.$ Then
\begin{itemize}
    \item [(i)] $\mathcal G$ is a bounded linear operator from $\mathring L^{p}(\mathbb S^{2})$ onto $\mathring W^{2,p}(\mathbb S^{2})$;
    \item[(ii)] $\mathcal G$ is a compact linear operator from $\mathring L^{p}(\mathbb S^{2})$ to $\mathring L^{r}(\mathbb S^2)$ for any $1\leq r\leq +\infty$;
    \item[(iii)]$\mathcal G$ is symmetric, i.e.,
    \begin{equation}\label{symc}
\int_{\mathbb S^{2}}f\mathcal Gg d\sigma=\int_{\mathbb S^{2}}g\mathcal Gf d\sigma,\quad\forall\, f,g\in\mathring{L}^{p}(\mathbb S^{2});
\end{equation}
\item[(iv)]$\mathcal G$ is positive-definite, i.e.,
\begin{equation}\label{podef}
\int_{\mathbb S^{2}}f\mathcal Gf d\sigma \geq 0, \quad\forall\,f\in\mathring{L}^{p}(\mathbb S^{2}),
\end{equation}
and the  equality holds  if and only if $f\equiv 0.$
\end{itemize}
\end{lemma}
 \begin{proof}
 Item (i) is a direct consequence of the $L^p$ estimate in Lemma~\ref{lpt}. Item (ii) follows from item (i) and standard Sobolev embeddings (cf. \cite[Section 3.3]{Hebey}). Items (iii) and (iv) follow  from integration by parts.
 \end{proof}

\subsection{A Poincar\'e-type inequality}\label{sec32}
Suppose that $f\in\mathring L^2(\mathbb S^2)$ has the Fourier expansion
 \begin{equation}\label{ugi}
 f=\sum_{j\geq 1}\sum_{|m|\le j}c_j^mY_j^m, \quad c_j^m:=\int_{\mathbb S^2}f\overline{ Y_j^m}d\sigma.
 \end{equation}
Then the orthogonal projection $\mathbb P_j$ of $f$ onto $\mathbb E_{\leq j}:=\bigoplus_{i=1}^j\mathbb E_i$ with respect to the $L^2$ inner product  is given by
\begin{equation}\label{potrj}
\mathbb P_jf=\sum_{i\leq j}\sum_{|m|\le i}c_i^mY_i^m,
\end{equation}
and the orthogonal projection $\mathbb P^\perp_j$ of $f$ onto $\mathbb E_{\leq j}^\perp$, the orthogonal complement of $\mathbb E_{\leq j}$ in $\mathring L^2(\mathbb S^2)$,  is given by
\begin{equation}\label{potrj6}
\mathbb P^\perp_jf=f-\mathbb P_jf.
\end{equation}

\begin{lemma}\label{poin}
The following  Poincar\'e-type inequality holds:
\begin{equation}\label{eei}
\int_{\mathbb S^2} \mathbb P^\perp_jf {\mathcal G \mathbb P^\perp_j f}d\sigma\leq \frac{1}{(j+1)(j+2)}\int_{\mathbb S^2}|\mathbb P^\perp_j f|^2d\sigma,\quad \forall\, f\in\mathring L^2(\mathbb S^2).
\end{equation}
In particular,
\begin{equation}\label{p1}
\int_{\mathbb S^2} \mathbb P^\perp_1f {\mathcal G \mathbb P^\perp_1 f}d\sigma\leq
\frac{1}{6}\int_{\mathbb S^2}|\mathbb P^\perp_1 f|^2d\sigma,\quad \forall\, f\in\mathring L^2(\mathbb S^2).
\end{equation}
\end{lemma}
\begin{proof}
Suppose that $f$ has the Fourier expansion \eqref{ugi}. Then
\[\mathbb P^\perp_j f=\sum_{i\geq j+1}\sum_{|m|\le i}c_i^mY_i^m,\quad \mathcal G\mathbb P^\perp_j f=\sum_{i\geq  j+1}\sum_{|m|\le i}\frac{c_i^m}{i(i+1)}Y_i^m.\]
Therefore
\begin{equation}\label{s2pjp}
\begin{split}
\int_{\mathbb S^2} \mathbb P^\perp_jf {\mathcal G \mathbb P^\perp_j f}d\sigma&=\sum_{i\geq j+1}\sum_{|m|\le i}\frac{|c_i^m|^2}{i(i+1)}\\
&\leq \frac{1}{(j+1)(j+2)}\sum_{i\geq j+1}\sum_{|m|\le i} |c_i^m|^2\\
&=\frac{1}{(j+1)(j+2)}\int_{\mathbb S^2}|\mathbb P^\perp_j f|^2d\sigma.
\end{split}
\end{equation}
\end{proof}

\begin{remark}
It is clear that equality in \eqref{s2pjp} holds precisely when $c_i^m=0$ for all $i\geq j+2$ and $|m|\leq i$, which is equivalent to $f\in\mathbb E_{\leq j+1}.$
In other words,  equality \eqref{eei} occurs if and only if $f\in\mathbb E_{\leq j+1}.$
 \end{remark}

\subsection{Properties of rearrangements}\label{sec33}

 In this subsection, we provide some geometric and functional properties of the set of rearrangements of a given function. The results are just stated  for the convenience of later  applications, not in their greatest generality.

\begin{lemma}\label{lemlg}
Let $1<p<+\infty.$ Let $\mathcal R$ be the set of rearrangements of a  function in  $L^p(\mathbb S^2)$. Denote by  $\overline{\mathcal R^{w}}$ the weak closure of $\mathcal R$ in $L^p(\mathbb S^2).$    Let $g\in L^{p/(p-1)}(\mathbb S^2).$  Define  $L_{g}: L^{p}(\mathbb S^{2})\to\mathbb R$ as follows:
\begin{equation}\label{lfud}
L_{g}(f)=\int_{\mathbb S^{2}}fgd\sigma,\quad f\in  L^{p}(\mathbb S^{2}).
\end{equation}
Suppose that $\tilde f\in\mathcal R$ satisfies
\[\tilde f(\mathbf x)=  \mathfrak g(g(\mathbf x)) \quad\mbox{\rm  a.e.  $\mathbf x\in \mathbb S^2$}\]
for some increasing function $\mathfrak g:\mathbb R\to\mathbb R\cup\{\pm\infty\}.$ Then $\tilde f$ is the unique maximizer of  $L_{g}$ relative to $\overline{\mathcal R^{w}}$.

\end{lemma}

\begin{proof}
See \cite[Theorem 3]{BMA}.

\end{proof}

\begin{lemma}\label{recvx}
Let $1< p<+\infty$. Let $\mathcal R$ be the set of rearrangements of a function in  $L^p(\mathbb S^2)$. Denote by  $\overline{\mathcal R^{w}}$ the weak closure of $\mathcal R$ in $L^p(\mathbb S^2).$ Then
\begin{itemize}
    \item [(i)] $\overline{\mathcal R^{w}}$ is convex, and thus is the closed convex hull of $\mathcal R$;
    \item[(ii)] $\mathcal R$ is exactly the set of extreme points of  $\overline{\mathcal R^{w}}$.
    \end{itemize}
\end{lemma}
\begin{proof}
See \cite[Lemmas 2.2 and 2.3]{BHP}.
\end{proof}

The following  lemma is a corollary of  Lemma~\ref{recvx}.

\begin{lemma}\label{rl3}
Let $1<p<+\infty$, and $\mathcal R$ be the set of rearrangements of a  function  in $L^p(\mathbb S^2)$. Then $\mathcal R$ is strongly closed in $L^p(\mathbb S^2).$

\end{lemma}

\begin{proof}
Define
\[\mathcal B =\{g\in L^p(\mathbb S^2)\mid \|g\|_{L^p(\mathbb S^2)}\leq \|f\|_{L^p(\mathbb S^2)},\,\,f\in\mathcal R\},\]
\[\mathcal S =\{g\in L^p(\mathbb S^2)\mid \|g\|_{L^p(\mathbb S^2)}=\|f\|_{L^p(\mathbb S^2)},\,\,f\in\mathcal R\}. \]
Note that $\mathcal B$ and $\mathcal S$ are well-defined since  the $L^p$ norm of $f$ does not depend on the choice of $f$ in $\mathcal R$. It is clear that
\begin{equation}\label{bsbs}
 \overline{\mathcal R^{w}}\subset\mathcal  B.
\end{equation}
Let $\{f_n\}\subset\mathcal R$ be a sequence such that $f_n$ converges  to some $\tilde f\in   \overline{\mathcal R^{w}}$ strongly in $L^p(\mathbb S^2)$ as $n\to+\infty.$ Then it is obvious that  $\tilde f\in\mathcal  S$. Therefore  $\tilde f$ must be an extreme point of $\mathcal B$ since $L^p(\mathbb S^2)$ is a strictly convex space when $1<p<+\infty$. Taking into account \eqref{bsbs}, we deduce that $\tilde f$ is also an extreme point of  $\overline{\mathcal R^{w}}$. Hence
 $\tilde f\in\mathcal R$ by  Lemma~\ref{recvx}(ii).
\end{proof}

The following two lemmas are standard and will be frequently used in the sequel.
\begin{lemma}\label{ucvx}
Let $1<p<\infty$. Let $\{f_n\}$ be a sequence in $L^p(\mathbb S^2)$ such that $f_n$ converges to some $f$ weakly in $L^p(\mathbb S^2)$, and
\[\lim_{n\to+\infty}\|f_n\|_{L^p(\mathbb S^2)}=\|f\|_{L^p(\mathbb S^2)}.\]
Then $f_n$ converges to $f$ strongly in $L^p(\mathbb S^2)$.
\end{lemma}
\begin{proof}
See \cite[Section 3.7]{Bre}.
\end{proof}

\begin{lemma}\label{wny}
 Let $1<p<+\infty$. Let $\mathcal R$ be the set of rearrangements of some function in $L^\infty(\mathbb S^2).$ Let $\{f_n\}\subset \mathcal R$ and $f\in\mathcal R$. Then $f_n\to f$ in $L^2(\mathbb S^2)$ if and only if $f_n\to f$ in $L^p(\mathbb S^2).$
\end{lemma}
\begin{proof}
Since $\{f_n\}$ is uniformly bounded in $L^\infty(\mathbb S^2)$, it follows that $f_n\rightharpoonup f$ in $L^2(\mathbb S^2)$ if and only if $f_n\rightharpoonup f$ in $L^p(\mathbb S^2)$. We may then apply Lemma~\ref{ucvx}.
\end{proof}

 \section{Stability criterion}\label{sec4}
The aim of this section is to prove a general stability criterion for subsequent use. The proof relies on the conservation law \eqref{cl2} and the transport nature of the vorticity equation \eqref{ave}.
This section is mostly inspired by Burton \cite{BAR}. See also \cite[Theorem 3.2]{Wang}.

Throughout this section, let $1<p<+\infty$ be fixed. Given
$\mathcal A\subset L^p(\mathbb S^2)$ and $\mathcal B\subset\mathcal A$,
we say that $\mathcal B$ is isolated in $\mathcal A$ (with respect to
the $L^p$ norm) if
\begin{equation}\label{deisd}
\inf_{\substack{f\in\mathcal B,\,g\in\mathcal A\setminus\mathcal B}}
\|f-g\|_{L^p(\mathbb S^2)}>0,
\end{equation}
where we adopt the convention that
\[
\inf_{\substack{f\in\mathcal B,\, g\in\mathcal A\setminus\mathcal B}}
\|f-g\|_{L^p(\mathbb S^2)}
=+\infty
\quad\text{if } \mathcal B=\varnothing
\text{ or } \mathcal B=\mathcal A.
\]
In particular, $\mathcal A$ is isolated in itself under this convention.
If $\mathcal B$ is isolated in $\mathcal A$, then both $\mathcal B$ and
$\mathcal A\setminus\mathcal B$ are relatively open and relatively
closed in $\mathcal A$.

\begin{theorem}[Stability criterion]\label{stac}
Let $\mathcal R$ be the set of rearrangements of some function in $\mathring L^p(\mathbb S^2)$ (cf. \eqref{uu01}). Let $\mathcal E: \mathring L^p(\mathbb S^2)\to\mathbb R$ be a functional satisfying the following conditions:
 \begin{itemize}
     \item[(i)] $\mathcal E$ is locally uniformly continuous in $\mathring L^p(\mathbb S^2)$, i.e., for any $r>0$ and any $\varepsilon>0$, there exists $\delta>0$ such that for any $f_1, f_2\in\mathring L^p(\mathbb S^2)$ with
   $\|f_1\|_{L^p(\mathbb S^2)}\leq r,$ $\|f_2\|_{L^p(\mathbb S^2)}\leq r,$
   \[\|f_1- f_2\|_{L^p(\mathbb S^2)}<\delta\quad\Longrightarrow\quad \left|\mathcal E(f_1)-\mathcal E (f_2)\right|<\varepsilon.\]
     In particular, $\mathcal E$ is continuous in $\mathring L^p(\mathbb S^2).$
     \item[(ii)]  $\mathcal E$ is a flow invariant of the vorticity equation \eqref{ave}, i.e., for any smooth solution $\zeta$ of the vorticity equation \eqref{ave}, it holds that
     \[\mathcal E(\zeta_t)=\mathcal E(\zeta_0),\quad\forall\,t>0.\]
 \end{itemize}
Let $\mathcal M$ be a nonempty subset of $\mathcal R$  satisfying the following assumptions:
  \begin{itemize}
     \item [(a)] There exists $M\in\mathbb R$ such that
     \[\mathcal M=\{f\in\mathcal R\mid \mathcal E(f)=M\}.\]
     \item[(b)]    For any   sequence $\{f_n\}\subset\mathcal R$ such that
     \[ \lim_{n\to+\infty}\mathcal E(f_n)=M,\]
     there exist a subsequence $\{f_{n_j}\}$ of $\{f_n\}$ and some $\tilde f\in\mathcal M$ such that
     \[\lim_{j\to+\infty}\|f_{n_j}-\tilde f\|_{L^p(\mathbb S^2)}=0.\]
 \end{itemize}
 Let $\mathcal N\subset\mathcal M$ be nonempty. If $\mathcal N$ is isolated in $\mathcal M$ (cf. \eqref{deisd}), then $\mathcal N$ is a stable set of the vorticity equation \eqref{ave} in the $L^p$ norm of the absolute vorticity, i.e.,  for  any $\varepsilon>0,$ there exists $\delta>0$ such that  for any smooth solution $\zeta$ of the vorticity equation \eqref{ave},
 \begin{equation}\label{impp1}
 \min_{f\in\mathcal N}\|f-\zeta_0\|_{L^p(\mathbb S^2)}<\delta\quad\Longrightarrow\quad \min_{f\in\mathcal N}\|f-\zeta_t\|_{L^p(\mathbb S^2)}<\varepsilon,\,\,\forall\,t>0.
 \end{equation}
\end{theorem}

\begin{remark}
 Theorem~\ref{stac} only requires $\mathcal M$ to be a set of $L^{p}$ functions, which means that it can  be applied to nonsmooth solutions. See Theorems~\ref{ast1} and \ref{thmzonal} in Section~\ref{sec5}.
\end{remark}

\begin{remark}\label{crek}
By the assumptions (a) and (b), one can easily check that $\mathcal M$ is compact in $L^p(\mathbb S^2)$. As a consequence, if   $\mathcal N$ is isolated in $\mathcal M$, then   $\mathcal N$ and $\mathcal M\setminus \mathcal N$ are also compact in $L^p(\mathbb S^2)$.
\end{remark}

\begin{remark}
In subsequent applications,  $\mathcal M$ is usually chosen as the set of maximizers or minimizers of $\mathcal E$ relative to $\mathcal R,$ or equivalently,
\[M=\max_{\mathcal  R}\mathcal E\quad\mbox{or}\quad M=\min_{\mathcal  R}\mathcal E.\]
\end{remark}

\begin{proof}

First we show that $\mathcal M$ is a  stable set. It suffices to show that for any  sequence of   solutions $\{\zeta^n\}$ of the vorticity equation \eqref{ave} and any sequence $\{t_n\}\subset\mathbb R_+$, if
 \begin{equation}\label{s34}
 \lim_{n\to+\infty}\|\zeta_0^n-f\|_{L^p(\mathbb S^2)}=0
 \end{equation}
for some $f\in\mathcal M,$
then there exists some subsequence $\{\zeta^{n_j}_{t_{n_j}}\}$ of $\{\zeta^{n}_{t_{n}}\}$ such that
\begin{equation}\label{s342}
  \lim_{j\to+\infty}\|\zeta_{t_{n_j}}^{n_j}-\tilde f\|_{L^p(\mathbb S^2)}= 0
 \end{equation}
 for some $\tilde f\in\mathcal M.$
If the  perturbations belong to $\mathcal R,$ i.e., $\{\zeta_t^n\}\subset\mathcal R$  for any $t\geq0$ (which only requires $\{\zeta_0^n\}\subset\mathcal R$ by \eqref{cl2}),   then the proof is straightforward:
  \begin{align*}
  &\lim_{n\to+\infty}\|\zeta^n_0-f\|_{L^p(\mathbb S^2)}=0\\
   \Longrightarrow& \lim_{n\to+\infty} \mathcal E(\zeta^n_{0})=M\quad(\mbox{by  (i) and (a)})\\
     \Longrightarrow&   \lim_{n\to+\infty}\mathcal E(\zeta^n_{t_n})=M\quad(\mbox{by (ii)})\\
   \Longrightarrow & \lim_{j\to+\infty}\|\zeta^{n_j}_{t_{n_j}}-\tilde f\|_{L^p(\mathbb S^2)}=0\quad \mbox{for  some subsequence $\{\zeta^{n_j}_{t_{n_j}}\}$ and some $\tilde f\in\mathcal M$}\quad(\mbox{by (b)}).
  \end{align*}
To deal with general perturbations,
 we follow Burton's idea in \cite{BAR} (see also \cite{BR,Cap,WZ,WGu2}). Let $\mathbf v^n=J\nabla \mathcal G( 2\omega\sin\theta-\zeta^n)$ be the velocity field related to $\zeta^n$, and let $\Phi^{n}:\mathbb S^2\times[0,+\infty)\to\mathbb S^2$ be the corresponding flow map,  i.e.,
\begin{equation}\label{flowmap}
\begin{cases}
\frac{d\Phi^{n}(\mathbf x,t)}{dt}=\mathbf v^n( \Phi^{n}(\mathbf x,t),t),&t>0,\\
\Phi^{n}(\mathbf x,0)=\mathbf x,&\forall\,\mathbf x\in  \mathbb S^2.
\end{cases}
\end{equation}
Denote by $\Psi^{n}(t,\cdot)$ the inverse map of $\Phi^{n}(t,\cdot)$. Define a sequence of ``followers" $\{\upsilon^n\}$ as follows:
\begin{equation}\label{naf3}
\upsilon^n_t(\mathbf x)=f(\Psi^{n}(\mathbf x,t)),\quad\mathbf x\in\mathbb S^2,\,t\geq 0.
\end{equation}
Since $\mathbf v^{n}$ is  divergence-free, we see that both  $\Phi^{n}$ and  $\Psi^{n}$ are area-preserving. As a result, we have that
\begin{equation}\label{pfth8i}
\upsilon^n_t\in  {\mathcal R}_{f}=\mathcal R,\quad \forall\,t\geq 0.
\end{equation}
On the other hand, by the vorticity-transport formula (cf. \cite[Section 1.6]{MB}),
\begin{equation}\label{naf4}
\zeta^n_t(\mathbf x) =\zeta^n_0(\Psi^{n}(\mathbf x,t)),\quad \forall\,\mathbf x\in\mathbb S^2,\,t\ge 0.
\end{equation}
From \eqref{naf3} and \eqref{naf4},
we obtain
\[(\upsilon^n_t-\zeta^n_t)(\mathbf x)=(f-\zeta^n_0)(\Psi^{n}(\mathbf x,t)),\quad \forall\,\mathbf x\in\mathbb S^2,\,t\ge 0,\]
which implies that
\[\upsilon^n_t-\zeta^n_t\in\mathcal R_{f-\zeta^n_0},\quad\forall\,t\ge 0.\]
  In particular,
  \begin{equation}\label{s692}
  \lim_{n\to+\infty}\|\upsilon^n_{t_n}-\zeta^n_{t_n}\|_{L^p(\mathbb S^2)}=\lim_{n\to+\infty}\|f-\zeta^n_0 \|_{L^p(\mathbb S^2)}=0.
  \end{equation}
Now \eqref{s342} follows from the following argument:
  \begin{align*}
   &\lim_{n\to+\infty}\|\zeta^n_0-f\|_{L^p(\mathbb S^2)} =0  \\
   \Longrightarrow&  \lim_{n\to+\infty} \mathcal E(\zeta^n_{0})=M\quad(\mbox{by  (i)})\\
   \Longrightarrow  & \lim_{n\to+\infty}\mathcal E(\zeta^n_{t_n})=M\quad(\mbox{by (ii)})\\
 \Longrightarrow & \lim_{n\to+\infty}\mathcal E(\upsilon^n_{t_n})=M  \quad (\mbox{by \eqref{s692} and (i)})\\
    \Longrightarrow &\lim_{j\to+\infty}\|\upsilon^{n_j}_{t_{n_j}}-\tilde f\|_{L^p(\mathbb S^2)}=0\,\,\mbox{ for some $\{\upsilon^{n_j}_{t_{n_j}}\}\subset\{\upsilon^n_{t_n}\}$  and   $\tilde f\in\mathcal M$}\quad (\mbox{by (b) and \eqref{pfth8i}})\\
 \Longrightarrow &\lim_{j\to+\infty}\|\zeta^{n_j}_{t_{n_j}}-\tilde f\|_{L^p(\mathbb S^2)}= 0\quad (\mbox{by \eqref{s692}}).
  \end{align*}

Having proved that $\mathcal M$ is stable, we can easily show that
$\mathcal N$ is also stable by using the time continuity of solutions. Suppose that $\mathcal N\subsetneqq \mathcal M$. Then both $\mathcal N$ and $\mathcal M\setminus \mathcal N$ are nonempty and compact in $L^p(\mathbb S^2)$   (cf. Remark~\ref{crek}).
Denote
\begin{equation}\label{pds}
d:=\min_{f\in\mathcal N,\, g\in\mathcal M\setminus\mathcal N}\|f-g\|_{L^p(\mathbb S^2)}>0.
\end{equation}
 Fix $0<\varepsilon< {d}/{2}.$
  Since  $\mathcal M$ is stable, there exists some $\delta>0$  such that for any smooth solution $\zeta$  of the vorticity equation \eqref{ave}, if
  \begin{equation}\label{bu21}
 \min_{f\in\mathcal N}\|\zeta_{0}-f\|_{L^{p}(\mathbb S^{2})}<\delta,
 \end{equation}
 then
\begin{equation}\label{bu22}
 \min_{f\in\mathcal M}\|\zeta_{t}-f\|_{L^{p}(\mathbb S^{2})}<\varepsilon,\quad\forall\,t>0.
\end{equation}
Without loss of generality, we assume that
$
\delta< {d}/{2}.
$
Denote
\[\rho_{\mathcal M}(t)= \min_{f\in\mathcal M}\|\zeta_{t}-f\|_{L^{p}(\mathbb S^{2})},\,\, \rho_{\mathcal N}(t)= \min_{f\in\mathcal N}\|\zeta_{t}-f\|_{L^{p}(\mathbb S^{2})},\,\, \rho_{\mathcal M\setminus \mathcal N}(t)= \min_{f\in\mathcal M\setminus\mathcal N}\|\zeta_{t}-f\|_{L^{p}(\mathbb S^{2})}.\]
Then it is easy to check that
\begin{equation}\label{bu91}
\rho_{\mathcal M}(t)=\min\{\rho_{\mathcal N}(t), \rho_{\mathcal M\setminus \mathcal N}(t)\},\quad \forall\,t\geq 0,
\end{equation}
\begin{equation}\label{bu919}
 \rho_{\mathcal N}(t)+\rho_{\mathcal M\setminus \mathcal N}(t)\geq d,\quad \forall\,t\geq 0.
\end{equation}
Since  $\zeta$ is   smooth, it holds that $\zeta\in C([0,+\infty);L^p(\mathbb S^2))$, which implies that   $\rho_{\mathcal M},$ $\rho_{\mathcal N}$ and $\rho_{\mathcal M\setminus \mathcal N}$ are all continuous functions on $[0,+\infty)$.
From \eqref{bu22}, we have that
\begin{equation}\label{bu99}
\rho_{\mathcal M}(t)<\varepsilon<\frac{d}{2},\quad\forall\,t>0.
\end{equation}
It follows from \eqref{bu21}   that
\begin{equation}\label{bu92}
\rho_{\mathcal N}(0)<\delta<\frac{d}{2}.
\end{equation}
Below we show that
$\rho_{\mathcal N}(t)<\varepsilon$ for any $t>0.$
Suppose by contradiction that there exists some $t_{1}>0$ such that
\begin{equation}\label{bu232}
\rho_{\mathcal N}(t_{1})\geq \varepsilon.
\end{equation}
From  \eqref{bu91}, \eqref{bu99} and \eqref{bu232}, we have  that
 \begin{equation}\label{bu24}
\rho_{\mathcal M\setminus \mathcal N}(t_{1}) <\varepsilon<\frac{d}{2},
\end{equation}
which in combination with \eqref{bu919} yields
 \begin{equation}\label{bu242}
 \rho_{\mathcal N}(t_{1})>\frac{d}{2},
 \end{equation}
 From \eqref{bu92}, \eqref{bu242} and the continuity of $\rho_{\mathcal N}$,  there exists some $t_{2}>0$ such that
   \begin{equation}\label{bu31}
\rho_{\mathcal N}(t_{2})=\frac{d}{2}.
\end{equation}
From \eqref{bu919} and \eqref{bu31}, we have that
   \begin{equation}\label{bu318}
\rho_{\mathcal M\setminus \mathcal N}(t_{2})\geq \frac{d}{2},
\end{equation}
which together with \eqref{bu91} gives
   \begin{equation}\label{bu3187}
\rho_{\mathcal M }(t_{2})=\frac{d}{2}.
\end{equation}
This contradicts \eqref{bu99}.
 To summarize, we have proved that for any $0<\varepsilon<d/2$, there exists some $\delta>0$ such that for any smooth solution $\zeta$  of the vorticity equation \eqref{ave},    if
  \begin{equation}\label{bu216}
 \min_{f\in\mathcal N}\|\zeta_{0}-f\|_{L^{p}(\mathbb S^{2})}<\delta,
 \end{equation}
 then
\begin{equation}\label{bu226}
 \min_{f\in\mathcal N}\|\zeta_{t}-f\|_{L^{p}(\mathbb S^{2})}<\varepsilon,\quad\forall\,t>0.
\end{equation}
Thus the stability of $\mathcal N$ has been verified.
\end{proof}

\begin{remark}\label{lessregular}
In Theorem~\ref{stac}, we only consider smooth perturbations to avoid some technical discussions. Checking the proof, the
 stability  may hold for less regular perturbations provided that they satisfy the conservation laws involved in the proof, belong to $C([0,+\infty);L^p(\mathbb S^2))$, and admit an area-preserving flow map
such that we can construct a ``follower'' of the perturbed solution.
   \end{remark}

\section{Warm-up: applications of stability criterion}\label{sec5}

In this section, we apply Theorem~\ref{stac} to prove the stability of some steady/traveling solutions of the Euler equation in the $L^p$ setting, including  degree-1 RH waves,  Arnold-type flows, and zonal flows with monotone absolute vorticity.
This section illustrates our strategy and serves as a warm-up for the proofs of our main results.

\subsection{Stability of degree-1 RH waves}\label{sec51}

In Section~\ref{sec24}, we have proved the  stability of the degree-1 RH orbit in the $L^2$ norm  of the absolute vorticity  based on the conservation laws \eqref{cl3} and \eqref{cl4}. Now with Theorem~\ref{stac} at hand,  we can show that the  stability actually holds in the $L^p$ norm of the absolute vorticity for any $1<p<+\infty.$
    \begin{theorem}\label{sd1rhw}
Let $1<p<+\infty$, and  let  $Y\in\mathbb E_1\setminus\{0\}$.
 \begin{itemize}
    \item [(i)]If   $\omega=0,$ then  $Y$  is stable in the $L^p$ norm of the absolute vorticity in the sense of Theorem~\ref{stac}.
    \item[(ii)] For general $\omega\in\mathbb R,$   the RH orbit $\mathcal H^+_{Y}$  is  stable in the $L^p$ norm of the absolute vorticity in the sense of Theorem~\ref{stac}. In particular, $Y$ is stable in the $L^p$ norm of the absolute vorticity if $Y$ is zonal (i.e., $Y=\beta Y_{1}^{0}$ for some $\beta\in\mathbb R$).
\end{itemize}
    \end{theorem}

 \begin{remark}
 Note that Theorem~\ref{sd1rhw}(ii) is optimal if $\omega\neq0$.
In fact, given $Y\in\mathbb E_{1},$  the solution to the vorticity equation \eqref{ave} with initial absolute vorticity $Y$ has the following expression (cf. \eqref{dg1rh}):
 \[\zeta_{t}(\varphi,\theta)=Y(\varphi+\omega t,\theta).\]
  If $\omega\neq 0,$ then  it is easy to see that
  \[\{\zeta_{t}\mid t>0\}=\mathcal H^{+}_{Y},\]
  which implies that $\mathcal H^{+}_{Y}$ is the smallest stable set containing $Y$.
 \end{remark}

  \begin{proof}[Proof of Theorem~\ref{sd1rhw}(i)]
  First we establish a suitable variational characterization for $Y$.
 Suppose that $Y$ has the form
 \[Y=aY_1^0+bY_1^1+cY_1^{-1},\quad b=-\overline c,\quad a\in\mathbb R.\]
 For  $f\in \mathring L^p(\mathbb S^2)$, define
 \begin{equation}\label{ed1}
 \mathcal E(f)=|c_1^0+  a|^2+|c_1^1+b|^2+|c_1^{-1}+  c|^2,
 \end{equation}
 where
 \[c_1^0=\int_{\mathbb S^2}f\overline{Y_1^0}d\sigma,\quad c_1^1=\int_{\mathbb S^2}f\overline{Y_1^1}d\sigma,\quad c_1^{-1}=\int_{\mathbb S^2}f\overline{Y_1^{-1}}d\sigma.\]
Note that if $f\in\mathring L^2(\mathbb S^2),$ $c_1^0, c_1^1, c_1^{-1}$ are exactly the Fourier coefficients of $f$ related to the spherical harmonics $Y_1^0, Y_1^1, Y_1^{-1}$, respectively.  By Lemma~\ref{pg0}(ii), it is easy to check that $\mathcal E$ is well-defined in $\mathring L^p(\mathbb S^2)$,   is  weakly continuous in $\mathring L^p(\mathbb S^2)$, and is uniformly continuous in any bounded subset of $\mathring L^p(\mathbb S^2)$ (i.e., $\mathcal E$ satisfies the condition (i) in Theorem~\ref{stac}). Moreover, by the conservation law \eqref{cl3}, $\mathcal E$ is a flow invariant (i.e., $\mathcal E$ satisfies the condition (ii) in Theorem~\ref{stac}).

 We claim that \begin{equation}\label{lf990}
     \mbox{$Y$ is the unique maximizer of $\mathcal E$ relative to $\mathcal R_Y$}.
 \end{equation}
 In fact, for any
 $f\in\mathcal R_Y$, it holds that $f\in\mathring L^{2}(\mathbb S^{2})$, and thus we can expand  $f$ in terms of the basis $\{Y_j^m\}_{j\ge 1, |m|\le j}$ as follows:
 \[f=\sum_{j=1}^{+\infty}\sum_{|m|\leq j}c_j^mY_j^m,\quad c_j^m=\int_{\mathbb S^2}f\overline{Y_j^m}d\sigma.\]
 Then we have that
 \begin{equation}\label{pptt}
 \begin{split}
 \mathcal E(f)&\leq 2(|c_1^0|^2+  |a|^2+|c_1^1|^2+|b|^2+|c_1^{-1}|^2+  |c|^2)\\
 &=2(|a|^2+ |b|^2 +  |c|^2)+2\left(\|f\|^2_{L^2(\mathbb S^2)}-\sum_{j\geq 2}\sum_{|m|\leq j }|c_j^m|^2\right)\\
  &=2(|a|^2+ |b|^2 +  |c|^2)+2\left(\|Y\|^2_{L^2(\mathbb S^2)}-\sum_{j\geq 2}\sum_{|m|\leq j }|c_j^m|^2\right)\quad \mbox{(by $f\in\mathcal R_Y$)}\\
    &=4(|a|^2+ |b|^2 +  |c|^2)-2 \sum_{j\geq 2}\sum_{|m|\leq j }|c_j^m|^2 \\
 &\leq 4(|a|^2+ |b|^2 +  |c|^2).
 \end{split}
 \end{equation}
 Moreover, it is easy to see that  the first inequality in \eqref{pptt} is an equality if and only if
 \[c_1^0=a,\quad c_1^1=b,\quad c_1^{-1}=c,\]
  and the
last inequality  in \eqref{pptt} is an equality if and only if
 \[c_{j}^m=0, \quad \forall\,j\geq 2,\, |m|\leq j.\]
In other words, we have proved that \[\mathcal E(f)\leq 4(|a|^2+ |b|^2 +  |c|^2),\quad\forall\,f\in\mathcal R_Y,\]
and the equality holds if and only if $f=Y$. Therefore $Y$ is the unique maximizer of $\mathcal E$ relative to $\mathcal R_Y.$

Now we prove compactness. More precisely, we will show that  for any  sequence $\{f_n\}\subset\mathcal R_Y$ such that
\begin{equation}\label{cnm1}
\lim_{n\to+\infty}\mathcal E(f_{n})=\mathcal E(Y),
\end{equation}
it holds that
 \begin{equation}\label{cnmm}
 \lim_{n\to+\infty}\|f_{n}-Y\|_{L^{p}(\mathbb S^{2})}=0.
\end{equation}
Since $\{f_n\}$ is obviously bounded in $L^p(\mathbb S^2),$ there is a subsequence, denoted by $\{f_{n_j}\}$, such that $f_{n_j}$ converges weakly to some  $\tilde f$ in $L^p(\mathbb S^2)$ as $j\to+\infty.$ It is clear that  $f_{n_j}$ also converges weakly to $\tilde f$ in $L^2(\mathbb S^2)$ as $j\to+\infty$. By the weak lower semicontinuity of the $L^2$ norm, we have that
 \begin{equation}\label{ias1}
 \|\tilde f\|_{L^2(\mathbb S^2)}\leq \|Y\|_{L^2(\mathbb S^2)}.
 \end{equation}
 By \eqref{cnm1} and the  weak continuity of $\mathcal E$ in $L^2(\mathbb S^2)$, we have that
 \begin{equation}\label{ias2}
 \mathcal E(\tilde f)=\mathcal E(Y)=4(|a|^2+ |b|^2 +  |c|^2).
  \end{equation}
Expand  $\tilde f$ as
 \[\tilde f=\sum_{j=1}^{+\infty}\sum_{|m|\leq j}\tilde c_j^mY_j^m,\quad \tilde c_j^m=\int_{\mathbb S^2}\tilde f\overline{Y_j^m}d\sigma.\]
  Similar to \eqref{pptt}, we have that
   \begin{equation}\label{pptt2}
 \begin{split}
 \mathcal E(\tilde f)&\leq 2(|\tilde c_1^0|^2+  |a|^2+|\tilde c_1^1|^2+|b|^2+|\tilde c_1^{-1}|^2+  |c|^2)\\
 &=2(|a|^2+ |b|^2 +  |c|^2)+2\left(\|\tilde f\|^2_{L^2(\mathbb S^2)}-\sum_{j\geq 2}\sum_{|m|\leq j }|\tilde c_j^m|^2\right)\\
  &\leq 2(|a|^2+ |b|^2 +  |c|^2)+2\left(\|Y\|^2_{L^2(\mathbb S^2)}-\sum_{j\geq 2}\sum_{|m|\leq j }|\tilde c_j^m|^2\right)\quad \mbox{(by \eqref{ias1})}\\
    &=4(|a|^2+ |b|^2 +  |c|^2)-2 \sum_{j\geq 2}\sum_{|m|\leq j }|\tilde c_j^m|^2 \\
 &\leq 4(|a|^2+ |b|^2 +  |c|^2).
 \end{split}
 \end{equation}
In combination with \eqref{ias2}, we deduce that  the three  inequalities in \eqref{pptt2} are in fact  equalities. In particular,
$\|\tilde f\|_{L^2(\mathbb S^2)}= \|Y\|_{L^2(\mathbb S^2)},$ and thus $\|f_{n_j}\|_{L^2(\mathbb S^2)}=\|\tilde f\|_{L^2(\mathbb S^2)}$ for any $j$ (since $\{f_{n_j}\}\subset\mathcal R_Y$).
Then by Lemma~\ref{ucvx},  $f_{n_j}$ converges strongly to $\tilde f$ in $L^2(\mathbb S^2)$ as $j\to+\infty$. Applying Lemma~\ref{rl3}, we obtain $\tilde f\in\mathcal R_Y,$ which together with   \eqref{lf990} and \eqref{ias2} yields  $\tilde f=Y$. In particular,
  $\|f_{n_j}\|_{L^p(\mathbb S^2)}=\|\tilde f\|_{L^p(\mathbb S^2)}$ for any $j$.
Applying Lemma~\ref{ucvx} again, we see that  $f_{n_j}$  converges to $\tilde f=Y$ strongly in $L^p(\mathbb S^2)$ as $j\to+\infty.$ By a standard contradiction argument, we can further show that $f_{n}$ converges to $Y$ strongly in $L^p(\mathbb S^2)$ as $n\to+\infty.$ Hence the required compactness has been verified.

Applying Theorem~\ref{stac} (with $\mathcal E$ defined by \eqref{ed1}  and $\mathcal N=\mathcal M=\{Y\}$), we deduce that $Y$ is stable as required.

\end{proof}

\begin{proof}[Proof of Theorem~\ref{sd1rhw}(ii)]
First we prove a variational characterization for $\mathcal H^+_Y.$  Suppose that $Y$ has the form
 \[Y=aY_1^0+bY_1^1+cY_1^{-1},\quad b=-\overline c,\quad a\in\mathbb R.\]
For $f\in  \mathring L^p(\mathbb S^2),$ define
 \begin{equation}\label{ed2}
 \mathcal E(f)=a c_1^0+   |c_1^{1}|^2,\quad c_1^0=\int_{\mathbb S^2}f\overline{Y_1^0}d\sigma,\quad c_1^1=\int_{\mathbb S^2}f\overline{Y_1^1}d\sigma.
 \end{equation}
  Then $\mathcal E$ is well-defined and weakly continuous in $\mathring L^p(\mathbb S^2),$ and satisfies the conditions (i) and (ii) in Theorem~\ref{stac}. It is easy to check that
    \begin{equation}\label{yeqt}
    \mathcal E(Y)= a^2+|b|^2.
    \end{equation}
   For any $f\in\mathcal R_Y\subset\mathring L^2(\mathbb S^2)$, expanding  $f$ as
 \[f=\sum_{j=1}^{+\infty}\sum_{|m|\leq j}c_j^mY_j^m,\quad c_j^m=\int_{\mathbb S^2}f\overline{Y_j^m}d\sigma,\]
 we have that
   \begin{equation}\label{9op}
   \begin{split}
   \mathcal E(f)&= a c_1^0 +\frac{1}{2} \left(\|f\|_2^2-|c_1^0|^2-\sum_{j\geq 2}\sum_{|m|\leq j}|c_j^m|^2\right)\quad(\mbox{using  $|c_1^1|=|c_1^{-1}|$})\\
   &\leq   a c_1^0 +\frac{1}{2} \left(\|f\|_2^2-|c_1^0|^2\right)\\
   &=  -\frac{1}{2}|c_1^0|^2 +a c_1^0 +\frac{1}{2}  \|Y\|_{L^{2}(\mathbb S^{2})}^2 \\
   &\leq\frac{1}{2}a^2+\frac{1}{2}\|Y\|_2^2\\ &=a^2+|b|^2\quad(\mbox{using  $ \|Y\|_{L^2(\mathbb S^2)}^2=a^2+2|b|^2$})\\
   &=\mathcal E(Y)\quad(\mbox{by \eqref{yeqt}}).
   \end{split}
   \end{equation}
   Moreover, the two inequalities in \eqref{9op} are equalities if and only if
   \[c_1^0=a,\quad c_j^m=0,\quad\forall\,j\geq 2.\]
In other words, $f$ is a maximizer of $\mathcal E$ relative to $\mathcal R_Y$ if and only if
\[f=aY_1^0+b'Y_1^1+c'Y_1^{-1},\quad b'=-\overline{c'},\quad |b'|=|b|.\]
In view of Lemma~\ref{ce1} in  Appendix~\ref{appb}, we deduce that
 $\mathcal H_Y$ is exactly the set of maximizers of $\mathcal E$ relative to $\mathcal R_Y$. Taking into account \eqref{y1eq}, we obtain  the following variational characterization for $\mathcal H^+_Y$:
\begin{equation}\label{hym1}
\mbox{ $\mathcal H^+_Y$ is exactly the set of maximizers of $\mathcal E$ relative to $\mathcal R_Y$.}
\end{equation}

Next we prove compactness, i.e.,  for any   sequence $\{f_n\}\subset\mathcal R_Y$ such that
\begin{equation}\label{ias6}
\lim_{n\to+\infty}\mathcal E(f_n)=\mathcal E(Y)=a^2+|b|^2,
\end{equation}
there is a subsequence, denoted by $\{f_{n_j}\}$, such that
  $f_{n_j} \to \tilde f$ strongly in $L^p(\mathbb S^2)$ as $j\to+\infty$ for some $\tilde f\in\mathcal H^+_Y$.
 In fact, since $\{f_n\}$ is obviously bounded in $L^p(\mathbb S^2)$, there is a subsequence $\{f_{n_j}\}$ converging to some $\tilde f$  weakly  in $L^p(\mathbb S^2),$ and thus  weakly in $L^2(\mathbb S^2)$, as $j\to+\infty$. By the weak lower semicontinuity of the $L^2$ norm and the weak continuity of $\mathcal E$ in $L^{2}(\mathbb S^{2})$, we have that
 \begin{equation}\label{ias9}
 \|\tilde f\|_{L^2(\mathbb S^2)}\leq \|Y\|_{L^2(\mathbb S^2)},\quad \mathcal E(\tilde f)=\mathcal E(Y)=a^2+|b|^2.
 \end{equation}
 Repeating the argument in \eqref{9op}, we deduce from \eqref{ias6} and \eqref{ias9} that
$\|\tilde f\|_{L^2(\mathbb S^2)} =\|Y\|_{L^2(\mathbb S^2)}.$  Hence $f_{n_j}$ converges strongly to $\tilde f$ in $L^2(\mathbb S^2)$ by Lemma~\ref{ucvx}. Applying Lemma~\ref{rl3}, we obtain $\tilde f\in\mathcal R_Y,$ which in combination with   \eqref{ias6} implies that $\tilde f$ is a maximizer of $\mathcal E$ relative to $\mathcal R_Y$. Hence $\tilde f\in\mathcal H^+_Y$ by \eqref{hym1}. Applying Lemma~\ref{ucvx} again, we see that   $f_{n_j}$ converges to $\tilde f$ strongly in $L^p(\mathbb S^2)$ as $j\to+\infty$.

Applying Theorem~\ref{stac} (with  $\mathcal E$ defined by \eqref{ed2} and $\mathcal N=\mathcal M=\mathcal H^+_Y$),
the stability of $\mathcal H^+_Y$ follows immediately.
\end{proof}

   \subsection{Arnold's first stability theorem}\label{sec52}

 Suppose that $\tilde \psi\in \mathring C^2(\mathbb S^2)$ solves the elliptic equation \eqref{semm}.
 Arnold's first stability theorem (cf. \cite{A1,A2,A3,MP}) in the classical sense asserts that
    if \[\mathfrak g\in C^1(\mathbb R),\quad \inf_{\mathbb R}\mathfrak g'>0,\]
    then the corresponding steady flow is stable in  $H^2$ norm of the stream function (or equivalently, the $L^2$ norm of the absolute vorticity).
 In this subsection, we show that this theorem can be improved, in the sense that the conditions on $\tilde\psi$ and $\mathfrak g$ can be weakened, and the stability conclusion can be strengthened.

  Our result will be stated in terms of the absolute vorticity. Recall that  $\tilde \psi$ solves \eqref{semm} if and only if the corresponding absolute vorticity $\tilde\zeta=\Delta\tilde\psi+2\omega\sin\theta$ satisfies
  \[\tilde \zeta=\mathfrak g(\omega\sin\theta-\mathcal G\tilde\zeta).\]
   \begin{theorem}\label{ast1}
         Let $1<p<+\infty.$ Let    $\tilde \zeta\in\mathring L^p(\mathbb S^2)$ satisfy
         \begin{equation}\label{lli}
         \tilde\zeta=\mathfrak g(\omega\sin\theta-\mathcal G\tilde\zeta)\quad\mbox{\rm a.e. on }\mathbb S^2,
         \end{equation}
         where $\mathfrak g:\mathbb R\to\mathbb R\cup\{\pm\infty\}$ is increasing, i.e., $\mathfrak g(s_1)\leq \mathfrak g(s_2)$ whenever $s_1\leq s_2.$  Then $\tilde{\zeta}$ is stable in the $L^p$ norm of the absolute
vorticity in the sense of Theorem~\ref{stac}.
    \end{theorem}

    \begin{remark}
    By Theorem 4(i) in \cite{CG}, if $\tilde\zeta$ satisfies \eqref{lli} with \[ \omega=0,\quad \mathfrak g\in C^1(\mathbb R),\quad \inf_{\mathbb R}\mathfrak g'>-2,\]
    then $\tilde \zeta\equiv 0.$ In this case, the conclusion of Theorem~\ref{ast1} is trivial. However, when $\omega\neq 0$, this assertion may fail. For example, when $\omega=1$ and $\mathfrak g$ is given by
    \[ \mathfrak g: s\mapsto 2s,\quad s\in\mathbb R,\]
    the equation \eqref{lli} admits a nontrivial solution $\tilde \zeta=\sin\theta$.

    \end{remark}

  \begin{proof}
First we establish a suitable variational characterization for $\tilde \zeta.$
   Define
    \begin{equation}\label{lf5}
    \mathcal E(f)=\frac{1}{2}\int_{\mathbb S^2}f\mathcal Gf d\sigma-\omega \int_{\mathbb S^2}\sin\theta f d\sigma.
    \end{equation}
  Then by Lemma~\ref{pg0}(ii), $\mathcal E$ is  well-defined and weakly continuous in $\mathring L^p(\mathbb S^2)$,    and satisfies the conditions (i)  in Theorem~\ref{stac}. Moreover, by the conservation laws \eqref{cl1} and \eqref{cl3}, $\mathcal E$ is a flow invariant of the vorticity equation \eqref{ave} (i.e., $\mathcal E$ satisfies the conditions (ii)  in Theorem~\ref{stac}).

    We claim that
    \begin{equation}\label{obar}
    \mbox{$\tilde\zeta$ is the unique minimizer of $\mathcal E$ relative to $\overline{\mathcal R ^w_{\tilde\zeta}}.$}
    \end{equation}
In fact, for $f\in \overline{\mathcal R ^w_{\tilde\zeta}} $,
  \begin{equation}\label{lf1}
  \begin{split}
  &\mathcal E(f)-\mathcal E(\tilde\zeta)  \\
=&\frac{1}{2}\int_{\mathbb S^2}\left(f\mathcal Gf-\tilde\zeta\mathcal G\tilde\zeta\right) d\sigma-\omega \int_{\mathbb S^2}\sin\theta (f-\tilde\zeta) d\sigma\\
=&\frac{1}{2}\int_{\mathbb S^2}(f-\tilde\zeta)\mathcal G (f-\tilde\zeta) d\sigma+\int_{\mathbb S^2}(f-\tilde\zeta)\mathcal G\tilde\zeta d\sigma-\omega \int_{\mathbb S^2}\sin\theta (f-\tilde\zeta) d\sigma\\
=&\frac{1}{2}\int_{\mathbb S^2}(f-\tilde\zeta)\mathcal G (f-\tilde\zeta) d\sigma+\int_{\mathbb S^2} f(\mathcal G\tilde\zeta-\omega\sin\theta ) d\sigma-\int_{\mathbb S^2} \tilde\zeta(\mathcal G\tilde\zeta-\omega\sin\theta ) d\sigma\\
 \geq& 0,
\end{split}
\end{equation}
and the equality holds if and only if $f=\tilde\zeta.$
 Here  we have used the symmetry of $\mathcal G$ (cf. Lemma~\ref{pg0}(iii)), the positive-definiteness of the operator $\mathcal G$ (cf. Lemma~\ref{pg0}(iv)), and the fact that $\tilde \zeta$ is the unique minimizer of the linear functional $L_{\mathcal G\tilde\zeta-\omega\sin\theta}$ (cf. \eqref{lfud}) relative to $\overline{\mathcal R ^w_{\tilde\zeta}}$, which is a consequence of Lemma~\ref{lemlg}.
  Hence the claim \eqref{obar} has been proved.

   Next we verify compactness. Given a sequence $\{f_n\}\subset\mathcal R_{\tilde\zeta}$ satisfying
   \begin{equation}\label{ww1}
   \lim_{n\to+\infty}\mathcal E(f_n)=\mathcal E(\tilde\zeta),
   \end{equation}
  we claim that $f_n$ converges strongly to $\tilde\zeta$ in $L^p(\mathbb S^2)$ as $n\to+\infty.$ Since $\{f_n\}$ is obviously bounded  in $L^p(\mathbb S^2)$, there is  a subsequence  $\{f_{n_j}\}$ such that $f_{n_j}$ converges weakly to some  $\tilde f$ in $L^p(\mathbb S^2)$ as $j\to+\infty.$ It is clear that $\tilde f\in \overline{\mathcal R^{w}_{\tilde \zeta}}$.  Moreover, in view of \eqref{ww1} and the weak continuity of $\mathcal E$, we have that
  \[\mathcal E(\tilde f)=\mathcal E(\tilde \zeta),\]
  which together with \eqref{obar} yields   $\tilde f=\tilde\zeta$. In particular, $\tilde f\in\mathcal R_{\tilde\zeta}.$
Applying  Lemma~\ref{ucvx}, we deduce that $f_{n_{j}}$ converges strongly to $\tilde f=\tilde\zeta$ in $L^p(\mathbb S^2)$ as $j\to+\infty.$  By a contradiction argument, we can further show that  $f_{n}$ converges to $\tilde\zeta$ strongly in $L^p(\mathbb S^2)$ as $n\to+\infty.$

Having proved the variational characterization  for $\tilde\zeta$ and the related compactness, the desired stability of $\tilde \zeta$ follows  from Theorem~\ref{stac} immediately.
  \end{proof}

Checking the proof of Theorem~\ref{ast1},  we have the following interesting corollary.
\begin{corollary}\label{cor555}
  Let $\tilde\zeta$ be as in Theorem~\ref{ast1}.
  \begin{itemize}
  \item[(i)] If $\omega=0$, then $\tilde\zeta\equiv 0.$
    \item[(ii)] For general $\omega\in\mathbb R$,  $\tilde\zeta$ must be zonal (i.e., $\tilde\zeta$ depends only on   $\theta$).
  \end{itemize}
\end{corollary}
\begin{proof}
By \eqref{obar}, $\tilde\zeta$ is the unique minimizer of $\mathcal E$ relative to $ \mathcal R_{\tilde\zeta},$ where $\mathcal E$ is defined by \eqref{lf5}.

If $\omega=0$, then it is easy to see that both $\mathcal E$ and $\mathcal R_{\tilde\zeta}$ are invariant under the action of the rotation group $\mathbf S\mathbf O(3)$, i.e., for any $\mathsf g\in\mathbf S\mathbf O(3)$,
\[
\mathcal E(f\circ \mathsf g)=\mathcal E(f),
\]
\[
f\in\mathcal R_{\tilde\zeta}\,\,\Longrightarrow\,\, f\circ \mathsf g\in\mathcal R_{\tilde\zeta}.
\]
Therefore for any $\mathsf g\in\mathbf S\mathbf O(3)$, $\tilde\zeta\circ \mathsf g$ is also a minimizer of  $\mathcal E$ relative to $\mathcal R_{\tilde\zeta}.$ By uniqueness, we have that
\[\tilde\zeta=\tilde\zeta\circ \mathsf g,\quad \forall\,\mathsf g\in\mathbf S\mathbf O(3),\]
which implies that $\tilde \zeta$ is constant. On the other hand, since $\tilde\zeta$ is of zero mean, we further deduce that $\tilde\zeta\equiv 0.$

For general $\omega\in\mathbb R$,  both $\mathcal E$ and $\mathcal R_{\tilde\zeta}$  are invariant under rotations about the polar axis. Then,  by uniqueness again,  $\tilde \zeta$ must be zonal.
\end{proof}

\begin{remark}
The interested reader can compare the above corollary with Theorem  4(i) in \cite{CG}.
\end{remark}

\subsection{Arnold's second stability theorem}\label{sec53}

 Suppose that $\tilde \psi\in \mathring C^2(\mathbb S^2)$ solves \eqref{semm}.
  Arnold's second stability theorem (cf. \cite{A1,A2,A3,MP}) in the classical sense asserts that, if
  \begin{equation}\label{lf9}
   \mathfrak g\in C^1(\mathbb R),\quad -2<\inf_{\mathbb R}\mathfrak g'\leq \sup_{\mathbb R}\mathfrak g'<0,
  \end{equation}
  then the corresponding steady flow is stable in the $H^2$ norm of the stream function (or, equivalently, the $L^2$ norm of the absolute vorticity).
   Recently, Constantin and Germain (cf. \cite[Theorem 5]{CG}) showed that the  conditions \eqref{lf9} can be weakened to
     \begin{equation}\label{lf19}
   \mathfrak g\in C^1(\mathbb R),\quad -6<\inf_{\mathbb R}\mathfrak g'\leq \sup_{\mathbb R}\mathfrak g'<0.
  \end{equation}
They proved this assertion based on the energy-Casimir (EC) functional method, which involves  the following conserved quantities:
\[\int_{\mathbb S^2}\Omega\sin\theta d\sigma,\,\,\|\mathbb P_1\Omega\|_2,\,\,\int_{\mathbb S^2}|\mathbf v|^2d\sigma,\,\,\int_{\mathbb S^2}\mathfrak F(\Omega+2\omega\sin\theta)d\sigma,\]
where
\begin{equation}\label{lf12}
\mathfrak f=\mathfrak g^{-1},\quad\mathfrak F(s)=\int_0^s\mathfrak f(\tau)d\tau.
\end{equation}

 Below  we give an alternative proof of  Constantin and Germain's result in the $L^p$ setting based on Theorem~\ref{stac}. We still state our result in terms of the absolute vorticity.

   \begin{theorem}\label{asst}
    Let $1<p<+\infty$ be fixed.
     Let    $\tilde \zeta\in\mathring L^p(\mathbb S^2)$ satisfy
         \begin{equation}\label{lli2}
         \tilde\zeta=\mathfrak g(\omega\sin\theta-\mathcal G\tilde\zeta)\quad\mbox{\rm  on }\mathbb S^2,
         \end{equation}
         where $\mathfrak g:\mathbb R\to\mathbb R$  satisfies \eqref{lf19}.
          Then $\tilde\zeta$ is stable in the $L^p$ norm of the absolute vorticity in the sense of Theorem~\ref{stac}.
    \end{theorem}

    \begin{remark}
      We believe that the assumptions on $\tilde\zeta$ and $\mathfrak g$ in Theorem~\ref{asst}  can be  further relaxed by applying Wolansky and Ghil's supporting functional method \cite{WG1,WG2}. See also \cite{WGu1,WZ}.
    \end{remark}

   \begin{proof}
First notice that $\tilde \zeta\in\mathring C^1(\mathbb S^2)$ by standard elliptic regularity theory.
Define
    \begin{align*}
      \mathcal E(f)=\frac{1}{2}\int_{\mathbb S^2}f\mathcal Gf d\sigma-\omega \int_{\mathbb S^2}\sin\theta f d\sigma- \frac{1}{4}\sum_{m=-1,0,1}\left|\int_{\mathbb S^2} f\overline{Y_1^m} d\sigma\right|^2+ \frac{1}{2}  \int_{\mathbb S^2}f\mathbb P_1\tilde\zeta d\sigma.
    \end{align*}
It is easy to check that $\mathcal E$ is  well-defined and weakly continuous in $\mathring L^p(\mathbb S^2),$ and satisfies the condition  (i)  in Theorem~\ref{stac}. Moreover, if $f\in \mathring L^2(\mathbb S^2),$ then
\[\mathcal E(f)=\frac{1}{2}\int_{\mathbb S^2}f\mathcal Gf d\sigma-\omega \int_{\mathbb S^2}\sin\theta f d\sigma-\frac{1}{4}\|\mathbb P_1(f-\tilde\zeta)\|^2_{L^2(\mathbb S^2)}+
    \frac{1}{4}\|\mathbb P_1\tilde\zeta\|^2_{L^2(\mathbb S^2)}.\]
We claim that $\mathcal E$ is a flow invariant (i.e., $\mathcal E$ satisfies the condition (ii) in Theorem~\ref{stac}).  Since the first two terms in the expression of $\mathcal E$ are both flow invariants by the conservation laws \eqref{cl1} and \eqref{cl3}, it suffices to show that the third term is also a flow invariant.
Notice that for any smooth solution $\zeta$ of the vorticity equation \eqref{ave},
\[\|\mathbb P_1(\zeta_t-\tilde\zeta)\|^2_{L^2(\mathbb S^2)}=|c_1^0(t)-\tilde c_1^0|^2+|c_1^1(t)-\tilde c_1^1|^2+|c_1^{-1}(t)-\tilde c_1^{-1}|^2,\]
 where $\{c_j^m(t)\}$ and $\{\tilde c_j^m\}$ are the Fourier coefficients of $\zeta_t$ and $\tilde\zeta,$ respectively. We distinguish two cases:
 \begin{itemize}
     \item [(i)]If $\omega=0,$ then $c_1^0(t)$,  $c_1^{1}(t)$ and  $c_1^{-1}(t)$ are all flow invariants by the conservation law \eqref{cl3}, and thus $\|\mathbb P_1(\zeta_t-\tilde\zeta)\|_{L^2(\mathbb S^2)}$ is a flow invariant.
     \item [(ii)]If $\omega\neq 0$, then by Theorem 4(iii) in \cite{CG},
     \[\tilde c_1^1=\tilde c_1^{-1}=0.\]
     Hence
     \[\|\mathbb P_1(\zeta_t-\tilde\zeta)\|^2_{L^2(\mathbb S^2)}=|c_1^0(t)-\tilde c_1^0|^2+|c_1^1(t)|^2+|c_1^{-1}(t)|^2,\]
     which remains  a flow invariant  by the conservation law \eqref{cl3}.
 \end{itemize}

To proceed, we claim that
 \begin{equation}\label{lf080}
\mbox{$\tilde \zeta$ is the unique maximizer of $\mathcal E$ relative to $\mathcal R_{\tilde\zeta}$.}
\end{equation}
   Let $\mathfrak f$ and $\mathfrak F$ be given by \eqref{lf12}. Then $\tilde\zeta$ satisfies
    \begin{equation}\label{lf20}
    \omega\sin\theta-\mathcal G\tilde\zeta=\mathfrak f(\tilde \zeta).
    \end{equation}
    Moreover, in view of \eqref{lf19}, $\mathfrak F$ satisfies for any $s,s_0\in\mathbb R$:
      \begin{equation}\label{lf29}
      \begin{split}
      \mathfrak F(s)&\leq \mathfrak F(s_0)+ \mathfrak f(s_0)(s-s_0)+\frac{1}{2}\sup_{\mathbb R}\mathfrak f'(s-s_0)^2\\
      &= \mathfrak F(s_0)+ \mathfrak f(s_0)(s-s_0)+\frac{1}{2}\left(-\frac{1}{6}-\delta\right)(s-s_0)^2,\quad \delta:=-\frac{1}{6}-\sup_{\mathbb R}\mathfrak f'>0.
      \end{split}
      \end{equation}
For any $f\in\mathcal R_{\tilde\zeta}$,  we estimate $\mathcal E(\tilde\zeta)-\mathcal E(f)$  as follows:
\begin{align*}
  &\mathcal E(\tilde\zeta)-\mathcal E(f)\\
  =&\mathcal E(\tilde\zeta)-\mathcal E(f)+\int_{\mathbb S^2}\mathfrak F(\tilde\zeta)d\sigma-\int_{\mathbb S^2}\mathfrak F(f)d\sigma\quad\mbox{(by $f\in\mathcal R_{\tilde\zeta}$)}\\
  =&\frac{1}{2}\int_{\mathbb S^2}\left(\tilde\zeta\mathcal G\tilde\zeta -f\mathcal G f\right) d\sigma+\int_{\mathbb S^2}\left(\mathfrak F(\tilde\zeta)-\mathfrak F(f)\right)d\sigma-\omega\int_{\mathbb S^2}\sin\theta (\tilde\zeta-f) d\sigma+\frac{1}{4}\|\mathbb P_1(f-\tilde\zeta)\|^2_{L^2(\mathbb S^2)}\\
 =&- \frac{1}{2}\int_{\mathbb S^2}(\tilde\zeta-f)\mathcal G(\tilde\zeta-f) d\sigma+ \int_{\mathbb S^2}\mathcal G\tilde\zeta (\tilde\zeta-f) d\sigma+\int_{\mathbb S^2}\left(\mathfrak F(\tilde\zeta)-\mathfrak F(f)\right)d\sigma\\
 &-\omega\int_{\mathbb S^2}\sin\theta (\tilde\zeta-f) d\sigma+\frac{1}{4}\|\mathbb P_1(f-\tilde\zeta)\|^2_{L^2(\mathbb S^2)}\quad\mbox{(by symmetry of $\mathcal G$)}\\
\geq &- \frac{1}{2}\int_{\mathbb S^2}(\tilde\zeta-f)\mathcal G(\tilde\zeta-f) d\sigma+ \int_{\mathbb S^2}\mathcal G\tilde\zeta (\tilde\zeta-f) d\sigma+\int_{\mathbb S^2}\left[\mathfrak f(\tilde\zeta)(\tilde\zeta-f)+\frac{1}{2}\left(\frac{1}{6}+\delta\right)(f-\tilde\zeta)^2\right]d\sigma \\
 &
 -\omega\int_{\mathbb S^2}\sin\theta (\tilde\zeta-f) d\sigma+\frac{1}{4}\|\mathbb P_1(f-\tilde\zeta)\|^2_{L^2(\mathbb S^2)} \quad\mbox{(by \eqref{lf29})}\\
=& - \frac{1}{2}\int_{\mathbb S^2}(\tilde\zeta-f)\mathcal G(\tilde\zeta-f) d\sigma+\frac{1}{2}\left(\frac{1}{6}+\delta\right)\int_{\mathbb S^2}(f-\tilde\zeta)^2d\sigma +\frac{1}{4}\|\mathbb P_1(f-\tilde\zeta)\|^2_{L^2(\mathbb S^2)} \,\,\mbox{(by \eqref{lf20})}\\
 =&- \frac{1}{2}\int_{\mathbb S^2}\mathbb P_1(\tilde\zeta-f)\mathcal G\mathbb P_1(\tilde\zeta-f) d\sigma- \frac{1}{2}\int_{\mathbb S^2}\mathbb P^\perp_1(\tilde\zeta-f)\mathcal G\mathbb P^\perp_1(\tilde\zeta-f) d\sigma+\frac{1}{12}\|f-\tilde\zeta\|^2_{L^2(\mathbb S^2)}\\
 & +\frac{1}{2}\delta\|f-\tilde\zeta\|^2_{L^2(\mathbb S^2)}+\frac{1}{4}\|\mathbb P_1(f-\tilde\zeta)\|^2_{L^2(\mathbb S^2)}\\
 =&  \frac{1}{12} \|\mathbb P^\perp_1(\tilde \zeta-f)\|^2_{L^2(\mathbb S^2)}- \frac{1}{2}\int_{\mathbb S^2}\mathbb P^\perp_1(\tilde\zeta-f)\mathcal G\mathbb P^\perp_1(\tilde\zeta-f) d\sigma+\frac{1}{12} \|\mathbb P_1(\tilde \zeta-f)\|^2_{L^2(\mathbb S^2)}\\
   &+\frac{1}{2}\delta\int_{\mathbb S^2}(f-\tilde\zeta)^2d\sigma \quad \left(\mbox{by $\mathcal G\mathbb P_1(\tilde\zeta-f)=\frac{1}{2}\mathbb P_1(\tilde\zeta-f)$}\right)\\
 \geq& \frac{1}{12} \|\mathbb P_1(\tilde \zeta-f)\|^2_{L^2(\mathbb S^2)}+\frac{1}{2}\delta\|f-\tilde\zeta\|^2_{L^2(\mathbb S^2)} \quad (\mbox{by \eqref{p1}})\\
 \geq &\frac{1}{2}\delta\|f-\tilde\zeta\|^2_{L^2(\mathbb S^2)}.
\end{align*}
To summarize, we have proved for any $f\in\mathcal R_{\tilde\zeta}$
\begin{equation}\label{decl}
 \mathcal E(\tilde\zeta)-\mathcal E(f)\geq \frac{1}{2}\delta\|f-\tilde\zeta\|^2_{L^2(\mathbb S^2)}\geq 0,
\end{equation}
and $\mathcal E(f)=\mathcal E(\tilde\zeta)$ if and only if $f=\tilde\zeta.$
Hence the claim \eqref{lf080} has been verified.

To apply Theorem~\ref{stac} to prove stability, it remains to verify compactness. Fix a sequence  $\{f_n\}\subset\mathcal R_{\tilde\zeta}$  such that
\begin{equation}\label{lfn}
\lim_{n\to+\infty}\mathcal E(f_n)=\mathcal E(\tilde\zeta).
\end{equation}
 By \eqref{decl}, it is obvious that $f_n$ converges strongly to $\tilde\zeta$ in $L^2(\mathbb S^2)$ as $n\to+\infty.$ Taking into account Lemma~\ref{wny}, we obtain strong convergence in $L^p(\mathbb S^2).$

 \end{proof}

Based on the variational characterization \eqref{lf080},  we can prove the following rigidity result, which has been obtained by Constantin and Germain using a more straightforward approach (cf. \cite[Theorem  4(ii)(iv)]{CG}).

\begin{corollary}
  Let $\tilde\zeta$ be as in Theorem~\ref{asst}.
  \begin{itemize}
  \item[(i)] If $\omega\neq 0$, then $\tilde \zeta$ is zonal (i.e., $\tilde\zeta$ depends only on   $\theta$).
    \item[(ii)] For general $\omega\in\mathbb R$,  $\tilde\zeta$ is zonal up to a transformation in $\mathbf S\mathbf O(3)$.
  \end{itemize}
\end{corollary}
  \begin{proof}
  In Cartesian coordinates, $\{x_1,x_2,x_3\}$ is an orthogonal basis of $\mathbb E_1,$ and thus we may write
$\mathbb P_1\tilde\zeta=\mathbf p\cdot\mathbf x$ for some $\mathbf p\in\mathbb R^3$.
Then
  $\mathcal E$ can be written as
\begin{align*}
 \mathcal E(f)=&\frac{1}{2}\int_{\mathbb S^2}f\mathcal Gf d\sigma -\omega \int_{\mathbb S^2}x_3 f d\sigma-\frac{1}{4}\| \mathbb P_1 f  \|^2_{L^2(\mathbb S^2)}+\frac{1}{2}\mathbf p\cdot\int_{\mathbb S^2}\mathbf xf d\sigma\\
 =&\frac{1}{2}\int_{\mathbb S^2}f\mathcal Gf d\sigma  -\frac{1}{4}\| \mathbb P_1 f  \|^2_{L^2(\mathbb S^2)}+ \mathbf q\cdot\int_{\mathbb S^2}\mathbf xf d\sigma,\quad \mathbf q:=\frac{1}{2}\mathbf p-\omega\mathbf e_3.
\end{align*}
If $\mathbf q=0$, then both $\mathcal E$ and
$\mathcal R_{\tilde\zeta}$ are invariant under all rotations
in $\mathbf S\mathbf O(3)$. By the uniqueness assertion
\eqref{lf080}, $\tilde\zeta$ is invariant under all such
rotations. Hence it is constant, and its zero mean implies
$\tilde\zeta=0$. Consequently $\mathbf p=0$ and $\omega=0$.
In this case assertion (ii) is immediate.

Suppose now that $\mathbf q\ne0$. Both $\mathcal E$ and
$\mathcal R_{\tilde\zeta}$ are invariant under rotations
about the axis $\mathbb R\mathbf q$. Again by
\eqref{lf080}, $\tilde\zeta$ is invariant under these
rotations, which proves assertion (ii).
Moreover,
\[
\int_{\mathbb S^2}\mathbf x\tilde\zeta\,d\sigma
=\int_{\mathbb S^2}\mathbf x
   (\mathbf p\cdot\mathbf x)\,d\sigma
=\frac{4\pi}{3}\mathbf p.
\]
The rotational invariance therefore implies
$\mathbf p=\lambda\mathbf q$ for some $\lambda\in\mathbb R$.
Since
\[
 \omega\mathbf e_3
 =\frac12\mathbf p-\mathbf q
 =\left(\frac\lambda2-1\right)\mathbf q,
\]
if $\omega\ne0$, then $\mathbf q$ is parallel to
$\mathbf e_3$. Thus $\tilde\zeta$ is invariant under
rotations about the polar axis, proving assertion (i).

   \end{proof}

\subsection{Stability of zonal flows with monotone absolute vorticity}

In this subsection, we show that any  zonal flow with monotone absolute vorticity is nonlinearly stable in $L^{p}$ norm of the absolute vorticity, which extends two previous results by Caprino-Marchioro \cite[Theorem 1.1]{Cap} and Taylor \cite[Theorem 4.1.1]{T}.

\begin{theorem}\label{thmzonal}
 Let $1<p<+\infty$. Suppose that $\tilde\zeta\in \mathring L^{p}(\mathbb S^{2})$ depends only on $x_{3}$ and is monotone in $x_{3}$. Then $\tilde\zeta$ is stable in the $L^p$ norm of the absolute vorticity in the sense of Theorem~\ref{stac}.
\end{theorem}

\begin{proof}
Without loss of generality, we assume that $\tilde\zeta$   is increasing, i.e.,  $\tilde\zeta(a)\leq \tilde\zeta(b)$ as long as $-1\leq a\leq b\leq 1$.
 Define
 \[\mathcal E(f)=\int_{\mathbb S^{2}}x_{3}fd\sigma.\]
 Then $\mathcal E$ satisfies the conditions (i) and (ii) in Theorem~\ref{stac}. By Lemma~\ref{lemlg},
 \begin{equation}\label{ccnns}
 \mbox{$\tilde \zeta$ is the unique maximizer of $\mathcal E$ relative to $\overline{\mathcal R^{w}_{\tilde\zeta}}$. }
 \end{equation}
  Therefore to apply Theorem~\ref{stac}, it suffices to verify compactness.
 Consider a sequence  $\{f_n\}\subset\mathcal R_{\tilde\zeta}$ satisfying
     \[ \lim_{n\to+\infty}\mathcal E(f_n)=M,\quad M:=\max_{\mathcal R_{\tilde\zeta}}\mathcal E=\max_{\overline{\mathcal R^{w}_{\tilde\zeta}}}\mathcal E.\]
     Since $\{f_n\}$ is obviously bounded in $L^{p}(\mathbb S^{2}),$ we can assume, up to a subsequence, that $f_{n}$ converges weakly to some $\hat \zeta\in\overline{\mathcal R^{w}_{\tilde\zeta}}$ in $L^{p}(\mathbb S^{2}).$  Since $\mathcal E$ is  weakly sequentially continuous in $L^{p}(\mathbb S^{2}),$ we have that
$\mathcal E(\hat \zeta)=M.$
     Hence $\hat\zeta$ is also a maximizer of $\mathcal E$ relative to $\overline{\mathcal R^{w}_{\tilde\zeta}}$, which together with \eqref{ccnns} yields $\hat\zeta=\tilde\zeta.$ In particular, $\hat\zeta\in\mathcal R_{\tilde\zeta}$.  Then strong convergence in $L^p(\mathbb S^2)$ follows from Lemma~\ref{ucvx}.

 \end{proof}

\section{Stability of degree-2 RH waves: $\alpha\neq 0$}\label{sec6}
 In this section, we give the proof of Theorem~\ref{thm1} by applying Theorem~\ref{stac}.

  Throughout this section, let $\alpha\in\mathbb R\setminus\{0\}$ and $Y\in\mathbb E_2\setminus\{0\}$, and fix $1<p<+\infty$. By \eqref{e2span}, we can assume that  $Y$ has the following form:
    \begin{equation}\label{yyd1}
     Y=a(3\sin^2\theta-1)+ b\sin (2\theta)\cos\varphi+  c\sin(2\theta)\sin\varphi
   +  d\cos^2\theta\cos(2\varphi)
   + e\cos^2\theta\sin(2\varphi),
    \end{equation}
where $a,b,c,d,e\in\mathbb R.$

  \subsection{Variational problem}\label{sec61}
As we have seen in Section~\ref{sec5}, the main challenge in applying Theorem~\ref{stac}  is to establish a suitable variational characterization for the solutions under consideration.
In this subsection, we solve a maximization problem and verify compactness of maximizing sequences.  In the next two subsections, we will show that $\mathcal H^{+}_{\alpha\sin\theta+Y}$ is an isolated set of maximizers of the maximization problem.

Consider
 \begin{equation}\label{vpn0}
   M=\sup_{f\in \mathcal R_{\alpha\sin\theta+Y}}\mathcal E(f),
 \end{equation}
 where
 \begin{equation}\label{yc1}
  \mathcal E(f)=\frac{1}{2}\int_{\mathbb S^2}f\mathcal Gf d\sigma-\frac{1}{4}
  \sum_{m=1,0,-1}\left|\int_{\mathbb S^2} f\overline{Y_1^m} d\sigma\right|^2
 + \frac{1}{6}\alpha\int_{\mathbb S^2} \sin\theta fd\sigma.
  \end{equation}
  By   Lemma~\ref{pg0}, it is easy to check that the functional $\mathcal E$ is well-defined, weakly continuous and locally uniformly continuous (cf. the condition (i) in Theorem~\ref{stac}) in $\mathring L^{p}(\mathbb S^{2})$.
    Moreover,  if $f\in\mathring L^{2}(\mathbb S^{2})$ has the expansion
\begin{equation}\label{t41}
f=\sum_{j\geq 1}\sum_{|m|\leq j}c_j^mY_j^m,\quad c_{j}^{m}=\int_{\mathbb S^{2}}f\overline{Y_{j}^{m}}d\sigma,
\end{equation}
    then  $\mathcal E(f)$ can be written as
     \begin{equation}\label{t42}
     \mathcal E(f)=\frac{1}{2}\sum_{j\geq 2}\sum_{|m|\leq j}\frac{|c_j^m|^2}{j(j+1)}+\frac{1}{6}\beta c_1^0, \quad \beta:=\sqrt{\frac{4\pi}{3}}\alpha.
     \end{equation}
Note that $\beta$ is exactly the real number such that $\alpha\sin\theta=\beta Y_1^0.$
Therefore, according to the conservation laws \eqref{cl1} and \eqref{cl3},
it is easy to see that $\mathcal E$ is a flow invariant (cf. the condition (ii) in  Theorem~\ref{stac}) of the vorticity equation \eqref{ave}.

\begin{proposition}\label{yc00}
Denote by $\mathcal M$ the set of maximizers of \eqref{vpn0}, i.e.,
  \begin{equation}\label{ycc1}
\mathcal M:=\left\{f\in\mathcal R_{\alpha\sin\theta+Y}\mid \mathcal E(f)=M\right\}.
  \end{equation}
  \begin{itemize}
    \item[(i)] It holds that \begin{equation}\label{t45b}
 M=\frac{1}{6}\beta^2+\frac{1}{12}\|Y\|_{L^2(\mathbb S^2)}^2,\quad\mathcal M=\mathcal R_{\alpha\sin\theta+Y}\cap (\alpha\sin\theta+\mathbb E_2),
\end{equation}
 where
         \[\alpha\sin\theta+\mathbb E_2:=\left\{\alpha\sin\theta+X\mid X\in\mathbb E_2\right\}.\]
         In particular, $\mathcal M$ is nonempty (since obviously $\alpha\sin\theta+Y$ is a maximizer).
    \item [(ii)] For any sequence $\{f_{n}\}\subset \mathcal R_{\alpha\sin\theta+Y}$ satisfying
$\lim_{n\to+\infty}\mathcal E(f_{n})=M,$
 there is a subsequence  of  $\{f_n\}$, denoted by $\{f_{n_{j}}\}$, such that   $f_{n_{j}}$ converges to some $\tilde f \in\mathcal M$ strongly in $L^p(\mathbb S^2)$ as $j\to+\infty.$ In particular, $\mathcal M$ is compact in $L^p(\mathbb S^2)$.
  \end{itemize}
  \end{proposition}

\begin{proof}
First we prove (i).
For any $f\in\mathcal R_{\alpha\sin\theta+Y}\subset  \mathring L^2(\mathbb S^2)$,  expand $f$ as in \eqref{t41}.  According to \eqref{t42}, we have that
   \begin{equation}\label{t43}
       \begin{split}
       \mathcal E(f)&=\frac{1}{2}\sum_{j\geq 2}\sum_{|m|\leq j}\frac{|c_j^m|^2}{j(j+1)}+\frac{1}{6}\beta c_1^0\\
       &= \frac{1}{2}\sum_{j\geq 2}\sum_{|m|\leq j}\frac{|c_j^m|^2}{6}-\frac{1}{2}\sum_{j\geq 3}\sum_{|m|\leq j}\left(\frac{1}{6}-\frac{1}{j(j+1)}\right) {|c_j^m|^2} +\frac{1}{6}\beta c_1^0\\
      & \leq \frac{1}{2}\sum_{j\geq 2}\sum_{|m|\leq j}\frac{|c_j^m|^2}{6}+\frac{1}{6}\beta c_1^0\\
       &
     =\frac{1}{12}\left(  \|f\|_{L^2(\mathbb S^2)}^2-|c_1^0|^2-|c_1^1|^2-|c_1^{-1}|^2\right) +\frac{1}{6}\beta c_1^0\\
       &\leq \frac{1}{12}\left(  \|f\|_{L^2(\mathbb S^2)}^2-|c_1^0|^2 \right)+\frac{1}{6}\beta c_1^0 \\
       &=-\frac{1}{12}|c_1^0|^2+\frac{1}{6}\beta c_1^0+\frac{1}{12}\left(\beta^2 +\| Y\|_{L^2(\mathbb S^2)}^2\right)\\
       &\leq \frac{1}{6}\beta^2+\frac{1}{12}\|Y\|_{L^2(\mathbb S^2)}^2.
       \end{split}
   \end{equation}
Moreover, it is easy to check that the first inequality in \eqref{t43} is an equality if and only if
\[c_j^m=0,\quad\forall\,j\geq 3,\,\, |m|\leq j,\]
the second inequality in \eqref{t43} is an equality if and only if
$c_1^1=c_1^{-1}=0,$
and the last inequality in \eqref{t43} is an equality if and only if
$c_1^0=\beta.$
In other words, we have proved that
 for any $f\in\mathcal R_{\alpha\sin\theta+Y}$,
  \[\mathcal E(f)\leq \frac{1}{6}\beta^2+\frac{1}{12}\|Y\|_2^2,\]
  and the equality holds if and only if $f\in\beta Y_1^0+\mathbb E_2.$
  Hence \eqref{t45b} has been verified.

  Next we prove (ii).  Fix $\{f_{n}\}\subset \mathcal R_{\alpha\sin\theta+Y}$ such that
$\lim_{n\to+\infty}\mathcal E(f_{n})=M$. It is clear that $\{f_{n}\}$ is  bounded in $\mathring L^{p}(\mathbb S^{2}),$ and thus there is a subsequence, denoted by $\{f_{n_{j}}\}$,  such that   $f_{n_{j}}$ converges to some $\tilde f $ weakly in $\mathring  L^p(\mathbb S^2)$ as $j\to+\infty.$
    Below we show that $\tilde f\in \mathcal M$, and that the convergence actually holds in the strong sense.
Since $\{f_{n_{j}}\}$ is bounded in $\mathring  L^{2}(\mathbb S^{2}),$ we  deduce that  $f_{n_{j}}$ converges to   $\tilde f $ weakly in $\mathring  L^2(\mathbb S^2)$ as $j\to+\infty.$   By the weak lower semicontinuity of the $L^2$ norm, we have that
\begin{equation}\label{t89}
\|\tilde f\|^{2}_{L^{2}(\mathbb S^{2})}\leq \liminf _{j\to+\infty}\|f_{n_{j}}\|^{2}_{L^{2}(\mathbb S^{2})}=\|\alpha\sin\theta+Y\|^{2}_{L^{2}(\mathbb S^{2})}=\beta^{2}+\|Y\|^{2}_{L^{2}(\mathbb S^{2})}.
\end{equation}
Expand   $\tilde f$ as
\begin{equation*}
\tilde f=\sum_{j\geq 1}\sum_{|m|\leq j}\tilde c_j^mY_j^m,\quad \tilde c_{j}^{m}=\int_{\mathbb S^{2}}\tilde f\overline{Y_{j}^{m}}d\sigma.
\end{equation*}
Similar to \eqref{t43}, we have that
        \begin{equation}\label{t499}
       \begin{split}
       \mathcal E(\tilde f)&=\frac{1}{2}\sum_{j\geq 2}\sum_{|m|\leq j}\frac{|\tilde c_j^m|^2}{j(j+1)}+\frac{1}{6}\beta \tilde c_1^0\\
       &= \frac{1}{2}\sum_{j\geq 2}\sum_{|m|\leq j}\frac{|\tilde c_j^m|^2}{6}-\frac{1}{2}\sum_{j\geq 3}\sum_{|m|\leq j}\left(\frac{1}{6}-\frac{1}{j(j+1)}\right) {|\tilde c_j^m|^2} +\frac{1}{6}\beta \tilde c_1^0\\
      & \leq \frac{1}{2}\sum_{j\geq 2}\sum_{|m|\leq j}\frac{|\tilde c_j^m|^2}{6}+\frac{1}{6}\beta \tilde c_1^0\\
       &
     =\frac{1}{12}\left(  \|\tilde f\|_{L^2(\mathbb S^2)}^2-|\tilde c_1^0|^2-|\tilde c_1^1|^2-|\tilde c_1^{-1}|^2\right) +\frac{1}{6}\beta \tilde c_1^0\\
       &\leq \frac{1}{12}\left(  \|\tilde f\|_{L^2(\mathbb S^2)}^2-|\tilde c_1^0|^2 \right)+\frac{1}{6}\beta \tilde c_1^0 \\
        &\leq \frac{1}{12}\left(  \|\alpha\sin\theta+Y\|_{L^2(\mathbb S^2)}^2-|\tilde c_1^0|^2 \right)+\frac{1}{6}\beta \tilde c_1^0 \\
       &=-\frac{1}{12}|\tilde c_1^0|^2+\frac{1}{6}\beta \tilde c_1^0+\frac{1}{12}\left(\beta^2 +\| Y\|_{L^2(\mathbb S^2)}^2\right)\\
       &\leq \frac{1}{6}\beta^2+\frac{1}{12}\|Y\|_{L^2(\mathbb S^2)}^2.
       \end{split}
   \end{equation}
   Here we have used \eqref{t89} in the third inequality.
On the other hand, by the weak continuity of $\mathcal E$ in $\mathring   L^{p}(\mathbb S^{2})$, we have  that
\begin{equation}\label{t90}
\mathcal E( \tilde f)=M=\frac{1}{6}\beta^2+\frac{1}{12}\|Y\|_{L^{2}(\mathbb S^{2})}^2.
\end{equation}
Combining \eqref{t90} and \eqref{t499}, we see that all the inequalities in \eqref{t499} are   equalities. In particular,
\begin{equation}\label{k12}
\|\tilde f\|_{L^{2}(\mathbb S^{2})}=\|\alpha\sin\theta+Y\|_{L^{2}(\mathbb S^{2})}.
\end{equation}
From \eqref{k12}, we can apply Lemma~\ref{ucvx} to deduce that $f_{n_{j}}$ converges to $\tilde f$ strongly in $L^{2}(\mathbb S^{2})$  as $j\to+\infty$.
By Lemma~\ref{rl3}, we further deduce that
 \begin{equation}\label{t94}
 \tilde f\in\mathcal R_{\alpha\sin\theta+Y}.
 \end{equation}
 In particular,
 \begin{equation}\label{k13}
\|f_{n_j}\|_{L^{p}(\mathbb S^{2})}=\|\alpha\sin\theta+Y\|_{L^{p}(\mathbb S^{2})}=\|\tilde f\|_{L^{p}(\mathbb S^{2})},\quad\forall\,j.
\end{equation}
Applying Lemma~\ref{ucvx} again, we deduce from \eqref{k13} that $f_{n_{j}}$ converges to $\tilde f$ strongly in $L^{p}(\mathbb S^{2})$  as $j\to+\infty$.
Finally, in view of \eqref{t90} and \eqref{t94},  $\tilde f$ is a maximizer of $\mathcal E$ relative to $\mathcal R_{\alpha\sin\theta+Y}.$
This completes the proof.
   \end{proof}

\subsection{$\mathcal H_{\alpha\sin\theta+Y}$ is  isolated in $\mathcal M$}

 Let $\mathcal H_{\alpha\sin\theta+Y}$ be defined according to \eqref{hy} or \eqref{hy2}. Then it is easy to check that $\mathcal H_{\alpha\sin\theta+Y}$ is compact in $L^p(\mathbb S^2)$, and satisfies
 \begin{equation}\label{pdss2}
\mathcal H_{\alpha\sin\theta+Y}\subset\mathcal R_{\alpha\sin\theta+Y}\cap (\alpha\sin\theta+\mathbb E_2).
 \end{equation}
 Consequently $\mathcal H_{\alpha\sin\theta+Y}\subset\mathcal M.$ The purpose of this subsection is to prove the following proposition.

 \begin{proposition}\label{yc088}
    $\mathcal H_{\alpha\sin\theta+Y}$ is  isolated in $\mathcal M$.
 \end{proposition}
    \begin{proof}
    Without loss of generality, we assume that
 $ \mathcal H_{\alpha\sin\theta+Y}\neq \mathcal M.$ It suffices to prove that
\begin{equation}\label{xdn33}
  \inf_{f\in \mathcal H_{\alpha\sin\theta+Y},\, g\in\mathcal M\setminus \mathcal H_{\alpha\sin\theta+Y}}\|f-g\|_{L^p(\mathbb S^2)}>0.
 \end{equation}
Since the $L^{2}$ topology and the $L^{p}$ topology are equivalent on $\mathcal R_{\alpha\sin\theta+Y}$, it suffices to show that
\begin{equation}\label{tmx01}
  \inf_{f\in \mathcal H_{\alpha\sin\theta+Y},\, g\in\mathcal M\setminus \mathcal H_{\alpha\sin\theta+Y}}\|f-g\|_{L^2(\mathbb S^2)}>0.
 \end{equation}
To this end,
we only need to prove that there exists some $\delta_0>0$, depending only on $p, $ $\alpha$ and $Y$, such that
 \begin{equation}\label{pdss1}
   \|f-g\|_{L^2(\mathbb S^2)}>\delta_0,\quad \forall\,f\in \mathcal H_{\alpha\sin\theta+Y},\,\,g\in\mathcal M\setminus\mathcal H_{\alpha\sin\theta+Y}.
 \end{equation}

Fix $f\in \mathcal H_{\alpha\sin\theta+Y}$ and $g\in\mathcal M\setminus\mathcal H_{\alpha\sin\theta+Y}.$
 Since $f, g\in\mathcal M$, it follows from Proposition~\ref{yc00}(i) that
\[f, g\in  \alpha\sin\theta+\mathbb E_2.\]
Hence we can assume that
    \begin{equation}\label{ffdb1}
        \begin{split}
     f=&\alpha\sin\theta+\tilde a(3\sin^2\theta-1)+ \tilde b\sin (2\theta)\cos\varphi+ \tilde c\sin(2\theta)\sin\varphi
    +  \tilde d\cos^2\theta\cos(2\varphi)\\
&  +  \tilde e\cos^2\theta\sin(2\varphi),
        \end{split}
 \end{equation}
    \begin{equation}\label{ffd1}
        \begin{split}
      g=&\alpha\sin\theta+    a'(3\sin^2\theta-1)+  b'\sin (2\theta)\cos\varphi+ c'\sin(2\theta)\sin\varphi
    +  d'\cos^2\theta\cos(2\varphi)\\
     &+  e'\cos^2\theta\sin(2\varphi),
        \end{split}
    \end{equation}
    where $\tilde a, \tilde b,\tilde c,\tilde d,\tilde e,a',b',c',d',e'\in\mathbb R$. Recall \eqref{yyd1}.
Since $f\in \mathcal H_{\alpha\sin\theta+Y}$, by Lemma~\ref{ce2} in Appendix~\ref{appb}, we have that
$a=\tilde a.$
Notice that the following six functions
 \[ \sin\theta,\,\,3\sin^2\theta-1,\,\, \sin(2\theta)\cos\varphi,\,\,\sin(2\theta)\sin\varphi,\,\,\cos^2\theta\cos(2\varphi),\,\, \cos^2\theta\sin(2\varphi) \]
 are orthogonal in $L^2(\mathbb S^2),$ so we have that
 \[  \|f-g\|_{L^2(\mathbb S^2)}\geq |a-a'|\left(\int_{\mathbb S^2}(3\sin^2\theta-1)^2 d\sigma\right)^{1/2}.\]
Therefore, to prove \eqref{pdss1},  it is sufficient to show that there exists some $\delta>0$, depending only on $\alpha$ and $Y$, such that
\begin{equation}\label{adell}
|a-a'|>\delta.
\end{equation}

 The rest of the proof is devoted to verifying \eqref{adell}.
 Since $g\in \mathcal R_{\alpha\sin\theta+Y},$ it holds that
\[
 \int_{\mathbb S^2}g^m d\sigma=\int_{\mathbb S^2}(\alpha\sin\theta+Y)^m d\sigma:=I_m
\]
for any positive integer $m$.
To prove \eqref{adell},  we need to determine the relations between $a,b,c,d,e$ and $a',b',c',d',e'$.
The  idea is to choose different values of $m$ and express $I_m$ in terms of $a,b,c,d,e$ and $a',b',c',d',e'$, respectively.  In principle, the more values of $m$ we choose to calculate, the more information we  obtain. However, it turns out that choosing $m=2, 3, 4, 5, 6,7$ is sufficient for our purpose, as we will see below.

Without loss of generality, we only calculate $I_m$ in terms of $a,b,c,d,e$. Note that
\begin{align*}
I_m&=\int_{-\frac{\pi}{2}}^{\frac{\pi}{2}}\int_{-\pi}^{\pi}[\alpha\sin\theta+a(3\sin^2\theta-1)+ b\sin (2\theta)\cos\varphi+  c\sin(2\theta)\sin\varphi
   +  d\cos^2\theta\cos(2\varphi)\\
   &+ e\cos^2\theta\sin(2\varphi)]^m \cos\theta d\varphi d\theta.
   \end{align*}
A Maple computation gives
\begin{equation}\label{i2com}
I_2=\frac{4\pi}{15}( 12a^2 + 5\alpha^2 + 4b^2 + 4c^2 + 4d^2 + 4e^2),
\end{equation}
\begin{equation}\label{i3com}
I_3= \frac{16\pi}{35}(4a^3 + 7a\alpha^2 + 2ab^2 + 2ac^2 - 4ad^2 - 4ae^2 + 2b^2d + 4bce - 2c^2d),
\end{equation}
\begin{equation*}
\begin{split}
I_4&= \frac{4\pi}{105}(144 a^{4}+264 a^{2} \alpha^{2}+96 a^{2} b^{2}+96 a^{2} c^{2}+96 a^{2} d^{2}+96 a^{2} e^{2}+21 \alpha^{4}+72 \alpha^{2} b^{2}\\
&+72 \alpha^{2} c^{2}+24 \alpha^{2} d^{2}+24 \alpha^{2} e^{2}+16 b^{4}+32 b^{2} c^{2}+32 b^{2} d^{2}+32 b^{2} e^{2}+16 c^{4}+32 c^{2} d^{2}\\
&+32 c^{2} e^{2}+16 d^{4}+32 d^{2} e^{2}+16 e^{4}),
\end{split}
\end{equation*}
\begin{equation*}
\begin{split}
I_5&= \frac{32\pi }{231 } (48 a^{5}+176 a^{3} \alpha^{2}+40 a^{3} b^{2}+40 a^{3} c^{2}-32 a^{3} d^{2}-32 a^{3} e^{2}+24 a^{2} b^{2} d+48 a^{2} b c e\\
&-24 a^{2} c^{2} d+33 a \alpha^{4}+66 a \alpha^{2} b^{2}+66 a \alpha^{2} c^{2}+8 a \,b^{4}+16 a \,b^{2} c^{2}-8 a \,b^{2} d^{2}-8 a \,b^{2} e^{2}\\
&+8 a \,c^{4}-8 a \,c^{2} d^{2}-8 a \,c^{2} e^{2}-16 a \,d^{4}-32 a \,d^{2} e^{2}-16 a \,e^{4}+22 \alpha^{2} b^{2} d+44 \alpha^{2} b c e\\
&-22 \alpha^{2} c^{2} d+8 b^{4} d+16 b^{3} c e+8 b^{2} d^{3}
+8 b^{2} d \,e^{2}+16 b \,c^{3} e+16 b c \,d^{2} e+16 b c \,e^{3}-8 c^{4} d\\
&-8 c^{2} d^{3}-8 c^{2} d \,e^{2}),
\end{split}
\end{equation*}

\begin{equation*}
\begin{split}
I_6&=\frac{4 \pi}{3003}  (10176 a^{6}+45552 a^{4} \alpha^{2}+10176 a^{4} b^{2}+10176 a^{4} c^{2}+5568 a^{4} d^{2}+5568 a^{4} e^{2}\\
&+1536 a^{3} b^{2} d+3072 a^{3} b c e-1536 a^{3} c^{2} d+15444 a^{2} \alpha^{4}+26208 a^{2} \alpha^{2} b^{2}+26208 a^{2} \alpha^{2} c^{2}\\
&+3744 a^{2} \alpha^{2} d^{2}+3744 a^{2} \alpha^{2} e^{2}+3264 a^{2} b^{4}+6528 a^{2} b^{2} c^{2}+4224 a^{2} b^{2} d^{2}+4224 a^{2} b^{2} e^{2}\\
&+3264 a^{2} c^{4}+4224 a^{2} c^{2} d^{2}+4224 a^{2} c^{2} e^{2}+4416 a^{2} d^{4}+8832 a^{2} d^{2} e^{2}+4416 a^{2} e^{4}\\
&+4992 a \alpha^{2} b^{2} d+9984 a \alpha^{2} b c e-4992 a \alpha^{2} c^{2} d+768 a \,b^{4} d+1536 a \,b^{3} c e-1536 a \,b^{2} d^{3}\\
&-1536 a \,b^{2} d \,e^{2}+1536 a b \,c^{3} e-3072 a b c \,d^{2} e -3072 a b c \,e^{3}-768 a \,c^{4} d+1536 a \,c^{2} d^{3}\\
&+1536 a \,c^{2} d \,e^{2}+429 \alpha^{6}+2860 \alpha^{4} b^{2}+2860 \alpha^{4} c^{2}+572 \alpha^{4} d^{2}+572 \alpha^{4} e^{2}+3120 \alpha^{2} b^{4}\\
&+6240 \alpha^{2} b^{2} c^{2}+3744 \alpha^{2} b^{2} d^{2} +3744 \alpha^{2} b^{2} e^{2}+3120 \alpha^{2} c^{4}+3744 \alpha^{2} c^{2} d^{2}+3744 \alpha^{2} c^{2} e^{2}\\
&+624 \alpha^{2} d^{4}+1248 \alpha^{2} d^{2} e^{2} +624 \alpha^{2} e^{4}+320 b^{6}+960 b^{4} c^{2}+1344 b^{4} d^{2} +960 b^{4} e^{2}\\
&+1536 b^{3} c d e+960 b^{2} c^{4} +1152 b^{2} c^{2} d^{2}+3456 b^{2} c^{2} e^{2}+960 b^{2} d^{4}+1920 b^{2} d^{2} e^{2}+960 b^{2} e^{4}\\
&-1536 b \,c^{3} d e +320 c^{6}+1344 c^{4} d^{2}+960 c^{4} e^{2}+960 c^{2} d^{4}+1920 c^{2} d^{2} e^{2}+960 c^{2} e^{4}+320 d^{6}\\
&+960 d^{4} e^{2}+960 d^{2} e^{4}+320 e^{6}).
\end{split}
\end{equation*}
\begin{equation*}
\begin{split}
I_7&=\frac{16\pi }{ 429} (576 a^{7}+3792 a^{5} \alpha^{2}+672 a^{5} b^{2}+672 a^{5} c^{2}-192 a^{5} d^{2}-192 a^{5} e^{2}+288 a^{4} b^{2} d\\
&+576 a^{4} b c e-288 a^{4} c^{2} d+2028 a^{3} \alpha^{4}+2736 a^{3} \alpha^{2} b^{2}+2736 a^{3} \alpha^{2} c^{2}+96 a^{3} \alpha^{2} d^{2}+96 a^{3} \alpha^{2} e^{2}\\
&+256 a^{3} b^{4}+512 a^{3} b^{2} c^{2}-64 a^{3} b^{2} d^{2}-64 a^{3} b^{2} e^{2}+256 a^{3} c^{4}-64 a^{3} c^{2} d^{2}-64 a^{3} c^{2} e^{2}\\
&-320 a^{3} d^{4}-640 a^{3} d^{2} e^{2}-320 a^{3} e^{4}+816 a^{2} \alpha^{2} b^{2} d+1632 a^{2} \alpha^{2} b c e-816 a^{2} \alpha^{2} c^{2} d \\
&+192 a^{2} b^{4} d+384 a^{2} b^{3} c e+192 a^{2} b^{2} d^{3} +192 a^{2} b^{2} d \,e^{2}+384 a^{2} b \,c^{3} e+384 a^{2} b c \,d^{2} e \\
&+384 a^{2} b c \,e^{3}-192 a^{2} c^{4} d-192 a^{2} c^{2} d^{3}-192 a^{2} c^{2} d \,e^{2}+143 a \alpha^{6}+650 a \alpha^{4} b^{2} +650 a \alpha^{4} c^{2}\\
&+52 a \alpha^{4} d^{2} +52 a \alpha^{4} e^{2}+480 a \alpha^{2} b^{4}+960 a \alpha^{2} b^{2} c^{2}+144 a \alpha^{2} b^{2} d^{2} +144 a \alpha^{2} b^{2} e^{2}+480 a \alpha^{2} c^{4}\\
&
+144 a \alpha^{2} c^{2} d^{2}+144 a \alpha^{2} c^{2} e^{2}-48 a \alpha^{2} d^{4}-96 a \alpha^{2} d^{2} e^{2}-48 a \alpha^{2} e^{4}+32 a \,b^{6}+96 a \,b^{4} c^{2}\\
&+96 a \,b^{2} c^{4}-96 a \,b^{2} d^{4}-192 a \,b^{2} d^{2} e^{2}-96 a \,b^{2} e^{4}+32 a \,c^{6}-96 a \,c^{2} d^{4}-192 a \,c^{2} d^{2} e^{2}\\
&-96 a \,c^{2} e^{4}-64 a \,d^{6}-192 a \,d^{4} e^{2}-192 a \,d^{2} e^{4}-64 a \,e^{6}+130 \alpha^{4} b^{2} d+260 \alpha^{4} b c e\\
&-130 \alpha^{4} c^{2} d+240 \alpha^{2} b^{4} d+480 \alpha^{2} b^{3} c e+144 \alpha^{2} b^{2} d^{3}+144 \alpha^{2} b^{2} d \,e^{2}+480 \alpha^{2} b \,c^{3} e\\
&+288 \alpha^{2} b c \,d^{2} e+288 \alpha^{2} b c \,e^{3}-240 \alpha^{2} c^{4} d-144 \alpha^{2} c^{2} d^{3}-144 \alpha^{2} c^{2} d \,e^{2}+32 b^{6} d+64 b^{5} c e\\
&+32 b^{4} c^{2} d+64 b^{4} d^{3}+64 b^{4} d \,e^{2}+128 b^{3} c^{3} e+128 b^{3} c \,d^{2} e+128 b^{3} c \,e^{3}-32 b^{2} c^{4} d+32 b^{2} d^{5}\\
&+64 b^{2} d^{3} e^{2}+32 b^{2} d \,e^{4} +64 b \,c^{5} e+128 b \,c^{3} d^{2} e+128 b \,c^{3} e^{3}+64 b c \,d^{4} e+128 b c \,d^{2} e^{3}\\
&+64 b c \,e^{5}-32 c^{6} d-64 c^{4} d^{3}-64 c^{4} d \,e^{2}-32 c^{2} d^{5}-64 c^{2} d^{3} e^{2}-32 c^{2} d \,e^{4}).
\end{split}
\end{equation*}
 Denote
 \[A:=\frac{15}{4\pi}I_{2},\quad B:=\frac{35}{16\pi}I_{3},\quad C:=\frac{105}{4\pi}I_{4},\quad D:=\frac{231}{32\pi}I_{5},\quad E:=\frac{3003}{4\pi}I_{6},\quad F:=\frac{429}{16\pi}I_{7}.\]
  Again, with the help of Maple, one can check  that
\begin{equation}\label{abcde1}
\begin{cases}
     12a^2  +  5\alpha^2 +4u  +  4v =A, \\
      4a^3 + 7a\alpha^2 + 2au - 4av + 2w=B, \\
     \alpha^{2} (36 a^{2}-\alpha^{2}+8u-4v)=\frac{1}{4}(C-A^2),\\
      \alpha^{2} (36 a^{3}-a \alpha^{2}+14au-4 a v+6w )=\frac{1}{2}(D-AB),\\
      \alpha^{4}  (36 a^{3}-a \alpha^{2}+10 a u-4 a v+2w)
    =\frac{1}{48 }\left(6A  D-6A^{2} B+3B  C-3F\right),\\ \alpha^2(324\alpha^2a^2+288a^2v-144aw+68\alpha^2u-44\alpha^2v+16u^2-16uv-32v^2\\-9\alpha^{4})
    =\frac{1}{16}(17AC+96B^2-12A^3-E),
\end{cases}
\end{equation}
where
\[
       u:=b^2+c^2,\quad v:=d^2+e^2,\quad w:=b^2d-c^2d+2bce.
\]
Similarly,  we have that
\begin{equation}\label{abcde2}
\begin{cases}
     12  a'^2  +  5\alpha^2 +4 u'   +  4 v'  =A, \\
      4 a'^3 + 7  a'\alpha^2 + 2 a' u' - 4  a'  v' + 2  w'=B, \\
    \alpha^{2}  (36   a'^{2}-\alpha^{2}+8  u'-4   v' )=\frac{1}{4}(C-A^2),\\
      \alpha^{2} (36 a'^{3}-  a' \alpha^{2}+14   a'  u'-4   a'  v'+6  w')=\frac{1}{2}(D-AB),\\
      \alpha^{4} (36  a'^{3}-a' \alpha^{2}+10   a'  u'-4  a'  v'+2  w')=\frac{1}{48}( 6A  D-6A^{2} B+3B  C-3F ),\\
      \alpha^2(324\alpha^2a'^2+288a'^2v'-144a'w'
      +68\alpha^2u'-44\alpha^2v'+16u'^2-16u'v'-32v'^2\\
   -9\alpha^{4}) =\frac{1}{16}(17AC+96B^2-12A^3-E),
\end{cases}
\end{equation}
where \[
       u':=b'^2+c'^2,\quad v':=d'^2+e'^2,\quad w':=b'^2d'-c'^2d'+2b'c'e'.
\]
Therefore  $(a,u,v,w)$ and $(a',u',v',w')$ are two solutions of the following  polynomial system:
 \begin{equation}\label{speo1}
 \begin{cases}
  3x_1^2 +x_2+x_3=b_{1}, \\
    4x_1^3 +7\alpha^2x_1+2x_1x_2-4x_1x_3+2x_4=b_{2}, \\
       9x_1^2   + 2x_2 - x_3=b_{3}, \\
       36x_1^3 -\alpha^2x_1 + 14x_1x_2 - 4x_1x_3   + 6x_4=b_{4}, \\
         36x_1^3-\alpha^2x_1+ 10x_1x_2 - 4x_1x_3  +2x_4=b_{5},   \\
         \alpha^2(324x_1^2+68x_2-44x_3)
          +288x_1^2x_3-144x_1x_4
          +16x_2^2-16x_2x_3-32x_3^2=b_{6},
\end{cases}
\end{equation}
where
\[b_{1}:=\frac{1}{4}(A-5\alpha^2),\quad b_{2}:=B,\quad b_{3}:=\frac{1}{16\alpha^2}(C-A^2)+\frac{1}{4}\alpha^2,\quad b_{4}:=\frac{1}{2\alpha^2}(D-AB), \]
\[b_{5}:=\frac{1}{48\alpha^4}( 6A  D-6A^{2} B+3B  C-3F),\quad b_{6}:=\frac{1}{16\alpha^{2}}(17AC+96B^2-12A^3-E)+9\alpha^{4}.\]
Since $g\notin \mathcal H_{\alpha\sin\theta+Y},$ using Lemma~\ref{ce2}, we deduce that
\[ (a,u,v,w)\neq (a',u',v',w'),\]
that is, $(a,u,v,w)$ and $(a',u',v',w')$ are two \emph{different} solutions of \eqref{speo1}.
 Applying Lemma~\ref{aps1} in Appendix~\ref{secap2}, we deduce that $|a-a'|$ is a positive number depending only on $\alpha,A,B,C,D,E,F$, and thus depending only on $\alpha$ and $Y$. This completes the proof.
   \end{proof}

\begin{remark}
We conjecture that
\[\mathcal H_{\alpha\sin\theta+Y} =\mathcal R_{\alpha\sin\theta+Y}\cap(\alpha\sin\theta+\mathbb E_{2}),\]
which is  more elegant and stronger than Proposition~\ref{yc088}. To show this, one needs to compute $I_m$ for more values of $m$ and make suitable combinations as in \eqref{abcde1}, which may be quite complicated. We do not pursue this further since Proposition~\ref{yc088} is sufficient for our purpose.
\end{remark}

\subsection{Proof of Theorem~\ref{thm1}}
In this subsection, we provide the proof of Theorem~\ref{thm1}. To begin with, we prove the following lemma.

\begin{lemma}\label{too}
For a function $f:\mathbb S^2\to\mathbb R,$   denote
  \begin{equation}\label{hyn20}
 \mathcal H^-_f=\{f\circ\mathsf g\mid \mathsf g\in\mathbf H\setminus\mathbf H^+\}.
 \end{equation}
Then either $\mathcal H_f^+\cap  \mathcal H_f^-=\varnothing$ or $\mathcal H^+_f=\mathcal H_f^-=\mathcal H_f$.
\end{lemma}

\begin{proof}
By the definition of $ \mathcal H^-_f,$ it is easy to see that
  \begin{equation}\label{hyn2}
 \mathcal H_f=  \mathcal H^+_f\cup  \mathcal H^-_f.
 \end{equation}
Suppose that $\mathcal H_f^+\cap  \mathcal H_f^-\neq\varnothing.$
Then there exist $\mathsf g_1\in\mathbf H^{+}$ and $\mathsf g_2\in \mathbf H\setminus \mathbf H^{+}$ such that
\[f\circ \mathsf g_1=f\circ\mathsf g_2,\]
which yields
\begin{equation}\label{ggg2}
f=f\circ\mathsf g_2\circ \mathsf g_1^{-1}.
\end{equation}
Now we show that
\begin{equation}\label{laa2}
\mathcal H_f^+= \mathcal H_f^-.
\end{equation}
First we  prove $\mathcal H_f^+\subset  \mathcal H_f^-$.
Suppose that $f\circ \mathsf g\in\mathcal H^+_f$ with $\mathsf g\in\mathbf H^{+}$. Then by \eqref{ggg2}, we have that
\[
f\circ \mathsf g=f\circ\mathsf g_2\circ \mathsf g_1^{-1}\circ \mathsf g\in\mathcal H_f^-.
\]
Here we used the fact that $\mathsf g_2\circ \mathsf g_1^{-1}\circ \mathsf g\in\mathbf H\setminus \mathbf H^{+},$ as one can easily check.   Conversely, if $f\circ \mathsf g\in\mathcal H^-_f$, where $\mathsf g\in\mathbf H\setminus \mathbf H^{+}$, then by \eqref{ggg2} again,
\[f\circ \mathsf g=f\circ\mathsf g_2\circ \mathsf g_1^{-1}\circ \mathsf g\in \mathcal H_f^+.\]
Here we used the fact that $\mathsf g_2\circ \mathsf g_1^{-1}\circ \mathsf g\in  \mathbf  H^{+}.$    Hence \eqref{laa2} has been proved. Taking into account \eqref{hyn2}, we further deduce that
 $\mathcal H^+_f=\mathcal H_f^-=\mathcal H_f$.
 \end{proof}

\begin{remark}
If  $Y\in\mathbb E_1,$ then  $\mathcal H^+_Y=\mathcal H_Y^-$ (cf. \eqref{y1eq} in Section~\ref{sec2}).  If $Y\in \mathbb E_2$, then $\mathcal H^+_Y=\mathcal H_Y^-$ if and only if $Y$ has the form
\begin{equation*}
Y(\varphi,\theta)=a(3\sin^2\theta-1)+b\sin(2\theta)\cos(\varphi+\mu)+c\cos^2\theta\cos(2\varphi+2\mu),
\end{equation*}
where $ a,b,c,\mu\in\mathbb R$. The proof is left to the reader.
\end{remark}

Having made enough preparations, we are ready to give the proof of Theorem~\ref{thm1}.
\begin{proof}[Proof of Theorem~\ref{thm1}]
Let $\mathcal E$ be defined by \eqref{yc1} and $\mathcal M$ be defined by \eqref{ycc1}. Then $\mathcal E$ and $\mathcal M$ satisfy all the assumptions  in Theorem~\ref{stac}.

We claim that
\begin{equation}\label{1opys3}
\mbox{$\mathcal H_{\alpha\sin\theta+Y}^{+}$ is   isolated in $\mathcal H_{\alpha\sin\theta+Y}.$}
\end{equation}
 In fact, by Lemma~\ref{too},
 either
\begin{equation}\label{1oyps1}
\mathcal H_{\alpha\sin\theta+Y}^{+}\cap\mathcal H_{\alpha\sin\theta+Y}^{-}=\varnothing,
\end{equation}
or
\begin{equation}\label{1oyps2}
\mathcal H_{\alpha\sin\theta+Y}^{+}=\mathcal H_{\alpha\sin\theta+Y}^{-}=\mathcal H_{\alpha\sin\theta+Y}.
\end{equation}
If \eqref{1oyps1} holds, then
\[\inf_{f\in\mathcal H_{\alpha\sin\theta+Y}^{+},\,g\in\mathcal H_{\alpha\sin\theta+Y}^{-}}\|f-g\|_{L^{p}(\mathbb S^{2})}>0,\]
since one can easily check that $\mathcal H_{\alpha\sin\theta+Y}^{+}$ and $\mathcal H_{\alpha\sin\theta+Y}^{-}$ are both compact in $L^{p}(\mathbb S^{2}).$ Hence \eqref{1opys3} has been verified.

From \eqref{1opys3}, taking into account Proposition~\ref{yc088},   we deduce  that $\mathcal H_{\alpha\sin\theta+Y}^{+}$ is   isolated in $\mathcal M$. Then the  stability of $\mathcal H^{+}_{\alpha\sin\theta+Y}$  is a straightforward consequence of Theorem~\ref{stac}.
  \end{proof}

\section{Stability of degree-2 RH waves: $\alpha=0$}\label{sec7}

Throughout this section, let $Y\in\mathbb E_{2}\setminus\{0\}$, and fix $1<p<+\infty$. Suppose that $Y$ has the form
  \begin{equation}\label{yyd2}
      Y=a(3\sin^2\theta-1)+b\sin(2\theta)\cos\varphi+ c\sin(2\theta)\sin\varphi+d\cos^2\theta\cos(2\varphi)+e\cos^2\theta\sin(2\varphi),
  \end{equation}
 where $a,b,c,d,e\in\mathbb R.$

\subsection{Variational  problem}\label{sec71}
As in Section~\ref{sec6}, in order to apply Theorem~\ref{stac}, we solve a maximization problem and verify compactness in this subsection.
In the next two subsections, we will show that $\mathcal O^{+}_{Y}$ is an isolated set of maximizers of the maximization problem.

Consider  \begin{equation}\label{vpoy1}
M=\sup_{\mathcal R_Y}\mathcal E,
\end{equation}
where
 \begin{equation}\label{zbb01}
 \mathcal E(f)=\frac{1}{2}\int_{\mathbb S^2}f\mathcal Gf d\sigma-\frac{1}{4}
  \sum_{m=1,0,-1}\left|\int_{\mathbb S^2} f\overline{Y_1^m} d\sigma\right|^2.
  \end{equation}
 By   Lemma~\ref{pg0}, one can check  that  $\mathcal E$ is well-defined, weakly continuous and locally uniformly continuous (cf. the condition (i) in  Theorem~\ref{stac}) in $\mathring L^{p}(\mathbb S^{2}).$  Moreover,  if $f\in\mathring L^{2}(\mathbb S^{2})$ has the expansion
\begin{equation}\label{h01}
f=\sum_{j\geq 1}\sum_{|m|\leq j}c_j^mY_j^m,\quad c_{j}^{m}=\int_{\mathbb S^{2}}f\overline{Y_{j}^{m}}d\sigma,
\end{equation}
    then  $\mathcal E(f)$ can be written as
\begin{equation}\label{lkj}
\mathcal E(f)=\frac{1}{2}\sum_{j\geq 2}\sum_{|m|\leq j}\frac{|c_j^m|^2}{j(j+1)}.
\end{equation}
In particular, by the conservation laws \eqref{cl1} and \eqref{cl3},  $\mathcal E$  is a flow invariant of the vorticity equation \eqref{ave} (cf. the condition (ii) in  Theorem~\ref{stac}).

\begin{proposition}\label{yc00i1}
Denote by $\mathcal M$ the set of maximizers of \eqref{vpoy1}, i.e.,
  \begin{equation}\label{ycc1e}
\mathcal M=\left\{f\in\mathcal R_{Y}\mid \mathcal E(f)=M\right\}.
  \end{equation}
  \begin{itemize}
    \item[(i)] It holds that \begin{equation}\label{t45}
  M= \frac{1}{12}\|Y\|_{L^2(\mathbb S^2)}^2,\quad\mathcal M=\mathcal R_{ Y}\cap  \mathbb E_2.
\end{equation}
    \item [(ii)] For any sequence $\{f_{n}\}\subset \mathcal R_{Y}$ satisfying
$\lim_{n\to+\infty}\mathcal E(f_{n})=M,$
 there is a subsequence  of  $\{f_n\}$, denoted by $\{f_{n_{j}}\}$, such that   $f_{n_{j}}$ converges to some $\tilde f \in\mathcal M$ strongly in $L^p(\mathbb S^2)$ as $j\to+\infty.$ In particular, $\mathcal M$ is compact in $L^p(\mathbb S^2)$.
  \end{itemize}

  \end{proposition}
 \begin{proof}
 The proof is analogous to that of Proposition~\ref{yc00}. For the reader's convenience, we provide the details below.

 First we prove (i).
Fix $f\in\mathcal R_{Y}.$ Assume that $f$ has the expansion  \eqref{h01}.
In view of \eqref{lkj}, we can repeat the argument in \eqref{t43} by setting $\beta=0$ to deduce that
\begin{equation}\label{h05}
\mathcal E(f)\leq \frac{1}{12} \|f\|^{2}_{L^{2}(\mathbb S^{2})} = \frac{1}{12} \|Y\|^{2}_{L^{2}(\mathbb S^{2})}.
\end{equation}
Moreover, the equality holds if and only if $f\in\mathbb E_{2}.$   This proves \eqref{t45}.

Next we prove (ii). Since $\{f_{n}\}$ is bounded in $\mathring L^{p}(\mathbb S^{2}),$ there is a subsequence $\{f_{n_{j}}\}$  such that   $f_{n_{j}}$ converges to some $\tilde f $ weakly in $\mathring  L^p(\mathbb S^2)$ as $j\to+\infty.$
Below we show that $\tilde f\in \mathcal M$ and $f_{n_{j}}$ converges to  $\tilde f$ strongly in $L^p(\mathbb S^2)$ as $j\to+\infty.$
Since $\{f_{n}\}$ is bounded in $\mathring L^{2}(\mathbb S^{2}),$ we see  that  $f_{n_{j}}$ also converges to   $\tilde f $ weakly in $\mathring  L^2(\mathbb S^2)$ as $j\to+\infty.$ As a result,
\begin{equation}\label{t8989}
\|\tilde f\|_{L^{2}(\mathbb S^{2})}\leq \liminf _{j\to+\infty}\|f_{n_{j}}\|_{L^{2}(\mathbb S^{2})}=\|Y\|_{L^{2}(\mathbb S^{2})}.
\end{equation}
Consider the Fourier expansion of $\tilde f$:
\begin{equation*}
\tilde f=\sum_{j\geq 1}\sum_{|m|\leq j}\tilde c_j^mY_j^m,\quad \tilde c_{j}^{m}=\int_{\mathbb S^{2}}\tilde f\overline{Y_{j}^{m}}d\sigma.
\end{equation*}
Repeating the argument in \eqref{t499} by setting $\beta=0$ and using \eqref{t8989}, we have that
        \begin{equation}\label{t4999}
       \mathcal E(\tilde f)
        \leq \frac{1}{12}\|\tilde f\|^2_{L^{2}(\mathbb S^{2})}\leq \frac{1}{12}   \| Y\|_{L^2(\mathbb S^2)}^2.
   \end{equation}
On the other hand, by the weak continuity of $\mathcal E$ in $\mathring   L^{p}(\mathbb S^{2})$,
\begin{equation}\label{t9090}
\mathcal E( \tilde f)=\lim_{j\to+\infty}\mathcal E(f_{n_j})=M= \frac{1}{12}\|Y\|_{L^{2}(\mathbb S^{2})}^2.
\end{equation}
With \eqref{t9090}, we see that   the two inequalities in \eqref{t4999} are in fact equalities. In particular,
\begin{equation}\label{t9094}
\|\tilde f\|_{L^{2}(\mathbb S^{2})}=\|Y\|_{L^{2}(\mathbb S^{2})}.
\end{equation}
Applying Lemma~\ref{ucvx}, we deduce from \eqref{t9094} that $f_{n_{j}}$ converges to $\tilde f$ strongly in $L^{2}(\mathbb S^{2})$  as $j\to+\infty$.
Then by Lemma~\ref{rl3}, we obtain
 \begin{equation}\label{t9494}
 \tilde f\in\mathcal R_{Y},
 \end{equation}
 which in conjunction with \eqref{t9090}  implies that $\tilde f$ is a maximizer of $\mathcal E$ relative to $\mathcal R_{Y}.$  Moreover, by \eqref{t9494} and the fact that $\{f_{n_j}\}\subset\mathcal R_Y$, we have that
$\|f_{n_j}\|_{L^p(\mathbb S^2)}=\|\tilde f\|_{L^p(\mathbb S^2)}$ for  all $j$.
Applying Lemma~\ref{ucvx} again, we  obtain the desired strong convergence in $L^p(\mathbb S^2)$.

 \end{proof}

\subsection{$\mathcal M=\mathcal O_{Y}$}

Recall \eqref{oy} in Section~\ref{sec21} for the definition of $\mathcal O_Y$. The following lemma shows that $\mathcal O_{Y}$ is exactly the set of maximizers of $\mathcal E$ relative to $\mathcal R_Y$.
\begin{proposition}\label{sh1}
  $\mathcal M=\mathcal O_Y$.
\end{proposition}

 \begin{proof}
 By Proposition~\ref{yc00i1}, it is sufficient to show that  $\mathcal O_Y=\mathcal R_{Y}\cap\mathbb E_2.$
Since the   inclusion $\mathcal O_Y\subset \mathcal R_{Y}\cap \mathbb E_2$ is obvious, the rest of the proof is devoted to verifying
\begin{equation}\label{bnjn1}
\mathcal R_{Y}\cap \mathbb E_2\subset \mathcal O_Y.
\end{equation}

Fix $f\in\mathcal R_Y\cap\mathbb E_{2}.$ Suppose that $f$ has the form
 \[f=a'(3\sin^2\theta-1)+b'\sin(2\theta)\cos\varphi+ c'\sin(2\theta)\sin\varphi+d'\cos^2\theta\cos(2\varphi)+e'\cos^2\theta\sin(2\varphi),\]
 where $a',b',c',d',e'\in\mathbb R.$
Since $f\in\mathcal R_Y$, we have that
\begin{equation}\label{me2}
\int_{\mathbb S^2}f^2d\sigma=\int_{\mathbb S^2}Y^2d\sigma,
\end{equation}
 \begin{equation}\label{me3}
\int_{\mathbb S^2}f^3d\sigma=\int_{\mathbb S^2}Y^3d\sigma.
\end{equation}
 Recall \eqref{yyd2}. By some straightforward calculations (cf. \eqref{i2com} and \eqref{i3com}), we obtain from  \eqref{me2} that
\begin{equation}\label{m22}
    3a^2+b^2+c^2+d^2+e^2=3a'^2+b'^2+c'^2+d'^2+e'^2,
\end{equation}
 and from  \eqref{me3} that
\begin{equation}\label{m33}
\begin{split}
&2a^3+ab^2+ac^2-2ad^2-2ae^2+b^2d+2bce-c^2d\\
=&2a'^3+a'b'^2+a'c'^2-2a'd'^2-2a'e'^2+b'^2d'+2b'c'e'-c'^2d'.
\end{split}
\end{equation}

Below we show that \eqref{m22} and \eqref{m33} ensure $f\in \mathcal O_Y$.
To this end,  denote
\[z=\sin\theta,\quad x=\cos\theta\cos\varphi,\quad y=\cos\theta\sin\varphi.\]
 Then $Y$ and $f$ can be written as the following quadratic forms:
 \[Y=a(3z^2-1)+2bxz+2cyz+d(x^2-y^2)+2exy,\]
  \[f=a'(3z^2-1)+2b'xz+2c'yz+d'(x^2-y^2)+2e'xy,\]
  or equivalently,
\begin{equation}\label{mf1}
Y=(x,y,z)\left(\begin{array}{ccc}
    d& e&b \\
    e &-d&c\\
    b&c&3a
\end{array}\right)\left(\begin{array}{c}
     x\\
     y\\
     z
\end{array}\right)-a,
\end{equation}
\begin{equation}\label{mf2}
f=(x,y,z)\left(\begin{array}{ccc}
    d'& e'&b' \\
    e' &-d'&c'\\
    b'&c'&3a'
\end{array}\right)\left(\begin{array}{c}
     x\\
     y\\
     z
\end{array}\right)-a'.
\end{equation}
  Taking into account the fact that $x^2+y^2+z^2=1,$ we obtain from \eqref{mf1} and  \eqref{mf2}  that
\[Y=(x,y,z)A\left(\begin{array}{c}
     x\\
     y\\
     z
\end{array}\right),\quad A:=\left(\begin{array}{ccc}
    d-a& e&b \\
    e &-d-a&c\\
    b&c&2a
\end{array}\right),\]
    \[f=(x,y,z)A'\left(\begin{array}{c}
     x\\
     y\\
     z
\end{array}\right),\quad A':=\left(\begin{array}{ccc}
    d'-a'& e'&b' \\
    e' &-d'-a'&c'\\
    b'&c'&2a'
\end{array}\right).\]
To show $f\in \mathcal O_Y,$ it suffices to show that there exists a $3\times 3$ orthogonal matrix $P$ such that
\begin{equation}\label{apa}
A=P^TA'P,
\end{equation} where $P^T$ denotes the transpose of $P$. Notice that  both $A$ and $A'$ are symmetric. Therefore, to prove \eqref{apa}, we need only to show that $A$ and $A'$ have the same eigenvalues, or, equivalently, $A$ and $A'$ have the same characteristic polynomial.
By a straightforward computation, the characteristic polynomial $  \mathsf P_A$  of $A$ is
\begin{align*}
   \mathsf P_A(\lambda)= &\lambda^3 -(3a^2 + b^2 + c^2 + d^2 + e^2)\lambda - 2a^3 - ab^2 - ac^2  + 2ad^2 + 2ae^2 - b^2d  - 2bce\\
   &+ c^2d,
\end{align*}
and the characteristic polynomial $  \mathsf P_{A'}$  of $A'$ is
\begin{align*}
\mathsf P_{A'}(\lambda)=&\lambda^3 -(3a'^2 + b'^2 + c'^2 + d'^2 + e'^2)\lambda - 2a'^3 - a'b'^2 - a'c'^2 + 2a'd'^2 + 2a'e'^2 - b'^2d' \\
&- 2b'c'e'+ c'^2d'.
\end{align*}
Taking into account \eqref{m22} and \eqref{m33}, we obtain $\mathsf P_A=\mathsf P_{A'}$. This completes the proof.
 \end{proof}

 \begin{remark}
Proposition~\ref{sh1} classifies the space $\mathbb E_{2}$ according to the equivalence relation of equimeasurability. More specifically, for any $f,g\in\mathbb E_{2}$, $f$ and $g$ are equimeasurable if and only if $g=f\circ\mathsf g$ for some $\mathsf g\in\mathbf O(3)$.
\end{remark}

\subsection{Proof of Theorem~\ref{thm2}}\label{sec73}

Now we are ready to prove  Theorem~\ref{thm2}.
\begin{proof}[Proof of Theorem~\ref{thm2}]

First we claim that
\begin{equation}\label{oyoyp}
\mathcal O_Y=\mathcal O_Y^+.
\end{equation}
In fact, for $Y\in\mathbb E_2$, we have
$Y(-\mathbf x)=Y(\mathbf x)$ for every $\mathbf x\in\mathbb S^2$.
If $\mathsf g\in\mathbf O(3)$ has determinant $-1$, then
$-\mathsf g\in\mathbf S\mathbf O(3)$ and
$Y\circ\mathsf g=Y\circ(-\mathsf g)$. This proves \eqref{oyoyp}.

Combining Proposition~\ref{sh1} and \eqref{oyoyp}, we have $\mathcal M=\mathcal O_Y^+$. By the properties of $\mathcal E$ established in Section~\ref{sec71}
and Proposition~\ref{yc00i1}, all the assumptions of
Theorem~\ref{stac} are satisfied with $\mathcal N=\mathcal M$.
This proves the desired stability of $\mathcal O_{Y}^{+}$.
 \end{proof}

\section{Theorems~\ref{thm1} and \ref{thm2} are optimal}\label{sec8}

First we show that Theorem~\ref{thm1} is optimal.
\begin{theorem}\label{thm81}

 Theorem~\ref{thm1} is optimal, in the sense that
 there does not exist a compact set $\varnothing\neq \mathcal X\subsetneqq \mathcal H^{+}_{\alpha\sin\theta+Y}$  such that \eqref{hrb1} holds with $\mathcal H^{+}_{\alpha\sin\theta+Y}$ replaced by $\mathcal X$.
\end{theorem}

 \begin{proof}
First we claim that for any $f_0,$ $\tilde f\in\mathcal H^+_{\alpha\sin\theta+Y}$, there exist a sequence of solutions $\{\zeta^n\}$ of the vorticity equation \eqref{ave} and a sequence $\{t_n\}\subset\mathbb R_+$ such that
 \begin{equation}\label{ppiio}
 \lim_{n\to+\infty}\|\zeta^n_0-f_0\|_{L^p(\mathbb S^2)}=0,\quad \lim_{n\to+\infty}\|\zeta^n_{t_n}-\tilde f\|_{L^p(\mathbb S^2)}=0.
 \end{equation}
Suppose that $f_0=\alpha\sin\theta+Y_0,$ where $Y_0\in\mathbb E_2$ is a rotation of $Y$ about the polar axis.
    Consider a sequence of  solutions $\{\zeta^n\}$  of the vorticity equation \eqref{ave}   such that
\[\zeta^n_0=\left(\alpha+\frac{1}{n}\right)\sin\theta+Y_0,\quad n\in\mathbb Z_+.\]
Then $\zeta_0^n\to f_0$ in $L^p(\mathbb S^2)$ as $n\to+\infty.$ By \eqref{dg2rh}, we have that
 \[ \zeta^n_{t}(\varphi,\theta)=\left(\alpha+\frac{1}{n}\right)\sin\theta+Y_0(\varphi-  c_nt,\theta),\quad   c_n=\frac{1}{3}\left(\alpha+\frac{1}{n}\right) -\omega.\]
Notice that $c_n\neq 0$ if  $n$ is  large enough. In fact, if $\alpha\geq 3\omega,$ then $c_n>0$ for any $n\in\mathbb Z_+.$ If $\alpha< 3\omega,$ then $c_n<0$  as long as $n$ is sufficiently large.
Therefore,  for sufficiently  large $n$, it holds that
 \begin{equation}\label{ffsc}
 \{\zeta^n_{t}\mid t>0\}=\mathcal H^{+}_{(\alpha+n^{-1})\sin\theta+Y_0}=\mathcal H^{+}_{(\alpha+n^{-1})\sin\theta+Y}.
 \end{equation}
 Suppose that $\tilde f=\alpha\sin\theta+\tilde Y,$ where $\tilde Y$ is a rotation of $Y$ about the polar axis.
 From \eqref{ffsc}, there exists a sequence $\{t_n\}\subset\mathbb R_+$ such that
 $\zeta^n_{t_n}=(\alpha+n^{-1})\sin\theta+\tilde Y,$ and thus
 $\zeta_{t_n}^n\to \tilde f$ in $L^p(\mathbb S^2)$ as $n\to+\infty$.  The claim has been proved.

 From the above claim, we can easily prove the theorem.
 Suppose that $\varnothing\neq \mathcal X\subsetneqq \mathcal H^{+}_{\alpha\sin\theta+Y}$ is compact in $L^p(\mathbb S^2)$. Fix $f_0\in\mathcal X.$
Since $\mathcal X\neq \mathcal H^+_{\alpha\sin\theta+Y}$,
there exists $\tilde f\in \mathcal H^+_{\alpha\sin\theta+Y}\setminus \mathcal X.$
By the above claim, there exist a sequence of solutions $\{\zeta^n\}$ of the vorticity equation \eqref{ave} and a sequence $\{t_n\}\subset\mathbb R_+$ such that \eqref{ppiio} holds.
Since  $\mathcal X$ is compact in $L^p(\mathbb S^2),$ we have that
\begin{equation}\label{ffc7}
\min_{f\in\mathcal X}\| f-\tilde f\|_{L^p(\mathbb S^2)}>\delta
\end{equation}
for some $\delta>0.$
From \eqref{ppiio} and \eqref{ffc7}, we have that
\begin{equation}\label{ffc8}
 \min_{f\in\mathcal X}\|\zeta^n_{t_n}- f\|_{L^p(\mathbb S^2)}>\delta \quad\mbox{for sufficiently large }n.
\end{equation}
To conclude, we have constructed a sequence of solutions $\{\zeta^n\}$ such that $\zeta^n_0$ converges to some $f_0\in\mathcal X$ in $L^p(\mathbb S^2)$ as $n\to+\infty$, but \eqref{ffc8} holds for some sequence $\{t_n\}\subset\mathbb R_+$. This means that  $\mathcal X$ is not stable as in \eqref{hrb1}.
  \end{proof}

Next we show that Theorem~\ref{thm2} is also optimal.   The proof makes use of generalized RH waves discussed in Appendix~\ref{appc}.

  \begin{theorem}
 Theorem~\ref{thm2} is optimal, i.e., there does not exist a compact set   $\varnothing\neq \mathcal X\subsetneqq \mathcal O^+_{Y}$  such that \eqref{hrb2} holds with $\mathcal O^+_{Y}$ replaced by $\mathcal X$.
  \end{theorem}

\begin{proof}

As in the proof of Theorem~\ref{thm81}, it suffices to show  that for any $f_0,$ $\tilde f\in\mathcal O^+_{Y}$, there exist a sequence of solutions $\{\zeta^n\}$ of the vorticity equation \eqref{ave} and a sequence $\{t_n\}\subset\mathbb R_+$ such that
\begin{equation}\label{zzs2}
 \lim_{n\to+\infty}\|\zeta^n_0-f_0\|_{L^p(\mathbb S^2)}=0,\quad \lim_{n\to+\infty}\|\zeta^n_{t_n}-\tilde f\|_{L^p(\mathbb S^2)}=0.
\end{equation}
By Euler's rotation theorem (cf. \cite{PPR}), we  can choose $\mathbf p\in\mathbb S^{2}$ and $\mu\in\mathbb R$ such that
\[\tilde f(\mathbf x)=f_{0}(\mathsf R^{\mathbf p}_{\mu}\mathbf x),\quad\mathbf x\in\mathbb S^{2},\]
where   $\mathsf R^{\mathbf p}_{\mu}$ is the rotation matrix with  rotation axis $\mathbf p$ and rotation angle $\mu$.
By Proposition~\ref{lem999} in Appendix~\ref{appc},  we can choose a sequence of solutions of the vorticity equation \eqref{ave}:
\[\zeta^{n}(\mathbf x,t)=\alpha_{n}\mathbf p\cdot(\mathsf R^{\mathbf e_3}_{\omega t}\mathbf x)+f_{0}(\mathsf R^{\mathbf p}_{-c_{n}t}\mathsf R^{\mathbf e_{3}}_{\omega t}\mathbf x),  \quad c_{n}=\frac{1}{3}\alpha_{n},\,\,n\in\mathbb Z_+,\]
where $\alpha_{n}\in\mathbb R$ is to be determined. To finish the proof, it suffices to show that there exist a sequence $\{\alpha_{n}\}\subset\mathbb R$ and a sequence $\{t_{n}\}\subset\mathbb R_{+}$ such that
\[\lim_{n\to+\infty}\alpha_{n}= 0,\quad \mathsf R^{\mathbf p}_{-c_{n}t_{n}}\mathsf R^{\mathbf e_{3}}_{\omega t_{n}}=\mathsf R^{\mathbf p}_{\mu} \quad\forall\,n.\]
We distinguish two cases:
\begin{itemize}
\item [(i)] ($\omega=0$) In this case, we can choose
\[t_{n}=n^{2},\quad\alpha_{n}=-3\left(\frac{2\pi}{n}+\frac{\mu}{n^{2}}\right)\]
such that $-c_{n}t_{n}=2n\pi+\mu$.
\item [(ii)]($\omega\neq 0$) In this case, we can choose
\[t_{n}=\frac{2 n^{2}\pi}{|\omega|},\quad \alpha_{n}=-3|\omega|\left(\frac{1}{n}+\frac{\mu}{2n^{2}\pi}\right)\]
such that $-c_{n}t_{n}=2n\pi+\mu$.
\end{itemize}
This completes the proof.
\end{proof}

\appendix

\section{Characterizations of  $\mathcal H_Y$ in terms of  Fourier coefficients}\label{appb}

In this appendix, we characterize the set $\mathcal H_Y$ in terms of the Fourier coefficients of $Y$ for $Y\in\mathbb E_1$ and $Y\in\mathbb E_2$.

First we consider the case $Y\in\mathbb E_1$.
\begin{lemma}\label{ce1}
Suppose that $Y\in\mathbb E_1$ has the form
\[Y=a Y_1^0+bY_1^1+cY_1^{-1},\quad a\in\mathbb R,\quad b=-\overline c.\]
Let $f\in\mathbb E_1$ be given by
\begin{equation*}
f =a'Y_1^0 +b'Y_1^1+c'Y_1^{-1},\quad  a'\in\mathbb R,\quad b'=-\overline c'.
\end{equation*}
Then  $f\in\mathcal H_Y$ if and only if the constants $a',b',c'$ satisfy
\begin{equation*}
a'=a,\quad |b'|=|b|.
\end{equation*}
\end{lemma}
\begin{proof}
By \eqref{hy2}, $f\in\mathcal H_Y$ if and only if there exists
$\beta\in\mathbb R$ such that either
\[
f(\varphi,\theta)=Y(\varphi+\beta,\theta)
\]
or
\[
f(\varphi,\theta)=Y(-\varphi+\beta,\theta).
\]
For the rotation branch, we have
\[
Y_1^1(\varphi+\beta,\theta)
=e^{i\beta}Y_1^1(\varphi,\theta),
\qquad
Y_1^{-1}(\varphi+\beta,\theta)
=e^{-i\beta}Y_1^{-1}(\varphi,\theta).
\]
Hence
\[
Y(\varphi+\beta,\theta)
=
aY_1^0
+be^{i\beta}Y_1^1
+ce^{-i\beta}Y_1^{-1}.
\]
For the reflection branch, using
$Y_1^{-1}=-\overline{Y_1^1}$, we obtain
\[
Y_1^1(-\varphi+\beta,\theta)
=-e^{i\beta}Y_1^{-1}(\varphi,\theta),
\]
and
\[
Y_1^{-1}(-\varphi+\beta,\theta)
=-e^{-i\beta}Y_1^1(\varphi,\theta).
\]
Consequently,
\[
Y(-\varphi+\beta,\theta)
=
aY_1^0
-ce^{-i\beta}Y_1^1
-be^{i\beta}Y_1^{-1}.
\]
In either case, comparison with
\[
f=a'Y_1^0+b'Y_1^1+c'Y_1^{-1}
\]
shows that
\[
a'=a,\qquad |b'|=|b|.
\]
Conversely, suppose that $a'=a$ and $|b'|=|b|$.
If $b=0$, then $b'=c=c'=0$, and hence $f=Y$.
If $b\neq0$, choose $\beta\in\mathbb R$ such that $
b'=be^{i\beta}$.
Since $c=-\overline b$ and $c'=-\overline{b'}$, it follows that
$
c'=ce^{-i\beta}$.
Therefore
\[
f(\varphi,\theta)=Y(\varphi+\beta,\theta),
\]
so that $f\in\mathcal H_Y^+\subset\mathcal H_Y$.
\end{proof}

Next we consider the case $Y\in \mathbb E_2.$  For ease of application, the characterization is given in terms of the Fourier coefficients associated with the real basis $\{R_j^m\}_{j\geq 1,|m|\leq j}$.

\begin{lemma}\label{ce2}
Suppose that $Y\in\mathbb E_2 $ has the form
\begin{equation*}
\begin{split}
Y = a(3\sin^2\theta-1)+b\sin (2\theta)\cos\varphi+c\sin(2\theta)\sin\varphi&+d\cos^2\theta\cos(2\varphi) +e\cos^2\theta\sin(2\varphi),
\end{split}
\end{equation*}
where $a, b, c, d, e\in\mathbb R.$
Let $f\in\mathbb E_2$ be given by
\begin{equation*}
\begin{split}
f =a'(3\sin^2\theta-1) +b'\sin (2\theta)\cos\varphi+c'\sin(2\theta)\sin\varphi
 +d'\cos^2\theta\cos(2\varphi)+e'\cos^2\theta\sin(2\varphi),
\end{split}
\end{equation*}
where $ a', b', c', d', e'\in\mathbb R.$
 Then $f\in\mathcal H_{Y}$ if and only if
   $a',b',c',d',e'$ satisfy
     \begin{equation}\label{bcbpcp}
     \begin{cases}
     a=a',\\
     b^2+c^2=b'^2+c'^2,\\
     d^2+e^2=d'^2+e'^2,\\
     b^2d-c^2d+2bce=b'^2d'-c'^2d'+2b'c'e'.
     \end{cases}
     \end{equation}

 \begin{proof}

 First we prove the ``only if" part. Suppose that $f\in\mathcal H_Y$. Then  there exists $\beta\in\mathbb R$ (cf.  \eqref{hy2} in Section~\ref{sec21}) such that
 \begin{equation}\label{nf200}
 f(\varphi,\theta) = Y(\pm\varphi+\beta,\theta),\quad\forall\,\varphi,\,\theta\in\mathbb R.
 \end{equation}
 In view of the expression for $Y$,  we have that
 \begin{equation}\label{nf20}
 \begin{split}
 &Y(\pm\varphi+\beta,\theta)\\ =&a(3\sin^2\theta-1)+(b\cos\beta+c\sin\beta)\sin(2\theta)\cos\varphi \pm ( c\cos\beta-b\sin\beta)\sin(2\theta)\sin\varphi\\
 &+(d\cos(2\beta)+e\sin(2\beta))\cos^2\theta\cos(2\varphi) \pm( e\cos(2\beta)-d\sin(2\beta) )\cos^2\theta\sin(2\varphi),
 \end{split}
 \end{equation}
 Using the fact that
\[\left\{3\sin^2\theta-1,\,\, \sin(2\theta)\cos\varphi,\,\,\sin(2\theta)\sin\varphi,\,\,\cos^2\theta\cos(2\varphi),\,\, \cos^2\theta\sin(2\varphi)\right\}\]
is a basis of
$\mathbb E_2$ (cf. \eqref{e2span}),  we can  compare \eqref{nf20} with the expression of $f$ to obtain
 \begin{equation}\label{nf21}
 \begin{split}
 & a=a', \quad b'=b\cos\beta+c\sin\beta,\quad c'=\pm ( c\cos\beta-b\sin\beta),\\
 & d'=d\cos(2\beta)+e\sin(2\beta),\quad e'=\pm( e\cos(2\beta)-d\sin(2\beta) ).
  \end{split}
  \end{equation}
 Then the desired relations \eqref{bcbpcp} follow from \eqref{nf21} by a straightforward computation.

Next we prove the ``if" part. Suppose that \eqref{bcbpcp} holds.
Since $a=a'$ and the common zonal term is invariant under $\mathbf H$,
we may subtract this term and assume that $a=a'=0$. Set
\[
r:=\sqrt{b^2+c^2}=\sqrt{b'^2+c'^2},\qquad
s:=\sqrt{d^2+e^2}=\sqrt{d'^2+e'^2}.
\]
If $r=s=0$, then $f=Y$.
If $r=0<s$, write
$d+ie=se^{2i\eta}$ and $d'+ie'=se^{2i\nu}$;
then $f(\varphi,\theta)=Y(\varphi+\eta-\nu,\theta)$.
If $s=0<r$, write
$b+ic=re^{i\xi}$ and $b'+ic'=re^{i\mu}$;
then $f(\varphi,\theta)=Y(\varphi+\xi-\mu,\theta)$.
Thus all zero-amplitude cases give $f\in\mathcal H_Y^+$.
It remains to consider $r>0$ and $s>0$.

Choose $\xi, \eta, \mu,\nu \in\mathbb R$ such that
 \[b=r\cos\xi,\quad c=r\sin\xi,\quad d=s\cos(2\eta),\quad e=s\sin(2\eta),\]
  \[b'=r\cos\mu,\quad c'=r\sin\mu,\quad d'=s\cos(2\nu),\quad e'=s\sin(2\nu).\]
  Then $Y$ and $f$ can be written as
  \begin{equation}\label{nf15}
Y(\varphi,\theta) = r\sin (2\theta)\cos(\varphi-\xi) +s\cos^2\theta\cos(2\varphi-2\eta),
\end{equation}
\begin{equation}\label{nf26}
f(\varphi,\theta) = r\sin (2\theta)\cos(\varphi-\mu)+s\cos^2\theta\cos(2\varphi-2\nu).
\end{equation}
In view of the last relation in \eqref{bcbpcp}, one can check that
 \begin{equation*}
   \cos(2\xi-2\eta) =\cos(2\mu-2\nu).
 \end{equation*}
Consequently, there exists  some integer $k$ such that
 \begin{equation}\label{nf91}
     2\xi-2\eta=2\mu-2\nu+2k\pi,
 \end{equation}
 or
  \begin{equation}\label{nf92}
     2\xi-2\eta+2\mu-2\nu=2k\pi.
 \end{equation}
 If \eqref{nf91} holds, then by \eqref{nf15} and \eqref{nf26}, we have that
 \begin{align*}
     f(\varphi,\theta) &= r\sin (2\theta)\cos(\varphi-\mu)+s\cos^2\theta\cos(2\varphi-2\nu)\\
     &=r\sin (2\theta)\cos(\varphi-\xi+\xi-\mu)+s\cos^2\theta\cos(2\varphi-2\eta+2\eta-2\nu)\\
     &=r\sin (2\theta)\cos(\varphi-\xi+\xi-\mu)+s\cos^2\theta\cos(2\varphi-2\eta+2\xi-2\mu-2k\pi)\\
     &=r\sin (2\theta)\cos(\varphi-\xi+\xi-\mu)+s\cos^2\theta\cos(2\varphi-2\eta+2\xi-2\mu)\\
     &=Y(\varphi+\xi-\mu,\theta),
 \end{align*}
which belongs to $\mathcal H^+_{ Y}$. If \eqref{nf92}  holds, then by \eqref{nf15} and \eqref{nf26},  we have that
  \begin{align*}
     f(\varphi,\theta) &= r\sin (2\theta)\cos(\varphi-\mu)+s\cos^2\theta\cos(2\varphi-2\nu)\\
     &=r\sin (2\theta)\cos(\varphi+\xi-\xi-\mu)+s\cos^2\theta\cos(2\varphi+2\eta-2\eta-2\nu)\\
     &= r\sin (2\theta)\cos(\varphi+\xi-\xi-\mu)+s\cos^2\theta\cos(2\varphi+2\eta-2\xi-2\mu+2k\pi)\\
     &=r\sin (2\theta)\cos(-\varphi-\xi+ \xi+\mu)+s\cos^2\theta\cos(-2\varphi-2\eta+2\xi+2\mu)\\
     &=Y(-\varphi+\xi+\mu,\theta),
 \end{align*}
which belongs to $\mathcal H^-_{Y}$ (cf. \eqref{hyn20}  for the definition of $\mathcal H^-_{Y}$). This completes the proof.
\end{proof}

\end{lemma}

\section{On a polynomial system}\label{secap2}
In this appendix, we study a system of polynomial equations arising in the proof of Proposition~\ref{yc088}.
\begin{lemma}\label{aps1}
Given $\alpha, b_1,b_2,b_3,b_4,b_5,b_6\in\mathbb R$ with $\alpha\neq 0$, the following  polynomial system has at most two solutions:
 \begin{numcases}{}
  3x_1^2 +x_2+x_3=b_1,\label{01}\\
    4x_1^3 +7\alpha^2x_1+2x_1x_2-4x_1x_3+2x_4=b_2,\label{02}\\
       9x_1^2   + 2x_2 - x_3=b_3,\label{03}\\
        36x_1^3 -\alpha^2x_1 + 14x_1x_2 - 4x_1x_3   + 6x_4=b_4,\label{04}\\
         36x_1^3-\alpha^2x_1+ 10x_1x_2 - 4x_1x_3  +2x_4=b_5,  \label{05}\\
         \alpha^2(324x_1^2+68x_2-44x_3)
          +288x_1^2x_3-144x_1x_4
          +16x_2^2-16x_2x_3-32x_3^2=b_6.\label{0i6}
 \end{numcases}
Moreover, if $(\tilde x_1,\tilde x_2,\tilde x_3,\tilde x_4)$ and $(\hat x_1,\hat x_2,\hat x_3,\hat x_4)$ are two different solutions, then $\tilde x_1\neq \hat x_1$, and consequently $|\tilde x_1-\hat x_1|$
is a positive number depending only on $\alpha, b_{1},b_{2},b_{3},b_{4},b_{5},b_{6}.$
\end{lemma}

\begin{proof}
From \eqref{03} and \eqref{05}, we have that
       \begin{equation}\label{06}
       -\alpha^2 x_1+2x_1x_2+2x_4=b_5-4b_3x_1.
        \end{equation}
From \eqref{04} and \eqref{05}, we have that
       \begin{equation}\label{07}
  4x_1x_2+4x_4=b_4-b_5,
        \end{equation}
        and thus
             \begin{equation}\label{072}
 x_4=-x_1x_2+\frac{1}{4}(b_4-b_5).
        \end{equation}
        From \eqref{06} and \eqref{07}, we have that
            \begin{equation}\label{08}
            (\alpha^{2}-4b_{3})x_{1}=\frac{1}{2}(b_{4}-3b_{5}).
              \end{equation}
        On the other hand, by \eqref{01} and \eqref{03},
        we have that
        \begin{equation}\label{09}
x_2=-4x_1^2+\frac{1}{3}(b_1+b_3),
        \end{equation}
    \begin{equation}\label{0922}
    x_3=x_1^2+\frac{1}{3}(2b_1-b_3).
        \end{equation}
    From  \eqref{072} and \eqref{09}, $x_4$ can be expressed in terms of $x_1$ as follows:
    \begin{equation}\label{10}
        x_4=4x_1^3-\frac{1}{3}(b_1+b_3)x_1+\frac{1}{4}(b_4-b_5).
    \end{equation}
    Inserting \eqref{07}  and \eqref{0922} into \eqref{02} yields
    \begin{equation}\label{11}
          \left(7\alpha^2 -  \frac{4}{3}(2b_1-b_3)\right) x_1=b_2+ \frac{1}{2}(b_5-b_4) .
    \end{equation}
If
\[
 \alpha^2\neq 4b_3\quad \mbox{or}\quad 7\alpha^2\neq \frac{4}{3}(2b_1-b_3),
\]
then $x_1$ can be solved in terms of $b_1,b_2,b_3,b_4,b_5$. Taking into account \eqref{09}--\eqref{10}, $x_2,x_3,x_4$ are also  determined by $b_1,b_2,b_3,b_4,b_5$.
Therefore, it remains to consider the case
    \begin{equation}\label{12}
      \alpha^2=4b_3,
    \end{equation}
      \begin{equation}\label{12t}
    7\alpha^2=\frac{4}{3}(2b_1-b_3).
    \end{equation}
     From \eqref{12} and \eqref{12t}, we have  that
        \begin{equation}\label{13}
     b_1=11b_3.
    \end{equation}
 From \eqref{08} and \eqref{12}, we have that
        \begin{equation}\label{131}
 b_4=3b_5.
    \end{equation}
    From \eqref{11}, \eqref{12t} and \eqref{131}, we have that
        \begin{equation}\label{132}
 b_2=b_5.
    \end{equation}
In view of \eqref{09}--\eqref{10} and \eqref{13}--\eqref{132}, $x_2,x_3,x_4$ can  be expressed in terms of $b_3,b_5$ and $x_1$:

        \begin{equation}\label{133}
x_2=-4x_1^2+4b_3,
    \end{equation}
            \begin{equation}\label{134}
     x_3=x_1^2+7b_3.
    \end{equation}
              \begin{equation}\label{135}
  x_4=4x_1^3-4b_3x_1+\frac{1}{2} b_5.
    \end{equation}
  Inserting \eqref{12} and \eqref{133}--\eqref{135} into \eqref{0i6}, we obtain a quadratic equation for $x_1$,
\begin{equation}\label{qdeq}
2048 b_3x_1^{2}-72 b_5x_1=1904 b_3^{2}+b_6.
\end{equation}
Since $b_{3}>0$ by \eqref{12} and the fact that $\alpha\neq 0$, it follows that \eqref{qdeq}   has at most two solutions. Moreover, by \eqref{133}--\eqref{135}, it is easy to see that
 if  $(\tilde x_1,\tilde x_2,\tilde x_3,\tilde x_4)$ and $(\hat x_1,\hat x_2,\hat x_3,\hat x_4)$ are two different solutions, then $\tilde x_1\neq \hat x_1$.
This completes the proof.
\end{proof}

\section{Generalized RH waves}\label{appc}

In this appendix, we deduce a class of exact solutions of  \eqref{ave} that are slightly more general than the RH waves \eqref{drhw}.  We begin with the following lemma.

For $\mathbf p\in\mathbb S^2$ and $\mu\in\mathbb R,$ denote by $\mathsf R^{\mathbf p}_{\mu}$ the 3D rotation matrix  with  rotation axis $\mathbf p$ and rotation angle $\mu$.
\begin{lemma}\label{lmc0}
Given $\mathsf R\in\mathbf S\mathbf O(3)$ and $\mathbf p\in\mathbb S^2$, it holds that
\[\mathsf R\mathsf R^{\mathbf p}_\mu =\mathsf R^{\mathsf R\mathbf p}_\mu\mathsf R,\quad\forall\,\mu\in\mathbb R.\]
\end{lemma}
\begin{proof}
Fix $\mu\in\mathbb R$.
By Rodrigues' rotation formula, for any   $\mathbf x\in\mathbb S^2$,
\[\mathsf R^{\mathbf p}_\mu\mathbf x=(\cos\mu)\mathbf x+\sin\mu(\mathbf p\times \mathbf x)+(1-\cos\mu)(\mathbf p\cdot\mathbf x)\mathbf p.\]
Hence   for any  $\mathbf x\in\mathbb S^2$,
\begin{align*}
   \mathsf R\mathsf R^{\mathbf p}_\mu \mathbf x&= \cos\mu (\mathsf R \mathbf x)+\sin\mu\left(\mathsf R\mathbf p\times  \mathsf R \mathbf x \right)+(1-\cos\mu)\left(\mathbf p\cdot  \mathbf x \right)\mathsf R\mathbf p \\
    &=\cos\mu (\mathsf R \mathbf x)+\sin\mu\left(\mathsf R\mathbf p\times  \mathsf R \mathbf x \right)+(1-\cos\mu)\left(\mathsf R\mathbf p\cdot  \mathsf R\mathbf x \right)\mathsf R\mathbf p\\
    &=\mathsf R^{\mathsf R\mathbf p}_\mu\mathsf R\mathbf x,
\end{align*}
where we used $\mathsf R(\mathbf p\times \mathbf x)=\mathsf R \mathbf p\times \mathsf R\mathbf x$ and $\mathsf R\mathbf p\cdot  \mathsf R\mathbf x=\mathbf p\cdot\mathbf x.$
\end{proof}

Based on Lemma~\ref{lmc0} and the rotational invariance of the vorticity equation $(V_{0})$ under $\mathbf S\mathbf O(3)$, we can prove the following lemma.
\begin{lemma}\label{lmc1}
Given $\mathbf p\in\mathbb S^2$,
\begin{equation}\label{rr}
 \alpha \mathbf p\cdot \mathbf x+Y(\mathsf R^{\mathbf p}_{-ct}\mathbf x),\quad Y\in \mathbb E_{j},\quad c:=\alpha\left(  \frac{1}{2}-\frac{1}{j(j+1)}\right)
\end{equation}
 is a solution of $(V_0).$
\end{lemma}

\begin{proof}
Take $\mathsf R\in\mathbf S\mathbf O(3)$ such that $\mathsf R\mathbf p=\mathbf e_3$. Fix $Y\in\mathbb E_j$. Then $ Y\circ \mathsf R^{-1}\in\mathbb E_j.$
By \eqref{drhw},
 \[\alpha \mathbf e_{3}\cdot  \mathbf x+  Y(\mathsf R^{-1}\mathsf R^{\mathbf e_{3}}_{-ct}\mathbf x),\quad c=\alpha\left(  \frac{1}{2}-\frac{1}{j(j+1)}\right)
 \]
 solves $(V_0)$.
 Taking into account the rotational invariance of  $(V_{0})$ under $\mathbf S\mathbf O(3)$ (cf. \cite[Section 2.3]{CG}),
 \[\alpha \mathbf e_{3}\cdot (\mathsf R\mathbf x)+  Y(\mathsf R^{-1}\mathsf R^{\mathbf e_{3}}_{-ct}\mathsf R\mathbf x)=\alpha \mathbf p\cdot  \mathbf x+  Y(\mathsf R^{-1}\mathsf R^{\mathbf e_{3}}_{-ct}\mathsf R\mathbf x),\quad c=\alpha\left(  \frac{1}{2}-\frac{1}{j(j+1)}\right)\]
 also  solves $(V_0)$.
By Lemma~\ref{lmc0},
 \[\mathsf R^{\mathbf e_{3}}_{\mu}\mathsf R=\mathsf R\mathsf R^{\mathbf p}_{\mu},\quad\forall\,\mu\in\mathbb R.\]
Hence
\[\alpha \mathbf p\cdot  \mathbf x+  Y(\mathsf R^{-1}\mathsf R^{\mathbf e_{3}}_{-ct}\mathsf R\mathbf x)=\alpha \mathbf p\cdot  \mathbf x+ Y( \mathsf R^{\mathbf p}_{-ct}\mathbf x),\quad c=\alpha\left(  \frac{1}{2}-\frac{1}{j(j+1)}\right)\]
solves $(V_0)$.
\end{proof}

The aim of this appendix is to prove the following proposition.

\begin{proposition}\label{lem999}
Let $\mathbf p\in\mathbb S^{2}$, $\alpha\in\mathbb R$, and $Y\in\mathbb E_{j}$ with $j\in\mathbb Z_+$.
Then
\[ \alpha\mathbf p\cdot (\mathsf R^{\mathbf e_3}_{\omega t}\mathbf x)+Y(\mathsf R^{\mathbf p}_{-ct}\mathsf R^{\mathbf e_{3}}_{\omega t}\mathbf x), \quad c:=\alpha\left(\frac{1}{2}-\frac{1}{j(j+1)}\right)\]
is a solution of the vorticity equation \eqref{ave}.

\end{proposition}

\begin{proof}
  Recall \eqref{iffsl}:
\begin{equation}\label{gg3e2}
\mbox{$\zeta(\mathbf x,t)$ solves $(V_{0})$ if and only if $\zeta(\mathsf R^{\mathbf e_{3}}_{\omega t}\mathbf x,t)$ solves \eqref{ave}}.
\end{equation}
The desired result follows  from Lemma~\ref{lmc1} and \eqref{gg3e2} immediately.
\end{proof}

\bigskip

 \noindent{\bf Acknowledgements:}
 The authors would like to thank Dr. Bao Wang from Ningbo University for helpful discussions on the proof of Proposition~\ref{yc088}. No artificial intelligence (AI) tools were used in conducting this research.
 D. Cao was supported by National Key R\&D Program (Grant 2023YFA1010001) and NNSF of China (Grant
12371212).  G. Wang was supported by NNSF of China (Grant 12471101) and  Fundamental Research Funds for the Central Universities (Grant DUT26YC106). B. Zuo was supported by NNSF of China (Grant 12101154).

\bigskip
\noindent{\bf  Data Availability} Data sharing is not applicable to this article as no datasets were generated or analysed during the current study.

\bigskip
\noindent{\bf Declarations}

\bigskip
\noindent{\bf Conflict of interest}  The authors declare that they have no conflicts of interest related to this work.

\phantom{s}
 \thispagestyle{empty}


\begin{thebibliography}{99}

\bibitem{Abe}
Abe, K., Choi K.: Stability of Lamb dipoles. \textit{Arch. Ration. Mech. Anal.} {\bf 244}, 877--917, 2022
\bibitem{A1}
Arnold V.I.:  Conditions for nonlinear stability of plane curvilinear flow of an ideal fluid.
\textit{Sov. Math. Dokl.}  {\bf 6}, 773--777, 1965


\bibitem{A2}
Arnold V.I.:  On an a priori estimate in the theory of hydrodynamical stability. \textit{Amer. Math. Soc. Transl.}  {\bf 79}, 267--269, 1969

\bibitem{A3}
Arnold V.I., Khesin B.A.: \textit{Topological Methods in Hydrodynamics}.   Springer, Cham,  2021

\bibitem{Baines}
Baines P.G.: The stability of planetary waves on a sphere. \textit{J. Fluid Mech.}  {\bf 73}, 193--213,  1976

\bibitem{BG}
Bardos C., Guo Y.,  Strauss W.:  Stable and unstable ideal plane flows. \textit{Chinese Ann. Math. Ser. B} {\bf 23}, 149--164, 2002


\bibitem{BR}
Batt J., Rein G.:
A rigorous stability result for the Vlasov-Poisson system in three dimensions. \textit{
Ann. Mat. Pura Appl.} {\bf164}, 133--154, 1993


\bibitem{BKM}
Beale J.,   Kato T., Majda  A.:   Remarks on the breakdown of smooth solutions for
the 3-D Euler equations. \textit{Comm. Math. Phys.} {\bf94}, 61--66, 1984

\bibitem{Benard}
B\'enard P.:  Stability of Rossby-Haurwitz waves. \textit{Quart. J. R. Met. Soc.} {\bf146}, 613--628,  2020


\bibitem{Bre}
Brezis  H.:  \textit{Functional Analysis, Sobolev Spaces and Partial Differential Equations}.  Springer, New York, 2011

\bibitem{BMA}
Burton G.R.: Rearrangements of functions, maximization of convex functionals, and vortex rings.  \textit{Math. Ann.} {\bf276}, 225--253,  1987
\bibitem{BHP}
Burton G.R.: Variational problems on classes of rearrangements and multiple configurations for steady vortices. \textit{Ann. Inst. H. Poincar\'e Anal. Non Lin\'eaire} {\bf6}, 295--319,  1989

\bibitem{BAR}
Burton G.R.: Global nonlinear stability for steady ideal fluid flow in bounded planar domains. \textit{Arch. Ration. Mech. Anal.} {\bf 176}, 149--163, 2005



\bibitem{Bjde}
Burton G.R.:  Compactness and stability for planar vortex-pairs with prescribed impulse. \textit{J. Differential Equations} {\bf270}, 547--572, 2021



\bibitem{Bcmp}
Burton G.R., Nussenzveig Lopes H.J.,   Lopes Filho M.C.: Nonlinear stability for steady vortex pairs. \textit{Comm. Math. Phys.} {\bf324}, 445--463, 2013




\bibitem{CWCV}
 Cao D., Wang G.: Steady vortex patches with opposite rotation directions in a planar ideal fluid. \textit{Calc. Var. Partial Differential Equations} {\bf58}, Paper No. 75, 2019


\bibitem{CWN}
 Cao D., Wang G.:  Nonlinear stability of planar vortex patches in an ideal fluid. \textit{J. Math. Fluid Mech.} {\bf23}, Paper No. 58, 2021

\bibitem{Cap}
Caprino S.,  Marchioro C.:  On nonlinear stability of stationary Euler flows on a rotating sphere. \textit{J. Math. Anal. Appl.} {\bf129},  24--36, 1988

 \bibitem{CD}
 Choi K.,  Lim D.: Stability of radially symmetric, monotone vorticities of 2D Euler equations. \textit{Calc. Var. Partial Differential Equations} {\bf61}, 120,  2022


\bibitem{CG}
Constantin A., Germain P.:  Stratospheric planetary flows from the perspective of the Euler equation on a rotating sphere. \textit{Arch. Ration. Mech. Anal.} {\bf245}, 587--644, 2022

\bibitem{Crai}
Craig R.A.: A solution of the nonlinear vorticity equation for atmospheric motion. \textit{J. Meteor.} {\bf2}, 175--178,   1945






 \bibitem{Dow}
Dowling T.E.: Dynamics of Jovian atmospheres. \textit{Annu. Rev. Fluid Mech.} {\bf27}, 293--334, 1995

\bibitem{GT}
Gilbarg  D.,  Trudinger N.:  \textit{Elliptic Partial Differential Equations of Second Order}. Springer-Verlag, Berlin, 2001

\bibitem{Haur}
Haurwitz B.:  The motion of atmospheric disturbances on the spherical earth. \textit{J. Mar. Res.} {\bf3}, 254--267,  1940

\bibitem{HYHH}
He M., Yamazaki Y.,   Hoffman P.,  Hall C.M.,  Tsutsumi  M.,  Li G., Chau  J.L.:
Zonal wave number diagnosis of Rossby wave-like oscillations using paired ground-based radars. \textit{ J. Geophys. Atmos.}   {\bf125},  e2019JD031599, 2020


 \bibitem{Hebey}
Hebey  E.: \textit{Sobolev Spaces on Riemannian Manifolds}. Springer-Verlag, Berlin, 1996

 \bibitem{Hos}
Hoskins  B.J.: Stability of the Rossby-Haurwitz wave. \textit{Quart. J. R. Met. Soc.} {\bf99}, 723--745, 1973

 \bibitem{Hos2}
Hoskins B.J.,  Hollingsworth A.:  On the simplest example of the barotropic instability of Rossby wave motion. \textit{J. Atmos. Sci.}  {\bf30}, 150--153,  1973

\bibitem{KBH}
Krishnamurti T.N.,  Bedi H.S.,   Hardiker  V.M.:  \textit{An introduction to global spectral modeling}. Springer Science \& Business Media, 2006


\bibitem{Lorenz}
Lorenz E.N.:  Barotropic instability of Rossby wave motion. \textit{J. Atmos. Sci.} {\bf29}, 258--264,  1972

\bibitem{MB}
Majda A., Bertozzi  A.: \textit{Vorticity and Incompressible Flow}.  Cambridge University Press, Cambridge, 2002



\bibitem{MP}
Marchioro C., Pulvirenti M.: \textit{Mathematical Theory of Incompressible Nonviscous Fluids}.  Springer, New York, 1994




\bibitem{Muller}
M\"uller C.:  \textit{Spherical Harmonics}. Springer, Berlin-New York, 1966


\bibitem{Nualart}
Nualart M.: On zonal steady solutions to the 2D Euler equations on the rotating unit sphere. \textit{Nonlinearity} {\bf 36},  4981--5006, 2023.


\bibitem{PPR}
Palais B., Palais R.,   Rodi S.:
A disorienting look at Euler's theorem on the axis of a rotation. \textit{Amer. Math. Monthly} {\bf116}, 892--909, 2009

\bibitem{Ross}
Rossby C.G.:  Relations between variations in the intensity of the zonal circulation of the atmosphere and the displacements of the semipermanent centers of action. \textit{J. Mar. Res.} {\bf2}, 38--55,   1939

\bibitem{Skiba0}
Skiba Y.N.:  On the normal mode instability of harmonic waves on a sphere. \textit{Geophys. Astrophys. Fluid Dyn.} {\bf92}, 115--127,  2000


\bibitem{Skiba1}
Skiba Y.N.:  On the spectral problem in the linear stability study of flows on a sphere.  \textit{J. Math. Anal. Appl.}  {\bf270}, 165--180,  2002


\bibitem{Skiba2}
Skiba Y.N.:  On Liapunov and exponential stability of Rossby-Haurwitz waves in invariant sets of perturbations. \textit{J. Math. Fluid Mech.} {\bf20},  1137--1154, 2018



\bibitem{T}
 Taylor M.:  Euler equation on a rotating surface. \textit{J. Funct. Anal.}  {\bf270},   3884--3945, 2016


\bibitem{Verk}
Verkley W.T.M.: The construction of barotropic modons on a sphere. \textit{J. Atmos. Sci.} {\bf41}, 2492--2504, 1984

 \bibitem{WGu1}
Wang G.: Nonlinear stability of planar steady Euler flows associated with semistable solutions of elliptic problems. \textit{Trans. Amer. Math. Soc.} {\bf375},   5071--5095, 2022

 \bibitem{WGu2}
Wang G.: Stability of 2D steady Euler flows related to least energy solutions of the Lane-Emden equation. \textit{J. Differential Equations} {\bf342}, 596--621,  2023


  \bibitem{Wang}
Wang G.:  Stability of two-dimensional steady Euler flows with concentrated vorticity.   \textit{Math. Ann.} {\bf389}, 121--168, 2024


\bibitem{WZ}
Wang   G.,  Zuo B.:
An extension of Arnold's second stability theorem in a multiply-connected domain. \textit{Calc. Var. Partial Differential Equations} {\bf64},  Paper No. 302, 2025

\bibitem{WZ0}
 Wang G.,  Zuo B.: Stability of sinusoidal Euler flows on a flat two-torus. \textit{Calc. Var. Partial Differential Equations} {\bf62}, Paper No. 207,  2023

\bibitem{WS}
Wirosoetisno D.,   Shepherd T.G.:  Nonlinear stability of Euler flows in two-dimensional periodic domains. \textit{ Geophys. Astrophys. Fluid Dynam.} {\bf90}, 229--246, 1999

\bibitem{WG1}
Wolansky G., Ghil M.: An extension of Arnold's second stability theorem for the Euler equations.  \textit{Phys. D} {\bf94}, 161--167, 1996


\bibitem{WG2}
Wolansky G., Ghil M.:
Nonlinear stability for saddle solutions of ideal flows and symmetry breaking. \textit{Comm. Math. Phys.} {\bf193}, 713--736,  1998




\end{thebibliography}
\end{document}